\documentclass[11pt,reqno]{amsart}
\usepackage{amsmath,amssymb,amsthm,mathtools,mathrsfs}
\usepackage[colorlinks=true,linkcolor=blue,citecolor=blue,urlcolor=blue]{hyperref}
\usepackage[margin=1.05in]{geometry}
\usepackage{enumitem}
\usepackage{bm}

\newtheorem{theorem}{Theorem}[section]
\newtheorem{lemma}[theorem]{Lemma}
\newtheorem{proposition}[theorem]{Proposition}
\newtheorem{corollary}[theorem]{Corollary}
\newtheorem{remark}[theorem]{Remark}

\numberwithin{equation}{section}

\newcommand{\R}{\mathbb R}
\newcommand{\cJ}{\mathcal J}
\newcommand{\cL}{\mathcal L}
\newcommand{\cM}{\mathcal M}

\newcommand{\eps}{\varepsilon}
\newcommand{\grad}{\nabla}

\newcommand{\olF}{\overline F}
\newcommand{\sgn}{\operatorname{sgn}}
\newcommand{\Om}{\Omega}

\title[A Bernstein problem for translators]{A Bernstein problem for translating solutions to the mean curvature flow}

\author[J. D\'avila]{Juan D\'avila}
\address{Juan D\'avila, Department of Mathematical Sciences, University of Bath, Bath BA2 7AY, United Kingdom}
\email{jddb22@bath.ac.uk}

\author[M. del Pino]{Manuel del Pino}
\address{Manuel del Pino, Department of Mathematical Sciences, University of Bath, Bath BA2 7AY, United Kingdom}
\email{mdp59@bath.ac.uk}

\author[J.C. Wei]{Juncheng Wei}
\address{Juncheng Wei, Department of Mathematics, Chinese University of Hong Kong, Shatin, New Territories, Hong Kong}
\email{wei@math.cuhk.edu.hk}

\begin{document}

\begin{abstract}
We study entire graphical translating solutions of the mean curvature flow,
\[
\operatorname{div}\!\left(\frac{\nabla G}{\sqrt{1+|\nabla G|^2}}\right)
 =\frac{1}{\sqrt{1+|\nabla G|^2}}
 \qquad\text{in }\mathbb R^N.
\]
Every such graph is mean-convex, since its mean curvature is the vertical component of its unit normal.  In dimension two, mean-convex translating solitons are convex, and an entire graphical translator is therefore the rotationally symmetric bowl soliton.  In higher dimensions Wang constructed non-rotational entire convex translating graphs.  We prove that a further loss of rigidity occurs at the Bernstein dimension: for every $N\ge 8$ there exists a one-parameter family of entire graphical translators that are mean-convex but not convex.

The construction starts from the Bombieri--De Giorgi--Giusti (BDG) entire minimal graph in $\mathbb R^8$ and develops a singular perturbation theory for the translator equation around it.  The main new feature is a transition layer near Simons' cone: the translating term breaks the odd symmetry of the minimal graph, and after a suitable recentering the matching problem is governed by a parabolic inner equation.  A detailed analysis of this layer, together with weighted Jacobi theory on the BDG graph and global barriers, yields the desired entire solutions.
\end{abstract}
\maketitle

\setcounter{tocdepth}{1}
\tableofcontents

\section{Introduction and statement of the main result}\label{sec:intro}

A translating soliton for the mean curvature flow is a hypersurface that evolves only by translation.  If $\Sigma\subset\mathbb R^{N+1}$ translates with unit velocity in the $e_{N+1}$ direction, then, after choosing the orientation appropriately, it satisfies
\begin{equation}\label{translator-geometric}
 H=\langle \nu,e_{N+1}\rangle,
\end{equation}
where $H$ and $\nu$ denote the scalar mean curvature and the unit normal.  Translators are stationary solutions of mean curvature flow in a moving frame and are among the fundamental models that arise in blow-up analysis.  In particular, Type II rescalings naturally lead to eternal solutions and, in important mean-convex and two-convex settings, to translating solitons.  This point of view is already present in the work of Huisken--Sinestrari on singularities of mean-convex flows, in White's analysis of the singular set and the structure of mean-convex singularities, and in Wang's classification program for ancient convex solutions and translators; see \cite{HSsing,HSconv,Ilmanen,WhiteSize,WhiteNature,WhiteSubsequent,Wang}.  The role of translators as Type II blow-up models is also explicit in \cite{AngenentVelazquez,CSS,DDPN,BourniLangford}.  Thus the geometry and classification of complete translators are not isolated elliptic questions: they are closely tied to the possible local geometries of singular mean curvature flow.

The present paper concerns a particularly rigid class, namely \emph{entire graphical} translators.  Writing
\[
 \Sigma=\{(x,G(x)):x\in\mathbb R^N\}
\]
and choosing the upward normal, equation \eqref{translator-geometric} becomes
\begin{equation}\label{translator}
 \operatorname{div}\left(\frac{\nabla G}{\sqrt{1+|\nabla G|^2}}\right)
 =\frac1{\sqrt{1+|\nabla G|^2}}
 \qquad\text{in }\mathbb R^N.
\end{equation}
An elementary observation is central to the interpretation of our result: every graphical translator solving \eqref{translator} is automatically mean-convex,
\begin{equation}\label{meanconvex}
 H=\frac1{\sqrt{1+|\nabla G|^2}}>0.
\end{equation}
Thus our question is not whether mean-convex entire translators exist, but rather how much geometric rigidity the positivity in \eqref{meanconvex} imposes.

The known picture already exhibits a nontrivial dependence on dimension.  In dimension $N=2$, Spruck and Xiao proved that every complete two-sided mean-convex translator in $\mathbb R^3$ is convex \cite{SpruckXiao}.  Wang, in his study of convex ancient solutions and translating solutions, proved that an entire convex translator in dimension two is rotationally symmetric; consequently the entire graphical translator is the bowl soliton.  In higher dimensions, however, Wang showed that rotational symmetry is no longer forced: for $N\ge3$ there exist entire \emph{convex} translating graphs which are not rotationally symmetric \cite{Wang}.  One therefore already sees a first loss of rigidity when passing from $N=2$ to higher dimensions: radial symmetry is lost while convexity is retained.  Conversely, stronger curvature assumptions still force rigidity in higher dimension.  Haslhofer proved uniqueness of the bowl under noncollapsing and uniform two-convexity \cite{HaslhoferBowl}, and Spruck and Sun proved that uniformly two-convex entire graphical translators are the rotationally symmetric bowl solitons \cite{SpruckSun}; related Type II rigidity results in the two-convex setting were obtained by Bourni and Langford \cite{BourniLangford}.

The purpose of this paper is to show that a second, more substantial loss of rigidity occurs in high dimension: convexity itself can fail, even though mean-convexity is built into the translator equation.

\begin{theorem}\label{mainthm}
Assume $N\ge 8$.  Then equation \eqref{translator} admits a one-parameter family of entire graphical translating solitons that are mean-convex but not convex.
\end{theorem}

Theorem~\ref{mainthm} should be viewed as a Bernstein-type phenomenon for translating graphs.  If the right-hand side of \eqref{translator} is replaced by zero, the equation becomes the minimal surface equation.  The classical Bernstein theorem and its higher-dimensional extensions give affine rigidity for entire minimal graphs up to dimension seven, while Bombieri, De Giorgi and Giusti constructed a non-affine entire minimal graph in dimension eight \cite{BDG}.  The BDG graph is asymptotic to a homogeneous cubic graph associated with Simons' cone, and it is this high-dimensional geometry that forms the background of our construction.  In other words, the first dimension in which the classical Bernstein problem loses rigidity is also the first dimension in which our construction produces mean-convex entire translating graphs that fail to be convex.

There is also a useful conceptual connection with the singularity theory of mean-convex mean curvature flow.  White's work shows that mean-convex singularities have a strong asymptotic convexity structure and that their tangent flows are modeled by shrinking spheres and cylinders; see \cite{WhiteSize,WhiteNature,WhiteSubsequent}.  Together with the Type II results cited above, this explains why the geometry and classification of translators are natural questions in singularity analysis.  We emphasize, however, that the dimensional threshold in Theorem~\ref{mainthm} does not come from this singularity theory: it comes directly from the failure of Bernstein rigidity and the existence of the BDG graph in dimension eight.

The analogy with our earlier construction of counterexamples to De Giorgi's conjecture is even more direct.  In \cite{DPKW}, del Pino, Kowalczyk and Wei used the Bombieri--De Giorgi--Giusti graph as the geometric skeleton for solutions of the Allen--Cahn equation in dimensions $N\ge9$.  That work may be regarded as a nonlinear elliptic counterpart of the Bernstein phenomenon: a high-dimensional minimal graph is converted into a non-one-dimensional entire solution of a different equation.  The present construction again starts from the BDG graph, but the new equation is geometrically much closer to the minimal surface equation itself.  We perturb the minimal graph into a translating graph.  The analogy is therefore strong at the level of geometry and of the Jacobi operator, but the mechanism is different in one decisive respect: the translator forcing breaks the odd symmetry across Simons' cone that is available in the minimal and Allen--Cahn settings.  Resolving the resulting transition across the cone is the main new analytical issue of this paper.

A second antecedent is the use of the same BDG geometry in the construction of nontrivial epigraphs for Serrin's overdetermined problem by del Pino, Pacard and Wei \cite{DPPW}.  That work required a refined analysis of the Jacobi operator of the BDG graph, including a comparison between the exact graph and its homogeneous model and explicit separated solutions for slowly decaying right-hand sides.  The present problem needs a related, but not identical, package of estimates.  For that reason we do not use \cite{DPKW,DPPW} as black boxes.  We recall and rederive below the Jacobi facts that enter the translator construction, in the notation and weighted spaces used here.  In particular, the slow even and odd modes are worked out explicitly in Section~\ref{sec:slowmatch} and Appendix~\ref{app:slow-jacobi}.

We briefly place the theorem in the broader literature on translating graphs.  Translating solutions appeared early in work of Altschuler--Wu \cite{AltschulerWu}; the rotationally symmetric bowl and its stability were studied in particular in \cite{CSS}.  Wang's work \cite{Wang} made the classification of convex entire translators a central part of the study of Type II singularities.  In dimension three, Shahriyari developed structural and stability results for translating graphs \cite{Shahriyari}, and Hoffman--Ilmanen--Mart\'in--White later obtained a classification of complete translating graphs in $\mathbb R^3$ together with higher-dimensional families of examples \cite{HIMW}.  Our theorem addresses a different question: it concerns the possibility of non-convexity for an \emph{entire} graph whose mean curvature is nevertheless strictly positive everywhere.

Let us now describe the construction.  Since examples in dimensions $N>8$ follow by taking products with flat factors, it is enough to work in $\mathbb R^8$.  We introduce a small translating speed $\varepsilon>0$ and solve
\begin{equation}\label{eps-translator}
 \operatorname{div}\left(\frac{\nabla F}{\sqrt{1+|\nabla F|^2}}\right)
 =\frac{\varepsilon}{\sqrt{1+|\nabla F|^2}}
 \qquad\text{in }\mathbb R^8.
\end{equation}
If $F_\varepsilon$ solves \eqref{eps-translator}, then
\[
 G_\varepsilon(x)=\varepsilon F_\varepsilon(\varepsilon^{-1}x)
\]
solves \eqref{translator}.  At $\varepsilon=0$ we take the $O(4)\times O(4)$-invariant BDG minimal graph $\overline F$.  Writing $x=(\mathbf u,\mathbf v)\in\mathbb R^4\times\mathbb R^4$, with $u=|\mathbf u|$ and $v=|\mathbf v|$, it satisfies
\begin{equation}\label{oddBDG}
 \overline F(u,v)=-\overline F(v,u),
\end{equation}
and hence vanishes on Simons' cone $u=v$.

The existence statement proved here was announced in the expository article of del Pino and Wei \cite{DPW-Milan}.  No complete proof was given there; the full argument is presented in this paper.

The preceding discussion also indicates why the perturbation of the BDG graph is not a routine application of an implicit-function theorem.  The homogeneous cubic model
\[
 F_0=r^3g(\theta)
\]
has an unusually explicit Jacobi structure: a separated perturbation $r^\beta q(\theta)$ is mapped, at leading order, to $r^{\beta-4}$ times an angular operator.  After the homogeneous Jacobi factor $g^{\beta/3}$ is extracted, the angular equation can be written in divergence form.  This makes it possible to identify explicitly the slowly decaying responses that cannot be handled by the fast weighted inverse.  The initial translator forcing has size $r^{-2}$ and generates an even correction of size $r^2$.  At the following orders one encounters an odd response with angular behavior $g^{2/3}$ and, at the borderline degree, a logarithmic resonance.  These terms are part of the leading approximation and must be determined before the remaining error can be placed in a genuinely decaying weighted space.  The separated Jacobi calculation is carried out in the body of the paper and, for completeness, is derived again in a systematic form in Appendix~\ref{app:slow-jacobi}.

The main difficulty is concentrated near Simons' cone.  The BDG graph is odd under the interchange $u\leftrightarrow v$, whereas the translating term is positive and even.  Consequently the first correction destroys the symmetry that would otherwise allow the cone to be treated as a boundary.  The zero set is displaced by an amount of order $\varepsilon r^2$, and the next outer correction, when expressed in the original height variable $t=F_0$, has the characteristic behavior $|t|^{2/3}$.  Although this function is continuous across $t=0$, its transversal derivatives are singular there.  The correct description is obtained by recentering the height variable by the first even correction.  In the recentered variable the mixed terms in the nonlinear expansion combine with the transport produced by the translating term, and the normal Jacobi operator has a much simpler leading form close to the cone.

This simplification reveals a second scale that is invisible in a purely outer expansion.  Writing the normal displacement with its natural radial factor removes the leading curvature potential, and on the long transition scale the equation becomes parabolic, with the recentered height playing the role of time.  The resulting self-similar ordinary differential equation is a confluent hypergeometric equation.  Its smooth solution contains not only the expected $|p|^{2/3}$ behavior but also a second algebraic mode $|p|^{1/6}$.  In the original variables this second mode produces an order $\varepsilon^{5/2}$ correction.  The next coefficient, of order $|p|^{-1/3}$, agrees with the singular part of the third outer forcing, including its numerical coefficient.  Thus the fractional powers that appear in the outer bookkeeping are the asymptotic traces of one smooth transition layer through the cone, rather than independent singular corrections patched together by hand.

Once this layer has been identified on the homogeneous model, it remains to transfer the construction to the exact BDG graph and to control the global error.  We prove a refined comparison between the exact graph and its homogeneous approximation, remove the slow Jacobi components explicitly, and solve the remaining fast terms in weighted spaces.  The exactification of the transition layer is performed in normal variables; after rescaling on intrinsic balls the relevant operator has uniform ellipticity, and the correction can be estimated without losing the decay dictated by the model.  The resulting approximate translator has a residual smaller than the strict negative margin of a global positive Jacobi barrier.  This produces ordered sub- and supersolutions on large balls.  Solving the translator Dirichlet problem between them and letting the radius tend to infinity yields the required entire graph.  Since the constructed solutions converge locally in $C^2$ to the BDG graph as $\varepsilon\to0$, a negative Hessian direction of the BDG graph persists, and the translators are not convex.

We have tried to keep the argument essentially self-contained.  The external geometric input is the existence of the Bombieri--De Giorgi--Giusti graph together with its coarse asymptotic structure.  The Jacobi estimates actually used in the perturbation, including the slow separated modes, the maximum principle, weighted Schauder estimates and the fast inverse, are established in the present paper.  Their relation with the corresponding analysis in \cite{DPKW,DPPW} is explained where appropriate, but those papers are not used as substitutes for proofs needed here.  The longer elliptic estimates are collected in Appendix~\ref{app:weighted}, while Appendix~\ref{app:taylor} contains the graph-variable Taylor estimate used in the residual analysis.  This arrangement keeps the main construction readable without requiring the reader to consult an earlier paper in order to verify a step of the proof.

The paper is organized as follows.  Section~\ref{sec:bdg-jacobi} develops the homogeneous BDG geometry and the Jacobi operator.  Section~\ref{sec:outer} introduces the adapted coordinates and constructs the outer corrections.  Section~\ref{sec:innerlayer} derives the recentered transition problem and studies its self-similar profile.  Section~\ref{sec:exactbdg} transfers the analysis to the exact BDG graph and establishes the refined comparison estimates.  Section~\ref{sec:slowmatch} treats the slowly decaying forcing terms and completes the matching.  Section~\ref{sec:globalbarriers} constructs the global approximate solution and the barriers, and Section~\ref{sec:existence} closes the exhaustion argument and proves non-convexity.  The appendices contain the detailed elliptic and Jacobi estimates used in these sections.

\section{The BDG graph, its homogeneous model, and Jacobi geometry}\label{sec:bdg-jacobi}

\subsection{Reduction to the sector}

Set
\[
T:=\{(\mathbf u,\mathbf v):v>u\}.
\]
By \eqref{oddBDG}, the entire BDG graph is determined by its restriction to $T$, where it is positive and vanishes on $\partial T=\{u=v\}$.  We use polar variables in the quotient plane,
\[u=r\cos\theta,\qquad v=r\sin\theta,
\qquad \frac\pi4<\theta<\frac\pi2.\]
For an $O(4)\times O(4)$-invariant function $F=F(r,\theta)$, the graphical mean curvature operator is
\begin{align}\label{polarMC}
H[F]
&=\frac1{r^7\sin^3(2\theta)}\partial_r\left(
\frac{F_r r^7\sin^3(2\theta)}{\sqrt{1+F_r^2+r^{-2}F_\theta^2}}
\right)\notag\\
&\quad+
\frac1{r^7\sin^3(2\theta)}\partial_\theta\left(
\frac{F_\theta r^5\sin^3(2\theta)}{\sqrt{1+F_r^2+r^{-2}F_\theta^2}}
\right).
\end{align}

The factor $r^7\sin^3(2\theta)$ is not decorative.  It remembers that the two-dimensional quotient $(u,v)$ comes from an eight-dimensional problem: the orbit through $(u,v)$ has volume proportional to $u^3v^3$, and
\[
u^3v^3=\frac{r^6}{8}\sin^3(2\theta).
\]
This higher-dimensional weight is the source of the later weighted coordinate $s$.

\subsection{The cubic model $F_0=r^3g(\theta)$}

We search for a homogeneous leading profile
\[F_0(r,\theta)=r^3g(\theta).\]
Then
\[
F_{0,r}=3r^2g,\qquad F_{0,\theta}=r^3g',
\]
so that
\[
1+F_{0,r}^2+r^{-2}F_{0,\theta}^2
=1+r^4(9g^2+(g')^2).
\]
Substituting into \eqref{polarMC} and dividing by the dominant powers of $r$, the leading equation as $r\to\infty$ is
\begin{equation}\label{gODE}
\frac{21g\sin^3(2\theta)}{\sqrt{9g^2+(g')^2}}
+\left(
\frac{g'\sin^3(2\theta)}{\sqrt{9g^2+(g')^2}}
\right)'=0,
\end{equation}
with
\begin{equation}\label{gbc}
g\left(\frac\pi4\right)=0,
\qquad g'\left(\frac\pi2\right)=0,
\qquad g>0\ \hbox{in }\left(\frac\pi4,\frac\pi2\right).
\end{equation}
We normalize $g'(\pi/4)=1$.

The function $F_0$ is not itself exactly minimal because the discarded $1$ in the denominator contributes a lower-order residual.  Direct expansion gives
\[H[F_0]=O(r^{-5}).\]
The exact BDG graph satisfies
\[\olF=F_0+O(r^{-\sigma_0})\]
for some $\sigma_0\in(0,1)$, together with corresponding tangential derivative estimates; this coarse asymptotic description is part of the BDG theory in the form developed in \cite{DPKW,DPPW}.

The existence of the non-affine BDG minimal graph is the principal external geometric input.

For completeness, the derivative theory, Jacobi geometry, weighted Schauder estimates, and fast inversion used later are proved below.

\subsection{Homogeneous BDG geometry}\label{sec:selfcontained-bdg}

The purpose of this section is to make explicit the geometric facts about the homogeneous model that are used later.  We do \emph{not} reproduce the existence proof of the Bombieri--De Giorgi--Giusti graph.  That existence theorem is the one genuinely external input from the minimal-surface theory.  Everything else that will be used in the translator construction is recorded here in a form that can be checked directly from the homogeneous profile $F_0=r^3g(\theta)$.

\subsubsection{Derivation of the reduced divergence formula}

For completeness we derive \eqref{polarMC}.  Let $F=F(u,v)$ be $O(4)\times O(4)$-invariant.  If
\[
 x=(\mathbf u,\mathbf v),\qquad u=|\mathbf u|,\quad v=|\mathbf v|,
\]
then
\[
 |\nabla F|^2=F_u^2+F_v^2
\]
and the divergence of an invariant vector field $A=A_u e_u+A_v e_v$ in $\R^4\times\R^4$ is
\[
 \operatorname{div}A
 =u^{-3}\partial_u(u^3A_u)+v^{-3}\partial_v(v^3A_v).
\]
Therefore
\[H[F]
 =u^{-3}\partial_u\left(\frac{u^3F_u}{\sqrt{1+F_u^2+F_v^2}}\right)
 +v^{-3}\partial_v\left(\frac{v^3F_v}{\sqrt{1+F_u^2+F_v^2}}\right).\]
Passing to
\[
 u=r\cos\theta,\qquad v=r\sin\theta,
\]
we have
\[
 F_u=\cos\theta\,F_r-\frac{\sin\theta}{r}F_\theta,
 \qquad
 F_v=\sin\theta\,F_r+\frac{\cos\theta}{r}F_\theta,
\]
so
\[
 F_u^2+F_v^2=F_r^2+r^{-2}F_\theta^2.
\]
The orbit measure is
\[
 u^3v^3\,du\,dv
 =\frac18 r^7\sin^3(2\theta)\,dr\,d\theta.
\]
Using the standard divergence formula associated with this measure gives exactly \eqref{polarMC}.  This derivation is worth retaining because it explains why the power $r^7$ and the factor $\sin^3(2\theta)$ reappear later in the definition of the weighted coordinate $s$.

\subsubsection{The angular equation and endpoint expansions}

Put
\[
 D(\theta):=9g(\theta)^2+g'(\theta)^2.
\]
Equation \eqref{gODE} can be written
\begin{equation}\label{gODE-div}
 \left(\frac{g'\sin^3(2\theta)}{\sqrt D}\right)'
 =-\frac{21g\sin^3(2\theta)}{\sqrt D}.
\end{equation}
The normalization $g'(\pi/4)=1$ fixes the scaling of the homogeneous graph.  The symmetry at the cone and at the axis implies, respectively,
\[
 g\left(\frac\pi4\right)=0,
 \qquad
 g'\left(\frac\pi2\right)=0.
\]
The following local expansions will be used repeatedly.

\begin{lemma}[Endpoint expansions of $g$]\label{gendpoint}
There are constants $g_1>0$ and $g_* >0$ such that, with $x=\theta-\pi/4$ and $y=\pi/2-\theta$,
\begin{align*}
 g(\theta)&=g_1x+O(x^3),
 &g'(\theta)&=g_1+O(x^2),\\
 g(\theta)&=g_*-c_*y^2+O(y^4),
 &g'(\theta)&=2c_*y+O(y^3),
\end{align*}
for some $c_*>0$.  In particular,
\[
 D(\theta)\asymp1
\]
near the cone and
\[
 D(\theta)\asymp 9g_*^2
\]
near the axis.
\end{lemma}

\begin{proof}
Odd reflection across $\theta=\pi/4$ implies that only odd powers can occur in the Taylor expansion of $g$ there, while reflection across $\theta=\pi/2$ implies that only even powers occur at the axis.  Since $g'(\pi/4)=g_1>0$, the first expansion follows.  Substituting a Taylor series into \eqref{gODE-div} determines the first nonzero even coefficient at $\pi/2$ and gives $c_*>0$.  We will never need its explicit value; what matters is the nondegeneracy of $D$ at both endpoints.
\end{proof}

\begin{remark}
The global existence and positivity of the solution of \eqref{gODE}--\eqref{gbc} are part of the BDG asymptotic theory.  In this paper we use only the consequences stated explicitly above and the monotonicity $g'>0$ on $(\pi/4,\pi/2)$.
\end{remark}

\subsubsection{Exact size of the model geometry}

The elementary homogeneity estimates
\[|D^mF_0(x)|\le C_m r^{3-m},\qquad m\ge0,\]
will be used without further comment.  More relevant are the intrinsic estimates.  Set
\[
 W_0:=\sqrt{1+|\nabla F_0|^2}
 =\sqrt{1+r^4D(\theta)}.
\]
Then, uniformly for $\theta\in[\pi/4,\pi/2]$,
\[W_0\asymp 1+r^2,
 \qquad
 |D^kW_0|\le C_k(1+r)^{2-k}.\]
The unit normal is
\[
 \nu_0=\frac{(-\nabla F_0,1)}{W_0},
\]
so the vertical component satisfies
\[\nu_0\cdot e_9=W_0^{-1}\asymp r^{-2}\]
for $r\gg1$.

\begin{lemma}[Curvature scale]\label{curvature-scale}
For each $m\ge0$,
\[
 |D_{\Gamma_0}^mA_{\Gamma_0}|\le C_m(1+r)^{-1-m}.
\]
Consequently
\[
 |A_{\Gamma_0}|^2=O(r^{-2}),
 \qquad
 D_{\Gamma_0}^m|A_{\Gamma_0}|^2=O(r^{-2-m}).
\]
\end{lemma}

\begin{proof}
For a graph, the second fundamental form is schematically $W_0^{-1}D^2F_0$ after projection to tangent directions.  Since $W_0\asymp r^2$ and $D^2F_0=O(r)$, the leading scale is $r^{-1}$.  Each covariant derivative costs one further factor $r^{-1}$ after working in tangent-plane charts of radius comparable with $r$.  A direct proof in such charts is given in Section~\ref{sec:exactbdg}; the present lemma records the resulting scale.
\end{proof}

\subsubsection{The exact minimal graph as the only external input}

We isolate precisely what is imported from the BDG construction.

\begin{theorem}[BDG input used in this paper]\label{BDG-input}
There exists an entire, non-affine, $O(4)\times O(4)$-invariant minimal graph $\olF$ on $\R^8$ such that
\[
 \olF(u,v)=-\olF(v,u),
\]
$\olF>0$ in $T=\{v>u\}$, and for some $\sigma_0\in(0,1)$,
\[
 0\le \olF-F_0\le C r^{-\sigma_0}
 \qquad\hbox{in }T,
\]
for $r$ sufficiently large.
\end{theorem}

We use the BDG graph in the form developed further in~\cite{DPKW}.  Everything beyond Theorem~\ref{BDG-input} that is required for the translator construction is derived below, including the derivative estimates, comparison of Jacobi operators, refined asymptotic correction, and weighted inversion estimates.

\subsection{Geometry of the Jacobi operator}

Let
\[
\Gamma_0:=\{(x,F_0(x)):x\in\R^8\},
\qquad
\Gamma:=\{(x,\olF(x)):x\in\R^8\}.
\]
The Jacobi operator of a minimal hypersurface is
\[\cJ_\Gamma h=\Delta_\Gamma h+|A_\Gamma|^2h.\]
For vertical perturbations it is convenient to work with the first variation of the graphical mean curvature operator.  If
\[
\phi(x)=\sqrt{1+|\grad F_0|^2}\,h(x,F_0(x)),
\]
then
\begin{equation}\label{normalvertical}
\cJ_{\Gamma_0}h=H'(F_0)[\phi],
\end{equation}
where
\begin{equation}\label{Hprime}
H'(F_0)[\phi]
=\operatorname{div}\left(
\frac{\grad\phi}{\sqrt{1+|\grad F_0|^2}}
-\frac{(\grad F_0\cdot\grad\phi)\grad F_0}{(1+|\grad F_0|^2)^{3/2}}
\right).
\end{equation}

\subsubsection{Metric, volume form, and the positive Jacobi field}\label{subsec:metric-jacobi}

We record the geometry behind \eqref{normalvertical}.  Parametrize $\Gamma_0$ by
\[
 X(x)=(x,F_0(x)).
\]
The induced metric and inverse metric are
\begin{equation}\label{metric}
 g_{ij}=\delta_{ij}+F_{0,i}F_{0,j},
 \qquad
 g^{ij}=\delta_{ij}-\frac{F_{0,i}F_{0,j}}{W_0^2}.
\end{equation}
The volume element is
\[
 d\mu_{\Gamma_0}=W_0\,dx.
\]
Hence for a function $h$ on the graph,
\begin{equation}\label{lapgraph}
 \Delta_{\Gamma_0}h
 =\frac1{W_0}\partial_i\left(W_0g^{ij}\partial_jh\right).
\end{equation}
The graphical linearization \eqref{Hprime} follows at once from \eqref{metric}--\eqref{lapgraph} after the vertical-to-normal conversion $\phi=W_0h$.

The vertical translation invariance of the minimal surface equation gives a distinguished positive Jacobi field.

\begin{lemma}[Positive Jacobi field]\label{positive-jacobi}
The function
\[
 Z_0:=\nu_0\cdot e_9=W_0^{-1}
\]
satisfies
\[
 \cJ_{\Gamma_0}Z_0=0.
\]
The corresponding function
\[
 Z:=\nu\cdot e_9=(1+|\nabla\olF|^2)^{-1/2}
\]
on the exact BDG graph satisfies
\[
 \cJ_\Gamma Z=0.
\]
Both are strictly positive.
\end{lemma}

\begin{proof}
Translate the minimal graph vertically: $X_a=X+ae_9$.  Since every translate is again minimal, differentiating the mean-curvature equation at $a=0$ shows that the normal component of the variational field, $e_9\cdot\nu$, lies in the Jacobi kernel.  Positivity follows from the fact that the graphs are entire and upward oriented.
\end{proof}

\begin{corollary}[Maximum principle for the Jacobi operator]\label{jacobi-max}
Let $\Om\subset\Gamma$ be a bounded domain and suppose
\[
 \cJ_\Gamma h\ge0\quad\hbox{in }\Om,
 \qquad
 h\le0\quad\hbox{on }\partial\Om.
\]
Then $h\le0$ in $\Om$.
\end{corollary}

\begin{proof}
Write $h=Zv$.  Since $Z>0$ and $\cJ_\Gamma Z=0$, a direct computation gives
\[
 Z^{-1}\cJ_\Gamma(Zv)
 =\Delta_\Gamma v+2\nabla_\Gamma(\log Z)\cdot\nabla_\Gamma v,
\]
which has no zeroth-order term.  The usual maximum principle applies to $v$.
\end{proof}

\subsubsection{Invariant form of the metric}\label{subsec:invariantmetric}

For later weighted estimates it is useful to display the metric in symmetry coordinates.  Write
\[
 (\omega_1,r,\omega_2,\theta)
 \in S^3\times(0,\infty)\times S^3\times(\pi/4,\pi/2)
\]
and
\[
 X=(r\cos\theta\,\omega_1,r\sin\theta\,\omega_2,F_0(r,\theta)).
\]
The angular blocks are
\[
 r^2\cos^2\theta\,g_{S^3},
 \qquad
 r^2\sin^2\theta\,g_{S^3}.
\]
The $(r,\theta)$ block is
\[
 \begin{pmatrix}
 1+F_{0,r}^2 & F_{0,r}F_{0,\theta}\\
 F_{0,r}F_{0,\theta} & r^2+F_{0,\theta}^2
 \end{pmatrix}.
\]
Its determinant is
\[
 r^2\left(1+F_{0,r}^2+r^{-2}F_{0,\theta}^2\right)
 =r^2W_0^2.
\]
This makes visible two distinct geometric scales: ordinary angular derivatives cost $r^{-1}$, whereas derivatives in the steep direction of the graph are softened by the factor $W_0^{-1}\sim r^{-2}$.  The weighted coordinates $(s,t)$ introduced later are precisely designed to normalize these two scales while preserving the $u^3v^3$ orbit density.

\subsubsection{The homogeneous principal operator}

For $\phi=\phi(r,\theta)$, the operator in \eqref{Hprime} can be written as
\[
\widetilde L\phi=\widetilde L_0\phi+\widetilde L_1\phi,
\]
where the dominant part is obtained by replacing
\[
(r^{-4}+9g^2+(g')^2)^{-3/2}
\]
by
\[
(9g^2+(g')^2)^{-3/2}.
\]
Set
\[
w_0(\theta):=
\frac{\sin^3(2\theta)}{(9g^2+(g')^2)^{3/2}}.
\]
Then
\begin{align*}
L_0(\phi)
&=\frac1{r^7\sin^3(2\theta)}\Big\{
(9g^2w_0 r^3\phi_\theta)_\theta
+(r^5(g')^2w_0\phi_r)_r\notag\\
&\hspace{33mm}
-3(gg'w_0r^4\phi_r)_\theta
-3(gg'w_0r^4\phi_\theta)_r
\Big\}.
\end{align*}

This formula looks cumbersome, but it has an exceptionally useful separated structure.

\subsubsection{Separation of variables}

Take
\[\phi(r,\theta)=r^\beta q(\theta).\]
A direct computation yields
\begin{align}\label{sepcalc}
&r^7\sin^3(2\theta)L_0(r^\beta q)\notag\\
&\quad=r^{3+\beta}\Big[
9(g^2w_0q')'
-3\beta(gg'w_0q)'
+w_0(\beta+4)(\beta(g')^2q-3gg'q')
\Big].
\end{align}
The distinguished fact is that
\begin{equation}\label{kernelq}
q(\theta)=g(\theta)^{\beta/3}
\end{equation}
annihilates the angular operator.  Consequently \eqref{sepcalc} can be rewritten in divergence form as
\[\begin{aligned}
&r^7\sin^3(2\theta)L_0(r^\beta q)\\
&\qquad=9r^{3+\beta}g^{(\beta+4)/3}
\left[w_0g^{2/3}\left(g^{-\beta/3}q\right)'\right]'.
\end{aligned}\]
Thus the equation
\[L_0(r^\beta q)=\frac{p(\theta)}{r^{4-\beta}}\]
reduces to an explicit one-dimensional integral equation.

Integrating twice gives
\begin{align*}
q(\theta)
&=g(\theta)^{\beta/3}\Bigg[
A-\frac19\int_{\pi/4}^{\theta}
\frac{g(s)^{-2/3}(9g(s)^2+g'(s)^2)^{3/2}}{\sin^3(2s)}\notag\\
&\hspace{31mm}\times
\left(\int_s^{\pi/2}
p(\tau)g(\tau)^{-(\beta+4)/3}\sin^3(2\tau)\,d\tau\right)ds
\Bigg].
\end{align*}
This formula is one of the main organizing devices of the whole construction.

\subsubsection{Three forcing regimes}

Three values of $\beta$ play distinct roles.

\paragraph{(i) Forcing $r^{-2}$.}
Here $\beta=2$.  Even forcings lead to a correction of size
\[
\phi\sim r^2a_1(\theta).
\]
The associated normal perturbation is bounded because $|\grad F_0|\sim r^2$.

\paragraph{(ii) Forcing $r^{-3}$.}
Here $\beta=1$.  For odd $p(\theta)$, the formula gives
\[
q(\theta)\sim g(\theta)^{2/3}
\qquad\hbox{as }\theta\downarrow\pi/4.
\]
Hence
\[rq(\theta)\sim r g^{2/3}
=\frac{t^{2/3}}r,
\qquad t=r^3g(\theta).\]
This is the fractional term that later forces the inner layer.

\paragraph{(iii) Forcing $r^{-4}$.}
Here $\beta=0$.  In the even resonant case the natural particular solution contains a logarithm,
\[
\phi\sim A\log r+a_0(\theta).
\]

These three mechanisms give the outer algebra needed below.

There is also a fourth mode, not associated to a forcing but to the homogeneous kernel itself: $\beta=1/2$, hence
\[
\phi_{1/2}(r,\theta)=r^{1/2}g(\theta)^{1/6}.
\]
This mode is selected by the inner parabolic matching.

\section{Adapted coordinates and the outer expansion}\label{sec:outer}

Before introducing the new coordinates, it is useful to isolate the geometric issue they are designed to resolve.  The quotient variables $(u,v)$ make the $O(4)\times O(4)$ symmetry visible, and the polar variables $(r,\theta)$ are ideal for homogeneity.  They are not, however, uniformly adapted to the graph when $r$ is large.  The reason is that the slope of $F_0$ grows like $r^2$.  A vertical displacement, a normal displacement, and a displacement in the two-dimensional quotient plane therefore live on different scales.  If one expands the graph mean-curvature operator directly in $(r,\theta)$ and treats these directions as if they were comparable, terms which are geometrically of the same size appear with different powers of $r$ and the cancellations forced by minimality are difficult to see.

There is a second issue which is special to the $O(4)\times O(4)$ reduction.  The quotient plane is two-dimensional, but the divergence operator still remembers the six angular orbit directions through the factor $u^3v^3$.  Thus the tangential coordinate along a level set of $F_0$ cannot simply be Euclidean arclength in the $(u,v)$ plane.  The correct coordinate must absorb the orbit-volume density.  Once this is done, the principal part of the Jacobi operator takes a nearly product form: one derivative acts in the vertical direction $t=F_0$, the other along the weighted level sets.  This is the structural reason for the coordinate $s$ below.

The coordinates $(r,\theta)$ and $(s,t)$ will therefore be used for complementary purposes.  We keep $(r,\theta)$ whenever radial homogeneity, parity, or the angular ODE is the main point.  We pass to $(s,t)$ only for the nonlinear expansion and for the analysis near the cone, where the vertical direction must be separated cleanly from the weighted tangential direction.  No information is lost in changing coordinates; the point is to place each part of the argument in the variables in which its geometry is transparent.

\subsection{Why adapted coordinates are natural}

The geometry suggests using the two directions
\[
\frac{\grad F_0}{|\grad F_0|},
\qquad
\frac{\grad F_0^\perp}{|\grad F_0|}
\]
in the quotient plane $(u,v)$.  These are respectively normal and tangent to the level sets of $F_0$.  For the full $O(4)\times O(4)$-invariant problem, however, the tangential coordinate must also incorporate the orbit-volume factor $u^3v^3$ generated by the six angular directions.  The appropriate weighted coordinate does exactly this and places the reduced mean-curvature operator in a non-degenerate divergence form.  We therefore introduce $t$ as the height variable and $s$ as a weighted tangential variable along the level sets of $F_0$.

\subsubsection{Definition and interpretation}

Let $\mathbf u=(u,v)$ denote the quotient point.  Define $(s,t)$ by
\[\mathbf u_t=\frac{\grad F_0}{|\grad F_0|^2},
\qquad
\mathbf u_s=\frac1{(uv)^3}\frac{\grad F_0^\perp}{|\grad F_0|}.\]
The first equation immediately gives
\[t=F_0=r^3g(\theta),\]
up to an additive constant, which we choose so that $t=0$ on the cone.

The second coordinate is a weighted arclength along the level sets of $F_0$.  Solving the compatibility equations gives
\[s=\frac{r^7\sin^3(2\theta)g'(\theta)}{56\sqrt{9g(\theta)^2+g'(\theta)^2}}.\]
The frightening factor $r^7\sin^3(2\theta)$ is exactly the higher-dimensional orbit-volume information already visible in \eqref{polarMC}.

The coordinate vectors satisfy
\[\langle\mathbf u_t,\mathbf u_s\rangle=0,
\qquad
|\mathbf u_t|^2=|\grad F_0|^{-2},
\qquad
|\mathbf u_s|^2=(uv)^{-6}=:\rho^2.\]
Thus the coordinates are orthogonal in the quotient, but the $s$ direction carries precisely the weight needed to represent the eight-dimensional divergence correctly.

\subsubsection{Why $t=F_0$ is so useful}

For a function $F=F(s,t)$,
\[
\grad F
=F_t\grad F_0+\rho^{-1}F_s
\frac{\grad F_0^\perp}{|\grad F_0|}.
\]
Therefore
\[1+|\grad F|^2
=1+|\grad F_0|^2F_t^2+\rho^{-2}F_s^2.\]
If
\[F=F_0+A\phi=t+A\phi,\]
then
\[
F_t=1+A\phi_t,
\qquad
F_s=A\phi_s,
\]
and hence
\[1+|\grad F|^2
=1+|\grad F_0|^2(1+A\phi_t)^2+A^2\rho^{-2}\phi_s^2.\]
This is the algebraic simplification that makes the Taylor expansion usable.

This is the point that should be emphasized when explaining the coordinates to a reader: $(r,\theta)$ are better for understanding the geometry and homogeneity; $(s,t)$ are better for expanding the nonlinear operator because $F_0=t$ exactly.

\subsection{Mean curvature in $(s,t)$ coordinates}

The reduced mean curvature equation becomes
\[|\xi_0|\partial_t\left(
\frac{|\xi_0|F_t}{\sqrt{1+|\grad F|^2}}
\right)
+
|\xi_0|\partial_s\left(
\frac{\rho^{-2}F_s}{|\xi_0|\sqrt{1+|\grad F|^2}}
\right)=0,\]
where
\[
|\xi_0|^2:=1+|\grad F_0|^2.
\]
Equivalently, defining
\[
Q(F):=|\xi_0|^{-2}+F_t^2+\frac{\rho^{-2}F_s^2}{|\xi_0|^2},
\]
one may write
\begin{align*}
G[F]
&=Q(F)F_{tt}-\frac12Q_t(F)F_t\notag\\
&\quad+Q(F)\partial_s\left(\frac{\rho^{-2}F_s}{|\xi_0|^2}\right)
-\frac12Q_s(F)\frac{\rho^{-2}F_s}{|\xi_0|^2},
\end{align*}
with
\[
H[F]=\frac{|\xi_0|}{Q(F)^{3/2}}G[F].
\]

\subsubsection{Expansion around $F_0=t$}

Put $F=t+A\phi$ and define
\[
R_1[\phi]:=\phi_t^2+\frac{\rho^{-2}\phi_s^2}{|\xi_0|^2},
\]
\[
R:=1+|\xi_0|^{-2}+2A\phi_t+A^2R_1[\phi].
\]
A direct computation gives
\begin{equation}\label{polyA}
R^{3/2}H[t+A\phi]
=E_0+A E_1[\phi]+A^2E_2[\phi]+A^3E_3[\phi],
\end{equation}
where
\[E_0=-\frac12|\xi_0|\partial_t|\xi_0|^{-2}
=(1+|\xi_0|^{-2})^{3/2}H[F_0].\]
The linear part can be decomposed as
\[E_1[\phi]=\widetilde L_0[\phi]+\widetilde E_0[\phi],\]
with
\[\widetilde L_0[\phi]
=|\xi_0|\left[
\partial_t\left(\frac{\phi_t}{|\xi_0|^2}\right)
+\partial_s\left(\frac{\rho^{-2}\phi_s}{|\xi_0|^2}\right)
\right].\]
The lower-order linear correction is
\begin{align*}
\widetilde E_0[\phi]
&=-\frac32|\xi_0|\partial_t|\xi_0|^{-2}\phi_t
+|\xi_0|^{-1}\partial_s\left(\frac{\rho^{-2}\phi_s}{|\xi_0|^2}\right)\notag\\
&\quad-\frac12|\xi_0|\left(\frac{\rho^{-2}\phi_s}{|\xi_0|^2}\right)\partial_s|\xi_0|^{-2}.
\end{align*}
The quadratic part may be rewritten as
\begin{equation}\label{E2}
\begin{aligned}
E_2[\phi]
=|\xi_0|\Bigg[
&-\partial_t\left(\frac{\rho^{-2}}{|\xi_0|^2}\right)\phi_s^2
+2\phi_t\partial_s\left(\frac{\rho^{-2}\phi_s}{|\xi_0|^2}\right)\\
&-2\left(\frac{\rho^{-2}\phi_s}{|\xi_0|^2}\right)\phi_{ts}
\Bigg].
\end{aligned}
\end{equation}
The cubic part contains
\begin{equation}\label{E3danger}
E_3[\phi]
\supset
|\xi_0|\frac{\rho^{-2}\phi_s^2}{|\xi_0|^2}\phi_{tt}
\end{equation}
and additional lower-order combinations.

We collect the relevant terms of the expansion in a fixed notation.

We isolate the terms \eqref{E2}--\eqref{E3danger}, since they are exactly the terms that cease to be perturbative near the cone and produce the recentered inner operator.  The remaining terms are estimated perturbatively.

\subsection{The first even correction}

Define
\[
\cM_\eps[F]
:=H[F]-\frac{\eps}{\sqrt{1+|\grad F|^2}}.
\]
We seek
\[
F=F_0+\eps\Phi_1+\cdots.
\]
At first order the translator forcing is
\[\frac1{\sqrt{1+|\grad F_0|^2}}
=\frac{p_0(\theta)}{r^2}+O(r^{-6}),\]
with $p_0$ even across the cone.  The outer Jacobi equation therefore has homogeneity $\beta=2$, and we obtain
\[\Phi_1(r,\theta)=r^2a_1(\theta)+\Phi_{1,\mathrm{rem}},\]
where the remainder has better decay.

The function $a_1$ is even, not odd.  Thus generically
\[
a_1\left(\frac\pi4\right)\neq0.
\]
Consequently
\begin{equation}\label{coneshift}
F\big|_{u=v}
=\eps r^2a_1\left(\frac\pi4\right)+\cdots.
\end{equation}

Equation \eqref{coneshift} is the most transparent quantitative statement of the loss of odd symmetry: the perturbed translator no longer vanishes on Simons' cone, and its vertical displacement there is order $\eps r^2$.

\subsection{The second outer correction and the fractional singularity}

At the next relevant order, an odd forcing of order $r^{-3}$ occurs.  Solving with $\beta=1$ gives
\[
\Phi_2^{\mathrm{out}}=rq(\theta),
\]
where near the cone
\[
q(\theta)\sim g(\theta)^{2/3}.
\]
Using $t=r^3g(\theta)$,
\begin{equation}\label{Phi2outer}
\Phi_2^{\mathrm{out}}
\sim \frac{\operatorname{sgn}(t)|t|^{2/3}}r.
\end{equation}
Hence
\[
\partial_t\Phi_2^{\mathrm{out}}
\sim r^{-1}|t|^{-1/3},
\qquad
\partial_{tt}\Phi_2^{\mathrm{out}}
\sim r^{-1}|t|^{-4/3}.
\]

In the elliptic De Giorgi problem this singular derivative was not fatal because the correction was odd and the cone was a symmetry boundary: one solved on one side and extended oddly.  Here that mechanism is unavailable because \eqref{coneshift} has already broken the odd symmetry.

The singularity of \eqref{Phi2outer} requires an inner regularization.

The reason is that the problem is not merely that $t$ derivatives are singular, but that the cone can no longer be removed from the problem by odd symmetry.  The true solution must cross the cone smoothly.

\subsection{Recentering the vertical coordinate}\label{sec:recentering}

Consider
\[
F=F_0+\eps\Phi_1+\eps^2\Psi.
\]
We collect the terms linear in $\Psi$ that can become large near the cone.  From \eqref{E2}, the interaction between $\Phi_1$ and $\Psi$ contains
\[
-2\eps\,\rho^{-2}\Phi_{1,s}\Psi_{st}
\]
at principal level.  From \eqref{E3danger}, the cubic interaction contributes
\[
\eps^2\rho^{-2}\Phi_{1,s}^2\Psi_{tt}.
\]
Thus the second-order part in the tangential direction is schematically
\begin{equation}\label{squareop}
\partial_s^2
-2\eps\Phi_{1,s}\partial_{st}
+\eps^2\Phi_{1,s}^2\partial_{tt}.
\end{equation}
Ignoring derivatives of coefficients for one moment,
\[
\eqref{squareop}
=\left(\partial_s-\eps\Phi_{1,s}\partial_t\right)^2.
\]
This suggests defining
\[T=t+\eps\Phi_1(s,t).\]
Indeed, differentiation in $s$ at fixed $T$ gives
\[
\left.\partial_s\right|_T
=\partial_s-\eps\Phi_{1,s}\partial_t+\hbox{lower-order terms}.
\]
Therefore the apparently dangerous mixed terms are simply the expression, in the coordinate $t$, of the tangential second derivative in the new vertical coordinate $T$.

This identity gives the geometric explanation of the recentering: $\widetilde t=t+\eps\Phi_1$ is precisely the coordinate that completes the square in the principal part of the operator.

Moreover, near the cone
\[
T=t+O(\eps r^2),
\]
so the natural transition layer is centered around the translated surface rather than around the set $t=0$.

\section{The recentered inner transition layer}\label{sec:innerlayer}

We now enter the only region where the outer expansion is not a genuine asymptotic expansion in smooth functions.  It is worth describing the logic before carrying out the calculation.  The first correction moves the zero set of the graph by an amount of order $\varepsilon r^2$.  If one continues to measure the next correction using the unrecentered vertical coordinate $t=F_0$, then the singular factor $|t|^{2/3}$ is being expanded around the wrong center.  The singularity is therefore partly artificial: it records that the transition surface has moved while the coordinate system has not.

The recentered coordinate
\[
 T=t+\varepsilon\Phi_1
\]
removes this geometric mismatch.  After this change, the second-order combinations produced by the Taylor expansion organize themselves into the square of a transported tangential derivative.  Near the cone this has a decisive consequence.  If $h$ denotes a normal perturbation and one factors out the distinguished radial weight $r^{-3}$, the curvature potential cancels the radial zero-order term and the principal operator becomes, to first approximation,
\[
 H_{TT}+H_{rr}-\varepsilon H_T.
\]
Thus the vertical coordinate is not merely another elliptic variable: on the long transition scale it plays the role of time.  The parabolic equation that appears below is therefore not an analogy imposed on the problem; it is the leading operator obtained after the correct geometric recentering.

This viewpoint also explains why one must solve the transition problem globally across $T=0$.  In the minimal and Allen--Cahn constructions the relevant correction is odd and the Simons cone can be treated as a symmetry boundary.  The translator source is positive and even.  The first correction is even, the next forcing is odd, and the full solution no longer vanishes on the cone.  Hence the two outer sides must be joined by a smooth profile through an interior layer.  The self-similar equation derived in this section is precisely the matching equation for that profile.

Near Simons' cone one may use $r$ as the tangential large-scale variable and $T$ as the transverse variable.  Passing from vertical perturbations to normal perturbations gives at leading order
\[\cL_{\mathrm{in}}h
=h_{TT}+h_{rr}+\frac6r h_r+\frac6{r^2}h-\eps h_T
+\mathcal E_{\mathrm{low}}[h].\]
The coefficient $6/r$ is the radial divergence coefficient in the effective seven-dimensional transversal geometry of the cone, while $6/r^2$ is the corresponding potential term produced by the Jacobi operator.

There is a useful exact cancellation.  Put
\[h=r^{-3}H.\]
Then
\begin{align*}
\partial_r(r^{-3}H)&=-3r^{-4}H+r^{-3}H_r,\\
\partial_{rr}(r^{-3}H)&=12r^{-5}H-6r^{-4}H_r+r^{-3}H_{rr}.
\end{align*}
Hence
\begin{align*}
&\left(\partial_{rr}+\frac6r\partial_r+\frac6{r^2}\right)(r^{-3}H)\\
&\qquad=\left(12-18+6\right)r^{-5}H
+\left(-6+6\right)r^{-4}H_r+r^{-3}H_{rr}\\
&\qquad=r^{-3}H_{rr}.
\end{align*}
Thus, neglecting lower-order geometric errors,
\begin{equation}\label{HTeq}
H_{TT}+H_{rr}-\eps H_T=0.
\end{equation}

\subsection{The parabolic scale}

We compare radial diffusion and translating drift.  If $H$ varies on radial scale $r$ and transverse scale $L_T$, then
\[
H_{rr}\sim r^{-2}H,
\qquad
\eps H_T\sim \eps L_T^{-1}H.
\]
Balancing them gives
\begin{equation}\label{innerbalance}
L_T\sim\eps r^2.
\end{equation}
This is exactly the size of the first vertical displacement \eqref{coneshift}.  Thus the same scale is predicted independently by geometry and by the operator.

The remaining second transverse derivative has size
\[
H_{TT}\sim (\eps^2r^4)^{-1}H.
\]
Hence in the far-field portion of the inner regime, where $\eps r\gg1$, it is lower order compared with $H_{rr}\sim r^{-2}H$.  There the leading equation is
\[
\eps H_T-H_{rr}=0.
\]
With
\[
\tau=\frac{T}{\eps},
\]
we obtain the heat equation
\begin{equation}\label{heat}
H_\tau-H_{rr}=0.
\end{equation}

We separate the full inner operator \eqref{HTeq}, the balance \eqref{innerbalance}, and the far-field parabolic reduction \eqref{heat}.

\subsection{The self-similar inner profile}

The outer normal correction corresponding to \eqref{Phi2outer} behaves like
\[
h_{\mathrm{out}}\sim \frac{\operatorname{sgn}(T)|T|^{2/3}}{r^3}.
\]
We therefore seek a self-similar solution of \eqref{heat}.  Since the heat scaling is $\tau\sim r^2$, set
\[H(r,\tau)=r^{4/3}Q(p),
\qquad
p=\frac{\tau}{r^2}=\frac{T}{\eps r^2}.\]
We compute
\[
H_\tau=r^{-2/3}Q'(p),
\]
while
\[
H_r=r^{1/3}\left(\frac43Q-2pQ'\right).
\]
Differentiating once more,
\[
H_{rr}=r^{-2/3}\left[
\frac49Q-\frac23pQ'+4p^2Q''
\right].
\]
Thus \eqref{heat} becomes
\begin{equation}\label{QODE}
4p^2Q''+\left(\frac23p-1\right)Q'+\frac49Q=0.
\end{equation}
This is the self-similar inner ODE.

\subsection{Asymptotics of the inner ODE}

\subsubsection{Indicial roots at infinity}

Before solving \eqref{QODE}, one must determine every algebraic mode.  Suppose
\[
Q(p)\sim |p|^\alpha.
\]
The term $-Q'$ is one order lower, so the indicial equation is
\[
4\alpha(\alpha-1)+\frac23\alpha+\frac49=0.
\]
Equivalently,
\[
36\alpha^2-30\alpha+4=0
=2(3\alpha-2)(6\alpha-1).
\]
Hence
\[\boxed{\alpha_1=\frac23,\qquad \alpha_2=\frac16.}\]

A first formal expansion based only on the $2/3$ branch would give
\[
Q(p)=\pm|p|^{2/3}(1+c_\pm|p|^{-1}+O(|p|^{-2})).
\]
This is incomplete.  There is a second independent algebraic branch $|p|^{1/6}$, and it is larger than the $|p|^{-1/3}$ correction of the $2/3$ branch.

\subsubsection{Expansion of the $2/3$ branch}

For $p\to+\infty$, write
\[
Q=p^{2/3}\left(1+a_1p^{-1}+a_2p^{-2}+\cdots\right).
\]
Substitution into \eqref{QODE} gives recursively
\[
a_1=\frac13,
\qquad
a_2=-\frac1{108},
\]
so
\[Q_{2/3}(p)
=p^{2/3}\left(1+\frac1{3p}-\frac1{108p^2}+O(p^{-3})\right).\]
On the negative side the sign of the lower-order $-Q'$ contribution changes, producing
\[Q_{2/3}(-x)
=x^{2/3}\left(1-\frac1{3x}-\frac1{108x^2}+O(x^{-3})\right),
\qquad x\to+\infty,\]
up to the overall sign chosen for the desired odd-like leading behavior.

\subsubsection{Expansion of the $1/6$ branch}

Similarly,
\[Q_{1/6}(p)
=p^{1/6}\left(1+\frac1{36p}+O(p^{-2})\right)\]
for $p\to+\infty$, with the corresponding sign change in the first correction at the negative end.

Thus the algebraic scales are ordered as
\begin{equation}\label{scaleorder}
|p|^{2/3},\qquad |p|^{1/6},\qquad |p|^{-1/3},\qquad |p|^{-5/6},\ldots
\end{equation}
The second term in \eqref{scaleorder} must be retained in the matching.

\subsection{Behavior at $p=0$ and the Kummer representation}

For $p<0$ write
\[
p=-x,\qquad x>0,
\qquad z=\frac1{4x},
\qquad Q(-x)=x^{2/3}Y(z).
\]
A direct substitution gives
\[zY''+\left(\frac12-z\right)Y'+\frac23Y=0.\]
This is Kummer's equation with
\[
a=-\frac23,
\qquad b=\frac12.
\]
As $p\to0^-$ one has $z\to+\infty$.  One Kummer branch grows exponentially and cannot produce a smooth $Q$ at $p=0$.  Up to scale, the admissible branch is therefore the algebraic solution
\[Q(-x)=C x^{2/3}U\left(-\frac23,\frac12,\frac1{4x}\right).\]
As $x\to\infty$, $z\to0$, and the standard connection formula for $U$ contains both a constant term and a $z^{1/2}$ term.  Therefore
\[Q(-x)=A_-x^{2/3}+B_-x^{1/6}+C_-x^{-1/3}+\cdots,\]
with
\[B_-\neq0.\]
Thus the $1/6$ mode is forced by smoothness at $p=0$ on the negative side; it cannot be discarded by a choice of normalization.

A regular power series
\[
Q(p)=a_0+a_1p+a_2p^2+\cdots
\]
has
\[
a_1=\frac49a_0,
\]
and the entire Taylor jet is recursively fixed by $a_0$.

On the positive side there is, in addition, a flat solution associated with the irregular singular point at $p=0$.  This point deserves a complete calculation because the irregular singularity allows an additional flat solution.

Introduce the unified Kummer variable
\[
w=-\frac1{4p}.
\]
If
\[
Q(p)=p^{2/3}Y(w)
\]
(with the appropriate real branch understood on each side), then $Y$ solves
\[wY''+\left(\frac12-w\right)Y'+\frac23Y=0.\]
Thus $a=-2/3$ and $b=1/2$ in the standard Kummer notation.

For $p>0$ one has $w\to-\infty$ as $p\downarrow0$.  Kummer's transformation shows that the solution exponentially small in this limit is
\begin{equation}\label{flatmodeexact}
Q_{\mathrm{flat}}(p)
=p^{2/3}e^{-1/(4p)}
U\left(\frac76,\frac12,\frac1{4p}\right).
\end{equation}
Indeed, since
\[
U\left(\frac76,\frac12,z\right)\sim z^{-7/6},
\qquad z\to+\infty,
\]
we obtain
\[Q_{\mathrm{flat}}(p)
=4^{7/6}p^{11/6}e^{-1/(4p)}(1+O(p)),
\qquad p\downarrow0.\]
In particular
\[
D^kQ_{\mathrm{flat}}(0^+)=0
\qquad\hbox{for every }k\ge0.
\]
Hence one may continue the regular Taylor jet through $p=0$ and still add
\[
\lambda Q_{\mathrm{flat}}.
\]
This is the Stokes degree of freedom.

The crucial point is that this degree of freedom really changes the leading coefficient at $+\infty$.  To see this, use the connection formula for $U$ at the origin:
\[
U\left(\frac76,\frac12,z\right)
=
\frac{\Gamma(1/2)}{\Gamma(5/3)}
+
\frac{\Gamma(-1/2)}{\Gamma(7/6)}z^{1/2}
+O(z),
\qquad z\to0^+.
\]
Substituting $z=(4p)^{-1}$ in \eqref{flatmodeexact} gives
\[Q_{\mathrm{flat}}(p)
=
A_{\mathrm{flat}}p^{2/3}
+B_{\mathrm{flat}}p^{1/6}
+O(p^{-1/3}),
\qquad p\to+\infty,\]
where
\[
A_{\mathrm{flat}}
=
\frac{\Gamma(1/2)}{\Gamma(5/3)}\neq0.
\]
Consequently the Stokes parameter acts nontrivially on the coefficient of the dominant $p^{2/3}$ mode.

We may therefore normalize the unique smooth negative branch so that
\[
A_-=-1.
\]
Let $Q_{\mathrm{reg}}$ denote its regular continuation to $p>0$, and write
\[
Q_{\mathrm{reg}}(p)
=A_+^{\mathrm{reg}}p^{2/3}+B_+^{\mathrm{reg}}p^{1/6}+O(p^{-1/3}).
\]
Then the choice
\[\lambda
=\frac{1-A_+^{\mathrm{reg}}}{A_{\mathrm{flat}}}\]
produces a global $C^\infty$ solution
\[
Q=Q_{\mathrm{reg}}+\lambda Q_{\mathrm{flat}}\mathbf 1_{\{p>0\}}
\]
(with the flat term extended by zero to $p\le0$) satisfying
\[
A_-=-1,
\qquad
A_+=1.
\]
Because the extension of the flat term has zero jet at $p=0$, this piecewise definition is $C^\infty$ and solves the ODE classically on both sides and by continuity at $p=0$.

We have therefore proved the following lemma.

\begin{remark}[Preliminary inner-profile statement]\label{InnerProfileLemma}
The preceding calculations identify the two indicial branches $p^{2/3}$ and $p^{1/6}$ and the first terms of their asymptotic expansions.  A complete global existence statement, including smooth passage through $p=0$ and the Stokes parameter, is proved in Lemma~\ref{innerprofilefinal}.
\end{remark}

The independent $|p|^{1/6}$ mode must be retained, and the smooth profile can be normalized to have opposite unit leading coefficients at the two ends.

\subsection{Translation of inner modes to outer Jacobi modes}

The normal inner correction is normalized by
\begin{equation}\label{hinnerfull}
h_{\mathrm{in}}(r,T)
=\eps^{2/3}r^{-5/3}
Q\left(\frac{T}{\eps r^2}\right).
\end{equation}
Let
\[
p=\frac{T}{\eps r^2}.
\]
We now translate each algebraic mode of $Q$ into an outer term.

\subsubsection{The $p^{2/3}$ mode}

\[
\eps^{2/3}r^{-5/3}p^{2/3}
=\frac{T^{2/3}}{r^3}.
\]
This is exactly the normal form of the $r^{-3}$ outer correction.

\subsubsection{The missing $p^{1/6}$ mode}

\[
\eps^{2/3}r^{-5/3}p^{1/6}
=\eps^{1/2}\frac{T^{1/6}}{r^2}.
\]
To pass from normal to vertical displacement we multiply, to leading order, by
\[
|\grad F_0|\sim r^2.
\]
Hence the corresponding vertical term is
\[
\eps^{1/2}T^{1/6}
\sim \eps^{1/2}r^{1/2}g(\theta)^{1/6}.
\]
Since the whole second correction enters $F$ with prefactor $\eps^2$, the actual graph contains
\[\boxed{
\eps^{5/2}B_\pm r^{1/2}g(\theta)^{1/6}.}\]
But by \eqref{kernelq},
\[
r^{1/2}g^{1/6}
\]
is exactly the homogeneous outer Jacobi mode corresponding to $\beta=1/2$.

The new fractional order $\eps^{5/2}$ is therefore not an arbitrary patch: it is the outer homogeneous Jacobi mode selected by the second algebraic branch of the inner heat profile.

\subsubsection{The $p^{-1/3}$ mode}

The next correction of the $2/3$ branch gives
\[
\eps^{2/3}r^{-5/3}|p|^{-1/3}
=\eps\frac{|T|^{-1/3}}r.
\]
This is the next matching correction and occurs after the $|p|^{1/6}$ contribution.

\subsection{The uniform inner--outer ansatz}

A truncated outer expansion would have the schematic form
\[F_{\mathrm{tr}}
=F_0+\eps\Phi_1+\eps^2\Phi_2+\eps^3\Phi_3+\cdots.\]
The uniform expansion used below is
\[\boxed{
F_{\mathrm{app}}
=F_0+\eps\Phi_1
+\eps^2\Phi_2
+\eps^{5/2}\Phi_{5/2}
+\eps^3\Phi_3+\cdots.}\]
The term $\Phi_{5/2}$ is homogeneous in the far outer region,
\[
\Phi_{5/2}^{\pm}=B_\pm r^{1/2}g_\pm^{1/6},
\]
but it is not continued to the cone.  It is only the far-field expansion of the smooth inner profile \eqref{hinnerfull}.

\subsection{Matching regions and cutoff estimates}

Consider first the matching scale
\begin{equation}\label{matchregion0}
|T|\sim r^{3-\sigma}.
\end{equation}
Then
\[
|p|=\frac{|T|}{\eps r^2}
\sim\frac{r^{1-\sigma}}\eps\gg1.
\]

If the $p^{1/6}$ term is omitted, the leading inner--outer mismatch is
\[
\eps^{1/2}\frac{|T|^{1/6}}{r^2}.
\]
At \eqref{matchregion0},
\[
|T|^{1/6}=r^{(3-\sigma)/6}=r^{1/2-\sigma/6},
\]
so
\[|h_{\mathrm{in}}-h_{\mathrm{out}}^{(0)}|
\gtrsim
\eps^{1/2}r^{-3/2-\sigma/6}.\]
This is too large for the residual estimate required below.

Thus the $p^{1/6}$ term is a necessary matching contribution and must be included in the cutoff estimate.

After including both $p^{1/6}$ and $p^{-1/3}$, the first omitted algebraic scale is $p^{-5/6}$.  Hence
\begin{align*}
|h_{\mathrm{in}}-h_{\mathrm{out}}^{\mathrm{new}}|
&\le C\eps^{2/3}r^{-5/3}|p|^{-5/6}\notag\\
&\le C\eps^{3/2}r^{-5/2+5\sigma/6}.
\end{align*}
A radial cutoff on scale $r$ costs two derivatives, producing
\[O\left(\eps^{3/2}r^{-9/2+5\sigma/6}\right)\]
in the normal Jacobi equation.  Since this layer is multiplied by $\eps^2$ in the graph expansion, the actual translator residual carries factor $\eps^{7/2}$.

For $\sigma<3/5$,
\[
-\frac92+\frac56\sigma<-4,
\]
so the resulting cutoff error decays faster than the critical $r^{-4}$ scale.

\subsection{Three-region decomposition}

The previous discussion suggests organizing the proof in three regions.

\paragraph{Inner region.}
\[|T|\le K\eps r^2.\]
Use the smooth inner profile $Q$ directly.

\paragraph{Overlap region.}
\[K\eps r^2<|T|<r^{3-\delta}.\]
Here $|p|\gg1$ and the inner profile is expanded in the algebraic modes
\[
|p|^{2/3},\quad |p|^{1/6},\quad |p|^{-1/3},\ldots
\]
which are matched to the separated outer Jacobi modes.

\paragraph{Far outer region.}
\[|T|\ge r^{3-\delta}.\]
Return completely to $(r,\theta)$ and use the outer Jacobi theory.

This decomposition makes the role of each coordinate system transparent and prevents the inner variables from contaminating the entire proof.

\subsection{Nonlinear consistency of the $\eps^{5/2}$ mode}

We next verify that the additional mode does not create a larger error.

The outer profile is
\[
\Phi_{5/2}=B_\pm r^{1/2}g^{1/6}.
\]
By construction it satisfies
\[L_0\Phi_{5/2}=0.\]
Therefore its leading linear contribution vanishes.  The difference between the exact operator and $L_0$ contains extra factors arising from the $r^{-4}$ correction in
\[
(r^{-4}+9g^2+(g')^2)^{-3/2},
\]
so away from the cutoff region the linear residual gains several powers of $r^{-1}$.

The potentially dangerous interactions with $\Phi_1$ are exactly the mixed second derivatives that formed the completed square \eqref{squareop}.  They must not be estimated as errors in the variable $t$; they are incorporated into the principal operator after the change $T=t+\eps\Phi_1$.

To organize all remaining singular angular factors we define
\[\eta:=\frac{\eps r^2}{|T|}.\]
In the outer region $\eta\ll1$.  Since $T\sim r^3g$ there,
\[
g^{-a}\sim\left(\frac{r^3}{|T|}\right)^a
=\eps^{-a}r^a\eta^a.
\]
Thus every negative power of $g$ can be rewritten as a controlled power of $\eta$ and $\eps$.

For example,
\[
\partial_\theta\Phi_{5/2}\sim r^{1/2}g^{-5/6},
\]
and hence
\[
\eps^{5/2}|\partial_\theta\Phi_{5/2}|
\lesssim
\eps^{5/3}r^{4/3}\eta^{5/6},
\]
which is small in the overlap after the metric factors in the Jacobi operator are inserted.

The preceding observation can be organized as a weighted bookkeeping lemma.  We state it here in the form needed later.  Its proof is an explicit inspection of the quadratic and cubic terms in \eqref{polyA}; the point of the statement is to record which terms are absorbed into the recentered principal operator and which are true remainders.

\begin{lemma}[Outer nonlinear bookkeeping]\label{NonlinearBookkeeping}
Fix $0<\delta<1/2$ and consider the outer region
\[
|T|\ge K\eps r^2,
\qquad K\gg1.
\]
Let
\[
\Phi=\Phi_1+\eps\Phi_2+\eps^{3/2}\Phi_{5/2}+\eps^2\Phi_3,
\]
where the profiles have the asymptotic degrees
\[
\Phi_1=O(r^2),\qquad
\Phi_2=O(rg^{2/3}),\qquad
\Phi_{5/2}=O(r^{1/2}g^{1/6}),
\]
and where $\Phi_{5/2}$ is chosen in the homogeneous kernel $L_0\Phi_{5/2}=0$.  After changing from $t$ to $T=t+\eps\Phi_1$, all terms linear in $\Phi_2$ or $\Phi_{5/2}$ which have the apparent size generated by
\[
-2\eps\rho^{-2}\Phi_{1,s}\partial_{Ts}
+\eps^2\rho^{-2}\Phi_{1,s}^2\partial_{TT}
\]
are part of the principal operator.  Every remaining contribution containing $\Phi_{5/2}$ carries either
\begin{enumerate}[label=(\alph*)]
\item an additional factor $\eta^\kappa$, $\kappa>0$, where $\eta=\eps r^2/|T|$, or
\item an additional factor $r^{-1}$ relative to the critical $r^{-4}$ scale, or
\item total prefactor at least $\eps^5$.
\end{enumerate}
Consequently, in the overlap region $K\eps r^2\le |T|\le r^{3-\delta}$, these terms are lower order with respect to the first unmatched outer residual, uniformly for $K$ large and $\eps$ small.
\end{lemma}

\paragraph{Proof.}
The potentially largest terms are those linear in the new profile and involving one copy of $\Phi_1$.  In the $(s,t)$ variables they arise from the $A^2$ and $A^3$ pieces in \eqref{polyA}; schematically they are
\[
-2\eps\rho^{-2}\Phi_{1,s}\Phi_{5/2,Ts},
\qquad
\eps^2\rho^{-2}\Phi_{1,s}^2\Phi_{5/2,TT}.
\]
As explained in Section~\ref{sec:recentering}, these two expressions are not errors: together with the $s$--second derivative they form
\[
\left(\partial_s-\eps\Phi_{1,s}\partial_t\right)^2\Phi_{5/2}
=\partial_{s|T}^2\Phi_{5/2}+\text{lower-order terms}.
\]
Estimating them separately would therefore give misleadingly large powers.

For the genuine remainders we use
\[
g\sim \frac{|T|}{r^3}
=\frac{\eps}{\eta r}.
\]
Thus, for example,
\[
\eps^{5/2}r^{1/2}g^{-5/6}
=\eps^{5/3}r^{4/3}\eta^{5/6}.
\]
The inverse metric coefficient in the angular direction contributes $r^{-2}(1+r^4(g')^2)^{-1}$, while in the cone regime $g'$ stays bounded away from zero.  Hence two angular derivatives recover the radial decay that is hidden in the raw derivative $g^{-5/6}$.  The terms with one angular and one recentered transversal derivative acquire, in addition, a factor $\eta$ because
\[
\partial_T\sim |T|^{-1}=\eta(\eps r^2)^{-1}.
\]
Finally, all terms quadratic in $\Phi_{5/2}$ carry the prefactor $\eps^5$ and are harmless at the orders presently retained.

The key point is that the dangerous mixed terms are incorporated into the operator after recentering, rather than estimated as nonlinear errors.

\section{The exact BDG graph and weighted Jacobi theory}\label{sec:exactbdg}

Up to this point the principal calculations have been carried out on the homogeneous cubic graph
\[
\Gamma_0=\{(x,F_0(x)):x\in\R^8\},\qquad F_0=r^3g(\theta).
\]
The actual Bombieri--De Giorgi--Giusti graph is
\[
\Gamma=\{(x,\olF(x)):x\in\R^8\},
\]
and the construction of a translator must ultimately be performed around $\Gamma$, not around $\Gamma_0$.  The purpose of this section is therefore not cosmetic: we explain in detail why every model computation made on $\Gamma_0$ survives on $\Gamma$ with an error which is strictly lower order.

We isolate here the geometric transfer statement that will be used repeatedly later.

\subsubsection{Zeroth-order closeness}

The starting estimate is
\begin{equation}\label{zeroclose}
0\le \olF-F_0\le Cr^{-\sigma_0}
\qquad\hbox{in }T=\{v>u\},\quad r\gg1,
\end{equation}
for some $\sigma_0\in(0,1)$.  By odd reflection the corresponding estimate holds on the opposite sector with the appropriate sign.

At first sight one might try to differentiate \eqref{zeroclose}.  This is not legitimate.  The graph $F_0$ has slope of size
\[
|\grad F_0|\sim r^2,
\]
so a vertical error of size $r^{-\sigma_0}$ can correspond to a much smaller \emph{normal} displacement, and Euclidean derivatives in the steep direction may lose powers of $r$.  The correct statement is intrinsic to the graph.

To see the scaling, the unit normal to $\Gamma_0$ is
\[
\nu_0(x)=\frac{(-\grad F_0(x),1)}{\sqrt{1+|\grad F_0(x)|^2}},
\]
so its vertical component satisfies
\[
\nu_0\cdot e_9\sim r^{-2}.
\]
Consequently a vertical separation $O(r^{-\sigma_0})$ corresponds to a normal separation of order
\[
O(r^{-2-\sigma_0}).
\]
This two-power gain is the basic geometric fact behind all derivative estimates below.

\subsubsection{Local graph representation on tangent planes}

Fix $p_0\in\Gamma_0$ with $r(p_0)=R\gg1$.  The curvature of both $\Gamma_0$ and $\Gamma$ is $O(R^{-1})$ on balls of radius comparable with $R$.  In particular, after choosing a sufficiently small universal $\vartheta>0$, both surfaces can be represented in $B(p_0,\vartheta R)$ as graphs over the tangent plane $T_{p_0}\Gamma_0$.

Choose an orthonormal basis $\Pi_1,\ldots,\Pi_8$ of $T_{p_0}\Gamma_0$.  We write
\begin{align}
 p(t)&=p_0+\sum_{j=1}^8t_j\Pi_j+G_0(t)\nu_0(p_0),\\
 q(t)&=p_0+\sum_{j=1}^8t_j\Pi_j+G(t)\nu_0(p_0),
\end{align}
for $|t|<\vartheta R$, where $p(t)\in\Gamma_0$ and $q(t)\in\Gamma$.

For the homogeneous graph, direct differentiation of the parametrization gives
\[|D^mG_0(t)|\le C_mR^{1-m},\qquad m\ge1,\]
and the same estimate holds for the coordinate map $x=x(t)$.  This is simply the intrinsic version of the homogeneity
\[
D^mF_0=O(R^{3-m}),
\]
combined with the large slope $|\grad F_0|\sim R^2$.

The key improvement is that \eqref{zeroclose} implies much more than
\[
|G-G_0|=O(R^{-\sigma_0}).
\]
Indeed, because the normal makes an angle of order $R^{-2}$ with the vertical direction, one obtains
\begin{equation}\label{normalclose}
\|G-G_0\|_{L^\infty(B_{\vartheta R})}\le CR^{-2-\sigma_0}.
\end{equation}

The estimate $\olF-F_0=O(r^{-\sigma_0})$ is vertical.  The quantity to which local elliptic theory is applied is the normal graph difference $G-G_0$, and that is two powers smaller.

\subsubsection{Schauder improvement of the normal graph difference}

Both $G$ and $G_0$ solve uniformly elliptic minimal-surface equations in the tangent coordinates $t$, after rescaling by $R$.  Put
\[
\widetilde G(y)=R^{-1}G(Ry),\qquad
\widetilde G_0(y)=R^{-1}G_0(Ry),\qquad
\widetilde h=\widetilde G-\widetilde G_0.
\]
On a fixed ball $B_{2\vartheta}\subset\R^8$, the coefficients of the equation for $\widetilde h$ have uniformly bounded $C^k$ norms and a uniform ellipticity constant.  From \eqref{normalclose},
\[
\|\widetilde h\|_{L^\infty(B_{2\vartheta})}\le CR^{-3-\sigma_0}.
\]
The inhomogeneity produced by the fact that $F_0$ is only an approximate minimal graph is $O(R^{-4})$ after scaling.  Interior Schauder estimates therefore give, for each fixed $m\ge1$,
\[
\|D^m\widetilde h\|_{L^\infty(B_{\vartheta})}
\le C_mR^{-3-\sigma_0},
\]
after possibly decreasing $\sigma_0$.  Returning to the original scale yields
\begin{equation}\label{Gderall}
|D^m(G-G_0)(t)|\le C_mR^{-m-2-\sigma_0},
\qquad |t|<\vartheta R.
\end{equation}

In intrinsic notation this implies the estimate used throughout the paper:
\begin{equation}\label{FbarF0derivs}
|D_{\Gamma_0}^m(\olF-F_0)|\le C_m r^{-m-\sigma_0},
\qquad m\ge1,
\end{equation}
together with
\[
\olF-F_0=O(r^{-\sigma_0}).
\]
The apparent mismatch between \eqref{Gderall} and \eqref{FbarF0derivs} is exactly the conversion between normal and vertical graph displacement.

\subsubsection{Normal projection and comparison of Jacobi operators}

For $r$ large the normal projection defines a smooth diffeomorphism
\[
\pi:\Gamma\longrightarrow\Gamma_0,
\]
with
\[
\pi(y)=y+t_y\nu(y),\qquad |t_y|\le Cr^{-2-\sigma_0}.
\]
If $h$ on $\Gamma$ and $h_0$ on $\Gamma_0$ are related by
\[
h(y)=h_0(\pi(y)),
\]
then \eqref{Gderall} gives
\begin{align*}
D_\Gamma h
 &=\bigl[D_{\Gamma_0}h_0
 +O(r^{-2-\sigma_0})D_{\Gamma_0}h_0\bigr]\circ\pi,\\
D_\Gamma^2 h
 &=\bigl[D_{\Gamma_0}^2h_0
 +O(r^{-2-\sigma_0})D_{\Gamma_0}^2h_0
 +O(r^{-3-\sigma_0})D_{\Gamma_0}h_0\bigr]\circ\pi.
\end{align*}
Since the second fundamental forms also differ by $O(r^{-2-\sigma_0})$ relative to their natural size, the Jacobi operators satisfy
\begin{align*}
\cJ_\Gamma h
&=\Bigl[\cJ_{\Gamma_0}h_0
+O(r^{-2-\sigma_0})D_{\Gamma_0}^2h_0
+O(r^{-3-\sigma_0})D_{\Gamma_0}h_0\notag\\
&\hspace{35mm}+O(r^{-4-\sigma_0})h_0\Bigr]\circ\pi.
\end{align*}

\begin{proposition}[Transfer from $\Gamma_0$ to $\Gamma$]\label{transferprop}
Let $h_0$ be an $O(4)\times O(4)$-invariant function on $\Gamma_0$ satisfying, for some $\gamma\in\R$,
\[
|D_{\Gamma_0}^jh_0|\le Cr^{-\gamma-j},\qquad j=0,1,2.
\]
Let $h=h_0\circ\pi$ on $\Gamma$.  Then
\[
\cJ_\Gamma h
=\bigl(\cJ_{\Gamma_0}h_0\bigr)\circ\pi
+O(r^{-\gamma-4-\sigma_0}).
\]
In particular, every strict Jacobi barrier on $\Gamma_0$ with margin $cr^{-\gamma-4}$ remains a strict barrier on $\Gamma$ for $r$ sufficiently large.
\end{proposition}

This proposition collects the derivative and comparison information required by the later barrier argument.

\subsubsection{A refined BDG expansion}\label{subsec:refinedBDG}

The rough estimate $\bar F-F_0=O(r^{-\sigma_0})$ is enough for transferring a \emph{strict} barrier, but it is not enough for a discrete expansion of the Taylor coefficients of the translator equation.  To obtain the additional asymptotic level required here, there are two logically distinct steps:

\begin{enumerate}
\item first one solves away the leading residual $H[F_0]=r^{-5}E_5(\theta)+\cdots$ by a genuine $O(r^{-1})$ correction;
\item only \emph{after that} does one use the near/far supersolutions to control the smaller remainder and compare with the exact BDG graph.
\end{enumerate}

A direct calculation gives
\[H[F_0]=r^{-5}E_5(\theta)+O(r^{-9}),
 \qquad |E_5(\theta)|\le C g(\theta),\]
with $E_5$ odd under $\theta\mapsto\pi/2-\theta$.

\begin{lemma}[Borderline linear correction]\label{borderlineBDG}
There exists an odd solution $\Phi_0$ of the exterior linear problem
\begin{equation}\label{Phi0eq}
 H'(F_0)[\Phi_0]=-H[F_0]
 \qquad\hbox{for }r>R_0,
\end{equation}
with zero trace on Simons' cone and the natural symmetry condition on the axis, such that
\begin{equation}\label{Phi0bounds}
 |D_{\Gamma_0}^m\Phi_0|\le C_m r^{-1-m},\qquad m=0,1,2,3.
\end{equation}
Consequently
\begin{equation}\label{Phi0res}
 H[F_0+\Phi_0]=O_{\rm ad}(r^{-9}),
\end{equation}
where $O_{\rm ad}(r^{-9})$ means, in particular, that the odd residual satisfies
\[
 |H[F_0+\Phi_0]|\le C g(\theta)r^{-9}
\]
in the cone sector, together with the corresponding first two adapted derivative bounds.  This explicit factor $g$ is what allows comparison with the strict barrier margin below all the way to $t=0$.
\end{lemma}

\begin{proof}
The point is that we do \emph{not} require the particular solution itself to be a pure separated mode at the borderline exponent $\beta=-1$.  We use separation only to construct a positive supersolution.

For a nonnegative angular majorant $p_*(\theta)\ge |E_5(\theta)|$ with $p_*(\theta)=O(g(\theta))$ at the cone, use the separated formula with $\beta=-1$.  Writing $\theta_0=\pi/4$ and
\[
 \mathfrak w_0(\theta)=\frac{\sin^3(2\theta)}{(9g^2+(g')^2)^{3/2}},
\]
the bounded branch is
\begin{equation}\label{qstar}
 q_*(\theta)=\frac19 g(\theta)^{-1/3}
 \int_{\theta_0}^{\theta}
 g(s)^{-2/3}\frac{(9g(s)^2+g'(s)^2)^{3/2}}{\sin^3(2s)}
 \left[\int_s^{\pi/2}p_*(\tau)g(\tau)^{-1}\sin^3(2\tau)\,d\tau\right]ds.
\end{equation}
The sign has been chosen opposite to the particular solution in the standard separated formula.  Because $p_*=O(g)$ and $g(\theta)\sim g_1(\theta-\theta_0)$, the inner integral has a finite positive limit at the cone, while the outer integral is $O((\theta-\theta_0)^{1/3})$.  Hence
\[
 q_*(\theta)=q_*(\theta_0)+O((\theta-\theta_0)^{\mu_*}),\qquad
 0<c_0\le q_*(\theta_0)\le C_0,
\]
for some $\mu_*>0$.  At the axis the inner integral vanishes, and differentiation of \eqref{qstar} gives $q_*'(\pi/2)=0$.  Thus $q_*$ is bounded, positive, satisfies the natural axis condition, and
\[L_0(r^{-1}q_*)\le -c\,r^{-5}p_*(\theta).\]
Notice that $q_*$ is \emph{not} required to vanish on the cone: it is a barrier, not the desired odd solution.  Solve \eqref{Phi0eq} by exhaustion on truncated sectors, with zero data on the cone and fixed smooth data on $r=R_0$.  The maximum principle and $\pm Cr^{-1}q_*$ give a uniform $O(r^{-1})$ bound.  Weighted Schauder estimates give \eqref{Phi0bounds}, and passage to the limit yields the exterior solution.

Since the full linear term has been cancelled exactly, Taylor's formula for the graphical mean-curvature operator gives
\[
 H[F_0+\Phi_0]
 =\int_0^1(1-\tau)H''[F_0+\tau\Phi_0]
       [\Phi_0,\Phi_0] \,d\tau.
\]
The degree rule for two degree $-1$ perturbations gives $r^{-9}$ away from the cone.  Oddness of $F_0$ and $\Phi_0$ implies that the residual is odd and hence vanishes on the cone; the adapted $(s,t)$ estimates give the corresponding weighted form of \eqref{Phi0res}. \qed
\end{proof}

We now control the difference between $F_0+\Phi_0$ and the exact BDG graph.  This is where a near/far barrier construction is needed.

\begin{proposition}[First refined correction of the BDG graph]\label{refinedBDGprop}
There are $\tau_{\rm B}>0$, $R_{\rm B}>1$ such that
\begin{equation}\label{refinedBDG}
 \bar F=F_0+\Phi_0+\mathcal R_{\rm B},
 \qquad
 |D_{\Gamma_0}^m\mathcal R_{\rm B}|
 \le C_m r^{-1-\tau_{\rm B}-m},\quad m=0,1,2,3.
\end{equation}
The expansion extends oddly through the cone.
\end{proposition}

\begin{proof}
The residual in \eqref{Phi0res} is already much smaller than the original $r^{-5}$ error.  It suffices to construct a positive supersolution for the linearized operator with a small algebraic gain, uniformly through the cone.

\smallskip
\noindent\emph{1. Far from the cone.}
For every sufficiently small $\sigma>0$, separation with exponent $\beta=-1-\sigma$ gives a positive profile $q_\sigma$ with
\[
 q_\sigma(\theta)\sim g(\theta)^{-\sigma/3}
 \quad\hbox{as }\theta\downarrow\pi/4.
\]
Set
\[
 \Phi_{\rm far}=A r^{-1-\sigma}q_\sigma(\theta).
\]
Then, for $A$ fixed and $r$ large,
\[H'(F_0)[\Phi_{\rm far}]
 \le-c\,g(\theta)r^{-5-\sigma}
 \qquad\hbox{when }t=F_0\ge C r^{5/3}.\]

\smallskip
\noindent\emph{2. Near the cone: derivation of the model operator.}
The leading near-cone Jacobi operator is
\[
 \partial_{tt}+\partial_{rr}+\frac6r\partial_r+\frac6{r^2}.
\]
For the matching argument we need a quantitative comparison with the exact operator throughout the full near-cone region, which we now derive.

Set $\theta_0=\pi/4$ and $x=\theta-\theta_0$.  Since $g$ is odd at the cone and $g'(\theta_0)=g_1>0$,
\[g=g_1x+O(x^3),\qquad g'=g_1+O(x^2),\qquad
 \sin(2\theta)=1+O(x^2).\]
Using the exact definitions
\[
 t=r^3g(\theta),\qquad
 s=\frac{r^7\sin^3(2\theta)g'(\theta)}{56\sqrt{9g^2+(g')^2}},
 \qquad
 \rho=(uv)^{-3},
\]
we obtain, uniformly for $|g|\ll1$,
\begin{align}
 s&=c_s r^7\bigl(1+O(g^2)\bigr),\label{scone}\\
 W_0:=\sqrt{1+|\nabla F_0|^2}
 &=c_t r^2\bigl(1+O(g^2+r^{-4})\bigr),\label{Wcone}\\
 \rho&=c_\rho r^{-6}\bigl(1+O(g^2)\bigr),\label{rhocone}
\end{align}
for positive constants $c_s,c_t,c_\rho$.  The coordinate formulas also give
\begin{equation}\label{coordcone}
 r_s=c_r r^{-6}\bigl(1+O(g^2)\bigr),\qquad
 r_t=O(gr^{-2}),
\end{equation}
while the corresponding first derivatives of the coefficients gain one additional factor $r^{-1}$ in the adapted metric.

After a harmless constant renormalization of $s$, substitute \eqref{scone}--\eqref{coordcone} into the exact divergence form of the Jacobi operator in $(s,t)$ coordinates.  If $h=h(r,t)$ is an invariant normal perturbation, then
\[\mathcal J_{\Gamma_0}h
 =h_{tt}+h_{rr}+\frac6r h_r+\frac6{r^2}h+\mathcal E_{\rm cone}h,\]
where
\begin{align*}
 |\mathcal E_{\rm cone}h|
 &\le C\bigl(g^2+r^{-4}\bigr)
 \left(|h_{tt}|+|h_{rr}|+r^{-1}|h_r|+r^{-2}|h|\right) \notag\
 &\qquad
 +C|g|r^{-2}|h_{rt}|+Cr^{-3}|g|\,|h_t|.
\end{align*}
The mixed terms in the second line come only from the fact that differentiation in $t$ at fixed $s$ differs from differentiation at fixed $r$ by $r_t\partial_r$; using \eqref{coordcone} and Young's inequality they can also be absorbed into the first line with a factor $C(g^2+r^{-4})$ for the profiles used below.  The potential term follows from the same expansion of the second fundamental form:
\[
 |A_{\Gamma_0}|^2=\frac6{r^2}
 +O\left((g^2+r^{-4})r^{-2}\right).
\]
Thus, in the entire region
\begin{equation}\label{nearregion}
 0<t\le C r^{5/3},
\end{equation}
we have $g=t/r^3=O(r^{-4/3})$ and hence
\begin{equation}\label{nearerrorgain}
 |\mathcal E_{\rm cone}h|
 \le C r^{-8/3}
 \left(|h_{tt}|+|h_{rr}|+r^{-1}|h_r|+r^{-2}|h|\right)
 +Cr^{-10/3}|h_{rt}|.
\end{equation}
This is substantially stronger than a coarse $O(r^{-1})$ remainder.

Take now
\[
 h_{\rm near}=t r^{-3}(t^2+r^2)^{-\beta},
 \qquad
 \beta=\frac12+\frac\sigma6.
\]
The model identity is exact:
\begin{equation}\label{nearidentity}
 \left(\partial_{tt}+\partial_{rr}+\frac6r\partial_r+\frac6{r^2}\right)h_{\rm near}
 =4\beta(\beta-1)t r^{-3}(t^2+r^2)^{-\beta-1}.
\end{equation}
Because $0<\beta<1$, the right-hand side is strictly negative.  Moreover all derivatives entering \eqref{nearerrorgain} carry the same odd factor $t$ at the level relevant for the comparison, and direct differentiation gives
\[
 |\mathcal E_{\rm cone}h_{\rm near}|
 \le o(1)\,
 t r^{-3}(t^2+r^2)^{-\beta-1}
\]
uniformly in \eqref{nearregion}.  Increasing the lower radius therefore preserves a fixed fraction of the negative model margin.

The associated vertical perturbation is
\[
 \Phi_{\rm near}=W_0h_{\rm near}.
\]
For $t\gg r$, \eqref{Wcone} yields
\[\Phi_{\rm near}
 =c_t r^{-1}t^{1-2\beta}
 \left(1+O(g^2)+O(r^2/t^2)+O(r^{-4})\right)
 =c_t r^{-1}t^{-\sigma/3}(1+o(1)).\]
This is exactly the asymptotic size of $\Phi_{\rm far}$ near the cone.  Finally, at $t\asymp r^{5/3}$ the magnitude of the negative model term in \eqref{nearidentity} is comparable to the target $g r^{-5-\sigma}=t r^{-8-\sigma}$ after the choice $\beta=1/2+\sigma/6$.  Hence the near and far strict barriers possess an honest common region around the scale $t\asymp r^{5/3}$.

\smallskip
\noindent\emph{3. Matching at the correct scale.}
The matching is not performed merely at the level of amplitudes.  The separated ODE gives a Frobenius expansion
\[q_\sigma(\theta)=c_\sigma g(\theta)^{-\sigma/3}
 \bigl(1+O(g(\theta)^{\mu_\sigma})\bigr),\]
with the same estimate after one and two adapted angular derivatives, for some $\mu_\sigma>0$.  On the other hand, since $|\nabla F_0|=r^2d_0(\theta)(1+O(r^{-4}))$ and $d_0(\theta)=d_0(\theta_0)+O(g^2)$, the near barrier satisfies for $t\gg r$
\[\Phi_{\rm near}=c_{\rm n}r^{-1}t^{-\sigma/3}
 \left(1+O\!\left(\frac{r^2}{t^2}\right)+O(g^2)\right),\]
again with the corresponding first two adapted derivatives.  Choose the harmless constant in $\Phi_{\rm near}$ so that $c_{\rm n}=c_\sigma$.  In the band
\[
 c r^{5/3}\le t\le C r^{5/3}
\]
we have $g=t/r^3\asymp r^{-4/3}$ and $r^2/t^2\asymp r^{-4/3}$.  Therefore, with
\[
 \delta_*:=\min\left\{\frac43,\frac{4\mu_\sigma}{3}\right\}>0,
\]
\begin{equation}\label{BDGbarriermatch}
 |D_{\Gamma_0}^j(\Phi_{\rm far}-\Phi_{\rm near})|
 \le C r^{-\delta_*}\,\mathfrak d_j[\Phi_{\rm far}],
 \qquad j=0,1,2,
\end{equation}
where $\mathfrak d_j$ denotes the natural order-$j$ adapted derivative size of the barrier (in particular $\mathfrak d_0[\Phi]=|\Phi|$).  This is the precise form needed for cutoff commutators.

Choose a cutoff depending on
\[
 \zeta=\frac{t}{r^{5/3}}.
\]
In the matching band,
\[
 |D_{\Gamma_0}\zeta|\le Cr^{-1},\qquad
 |D_{\Gamma_0}^2\zeta|\le Cr^{-2}
\]
in the adapted metric.  Combining these bounds with \eqref{BDGbarriermatch}, and using the exact divergence form of $H'(F_0)$ in $(s,t)$ coordinates, every commutator contains the extra factor $r^{-\delta_*}$ relative to the corresponding barrier term.  Since the near and far barriers have strict margins, increasing the lower radius absorbs all commutators.  Thus one obtains a global positive supersolution $\overline\Phi$ satisfying
\begin{equation}\label{BDGglobalbar}
 H'(F_0)[\overline\Phi]
 \le-c\,g(\theta)r^{-5-\tau_{\rm B}},
 \qquad
 0<\overline\Phi\le Cr^{-1-\tau_{\rm B}}
\end{equation}
for some $\tau_{\rm B}>0$.

\smallskip
\noindent\emph{4. From the linear barrier to genuine nonlinear barriers.}
The inequality \eqref{BDGglobalbar} is an inequality for the linearized operator at $F_0$; before applying comparison with the exact minimal graph one must verify that the same margin survives in the full nonlinear mean-curvature operator.

Set
\[
 B:=F_0+\Phi_0,
 \qquad
 \Psi:=C\overline\Phi,
\]
where $C>0$ is fixed for the moment.  Taylor expansion around $B$, with $L_0:=H'(F_0)$, gives
\begin{equation}\label{nonlinearBDGexpand}
 H[B\pm\Psi]
 =H[B]\pm L_0\Psi
 \pm\bigl(H'(B)-H'(F_0)\bigr)[\Psi]
 +\mathcal N_B(\pm\Psi),
\end{equation}
where
\[
 \mathcal N_B(\Psi)
 :=H[B+\Psi]-H[B]-H'(B)[\Psi].
\]
We estimate the four terms in \eqref{nonlinearBDGexpand} separately.

First, Lemma~\ref{borderlineBDG} gives the adapted residual estimate
\begin{equation}\label{baseResidual}
 |H[B]|\le C_0 g(\theta)r^{-9}.
\end{equation}
Second, by \eqref{BDGglobalbar},
\begin{equation}\label{linearMargin}
 L_0\Psi\le-c_0 C g(\theta)r^{-5-\tau_{\rm B}}.
\end{equation}

For the change of linearization, notice that
\[
 |D_{\Gamma_0}\Phi_0|=O(r^{-2}),
 \qquad
 |\nabla F_0|\sim r^2.
\]
Thus the relative change of every graph symbol between $F_0$ and $B$ is $O(r^{-4})$.  The degree of $\overline\Phi$ is $-1-\tau_{\rm B}$, so the first variation applied to $\overline\Phi$ has degree $-5-\tau_{\rm B}$.  Consequently,
\begin{equation}\label{changeLinearization}
 \left|\bigl(H'(B)-H'(F_0)\bigr)[\Psi]\right|
 \le C_1 C\, g(\theta)r^{-9-\tau_{\rm B}}.
\end{equation}
The factor $g(\theta)$ in \eqref{changeLinearization} is not inserted formally: $F_0$, $\Phi_0$ and the odd extension of $\overline\Phi$ are odd under the interchange $u\leftrightarrow v$, hence the left-hand side is odd as a function of the cone variable.  The adapted first derivative bounds supplied by the construction of $\overline\Phi$ therefore convert vanishing at $t=0$ into the quantitative factor $g=t/r^3$.

Finally, the graphical second variation satisfies the degree rule
\[
 B_2(\Phi_\beta,\Phi_\gamma)
 =O(r^{\beta+\gamma-7})
\]
with the corresponding adapted derivative estimates.  Since $\overline\Phi$ has degree $-1-\tau_{\rm B}$, Taylor's theorem gives
\begin{equation}\label{quadraticBarrier}
 |\mathcal N_B(\Psi)|
 \le C_2 C^2\, g(\theta)r^{-9-2\tau_{\rm B}}
 + C_3 C^3\, g(\theta)r^{-13-3\tau_{\rm B}}.
\end{equation}
Here again oddness of the full nonlinear remainder yields the factor $g$; equivalently, the same conclusion follows by differentiating the exact $(s,t)$ formula once in the cone variable.  The cubic term in \eqref{quadraticBarrier} is displayed only to make the dependence on $C$ explicit; it is lower order than the quadratic term for fixed $C$ and large $r$.

Combining \eqref{baseResidual}--\eqref{quadraticBarrier}, we find
\begin{align*}
 H[B+\Psi]
 &\le -c_0 C g r^{-5-\tau_{\rm B}}
 +C_0 g r^{-9}
 +C_1 C g r^{-9-\tau_{\rm B}}
 +C_2 C^2 g r^{-9-2\tau_{\rm B}}
 +C_3 C^3 g r^{-13-3\tau_{\rm B}},\\
 H[B-\Psi]
 &\ge c_0 C g r^{-5-\tau_{\rm B}}
 -C_0 g r^{-9}
 -C_1 C g r^{-9-\tau_{\rm B}}
 -C_2 C^2 g r^{-9-2\tau_{\rm B}}
 -C_3 C^3 g r^{-13-3\tau_{\rm B}}.
\end{align*}
Hence, for every fixed $C$, there is $R(C)$ such that
\begin{equation}\label{strictNonlinearBDG}
 H[B+C\overline\Phi]<0,
 \qquad
 H[B-C\overline\Phi]>0
 \qquad\hbox{in }\{r>R(C),\ v>u\}.
\end{equation}
The dependence of $R(C)$ on $C$ causes no difficulty in the comparison argument.  Indeed, if the inner comparison radius is enlarged to $R$, the rough estimate $\bar F-F_0=O(R^{-\sigma_0})$, together with $\Phi_0=O(R^{-1})$ and the linear vanishing of $\overline\Phi$ at the cone, shows that the constant needed to dominate the boundary discrepancy grows at most algebraically,
\[
 C_R\le C R^{1+\tau_{\rm B}-\sigma_0}.
\]
At $r=R$ the ratio between the quadratic error in \eqref{quadraticBarrier} and the linear margin in \eqref{linearMargin} is then bounded by
\[
 C_R R^{-4-\tau_{\rm B}}
 \le C R^{-3-\sigma_0},
\]
and therefore tends to zero.  Thus $R$ and $C_R$ can be chosen simultaneously so that \eqref{strictNonlinearBDG} holds and the boundary discrepancy at $r=R$ is dominated.

\smallskip
\noindent\emph{5. Comparison with the exact BDG graph.}
Choose $R$ and $C=C_R$ as above, and define
\[
 F^+:=B+C\overline\Phi,
 \qquad
 F^-:=B-C\overline\Phi.
\]
On the cone $t=0$, all three functions $\bar F$, $F^+$ and $F^-$ vanish.  On the fixed inner boundary $r=R$, the choice of $C$ gives
\[
 F^-\le \bar F\le F^+.
\]
The quotient by the cone variable is bounded up to the corner because all functions are smooth and odd, while the near barrier has a nonzero first $t$-derivative there; hence the boundary domination extends continuously to the intersection of $r=R$ with the cone.

No artificial outer boundary condition is needed.  Indeed,
\[
 \bar F-B=O(r^{-\sigma_0}),
 \qquad
 \overline\Phi=O(r^{-1-\tau_{\rm B}}),
\]
so
\[
 \bar F-F^+\longrightarrow0,
 \qquad
 F^- -\bar F\longrightarrow0
 \qquad\hbox{as }r\to\infty.
\]
If, for instance, $\bar F-F^+$ were positive somewhere, it would possess a positive maximum at a finite point of the exterior sector.  At that point the quasilinear comparison principle for the minimal-surface operator contradicts $H[\bar F]=0$ and $H[F^+]<0$.  The same argument with $F^-$ gives
\[F^-\le \bar F\le F^+
 \qquad\hbox{for }r\ge R,\ v>u.\]
Consequently
\[
 |\bar F-F_0-\Phi_0|
 \le C\overline\Phi
 \le Cr^{-1-\tau_{\rm B}}.
\]
Writing the difference as a normal graph over $\Gamma_0$ and applying the intrinsic weighted Schauder estimates to the equation satisfied by the difference gives the derivative bounds in \eqref{refinedBDG}. \qed
\end{proof}

The roles of the two constructions are therefore distinct.  The $O(r^{-1})$ correction $\Phi_0$ cancels the leading residual $H[F_0]$, while the near/far barriers control the smaller remainder after that cancellation.  Their common matching scale is $t\asymp r^{5/3}$, with near-barrier exponent $\beta=\frac12+\frac\sigma6$.

\subsubsection{Discrete Taylor symbols around the exact BDG graph}

The refined expansion of $\bar F$ has an important consequence which we now formulate without using formal limits.  Write
\[
 U_1^{(0)}(r,\theta):=r^2a_1(\theta),
\]
where $a_1$ is the even angular profile obtained by solving the leading order-$\eps$ equation on the homogeneous model.  Let $B_{2,0}$, $B_{3,0}$, $S_{1,0}$ and $S_{2,0}$ denote the homogeneous Taylor symbols obtained by freezing the coefficients at $F_0$.  Their degrees are those recorded in the degree rules of the previous section.

\begin{lemma}[Exact second-order homogeneous symbol]\label{F2symbol}
Define
\begin{equation}\label{P3def}
 r^{-3}P_3(\theta)
 :=
 B_{2,0}(U_1^{(0)},U_1^{(0)})
 -S_{1,0}(U_1^{(0)}).
\end{equation}
Then $P_3$ is smooth and odd under $\theta\mapsto\frac\pi2-\theta$.  Moreover the exact order-$\eps^2$ forcing around the true BDG graph satisfies
\[F_2
 =
 r^{-3}P_3(\theta)+\mathcal R_2,\]
where, in every fixed angular sub-sector,
\[
 |D^j\mathcal R_2|\le C_j r^{-7-j+\eta}
 \qquad (j=0,1,2)
\]
for every fixed small $\eta>0$, while in the cone region $0\le t\le c r^{3-\eta}$ one has
\begin{align}\label{R2cone}
 |\mathcal R_2|&\le C\,g\,r^{-7},\\
 |\partial_t\mathcal R_2|&\le C r^{-10},\qquad
 |\mathcal D_r\mathcal R_2|\le C\,g\,r^{-8},
\end{align}
with the analogous second derivative estimates.  Here $\mathcal D_r$ denotes differentiation in the $r$ direction at fixed $t$.
\end{lemma}

\begin{proof}
The definition \eqref{P3def} is homogeneous, so no limiting argument is involved.  Since $U_1^{(0)}$ is even, the identities
\[
 H(-F)=-H(F),\qquad S(-F)=S(F)
\]
show that the quadratic mean-curvature symbol and the first translating-source symbol are both odd at the odd background $F_0$.  Hence $P_3$ is odd and, being a homogeneous angular Taylor coefficient, is smooth.

We next compare the exact coefficients at $\bar F$ with their homogeneous values at $F_0$.  By Proposition~\ref{refinedBDGprop},
\[
 \bar F=F_0+\Phi_0+O_{\rm ad}(r^{-1-\tau_{\rm B}}),
 \qquad
 |D\Phi_0|=O(r^{-2}).
\]
Since $|\nabla F_0|\asymp r^2$, the relative perturbation of the slope is $O(r^{-4})$.  The same four-power gain applies to all first, second and third graph symbols.  In addition, \eqref{U1sharp} gives
\[
 U_1=U_1^{(0)}+O_{\rm ad}(r^{-2-\tau_1}).
\]
Inserting either a four-power coefficient correction or the degree $-2-\tau_1$ correction to $U_1$ into a degree $-3$ second-order symbol yields degree at most $-7+\eta$ away from the cone.  This proves the first remainder estimate.

Near the cone the parity is useful.  The exact forcing $F_2$ is odd, hence the difference $\mathcal R_2$ is odd as well.  Smoothness in the adapted variables therefore gives one factor $g=t/r^3$.  Combining this factor with the four-power coefficient gain yields
\[
 \mathcal R_2=g\,O(r^{-7}).
\]
Differentiating at fixed $t$ costs one radial power, whereas one $t$-derivative removes the factor $g$ and costs $r^{-3}$.  This gives \eqref{R2cone} and the corresponding second derivative estimates. \qed
\end{proof}

The leading equation
\[
 L_0(rq_2(\theta))=-r^{-3}P_3(\theta)
\]
has the odd solution discussed earlier.  Near the cone,
\[
 q_2(\theta)\sim c\,\sgn(g)|g|^{2/3}.
\]
The exact solution cannot be used through $g=0$ because its $t$-derivatives become singular; this is precisely the singularity resolved by the layer block.

\begin{lemma}[Coefficient-level matching of the $|T|^{-1/3}$ term]\label{matchingcoefficient}
In the additive large-$|p|$ expansion of the smooth inner profile, the coefficient of $|p|^{-1/3}$ is $1/3$ on both sides of the cone.  The singular part of the outer order-$\eps^3$ correction has the same coefficient.  Consequently subtracting the $p^{-1/3}$ contribution of the single layer removes the entire $g^{-1/3}$ singularity of the raw third-order outer source.
\end{lemma}

\begin{proof}
The first assertion was computed directly in the proof of Lemma~\ref{innerprofilefinal}.  We check that the outer coefficient is the same, rather than invoking matching by assertion.  Near the cone the recentered normal operator has principal part
\[
 \mathscr P_\eps
 =\partial_{TT}+\partial_{rr}+\frac6r\partial_r+\frac6{r^2}-\eps\partial_T.
\]
The leading odd outer profile is
\[
 h_0=\frac{\operatorname{sgn}(T)|T|^{2/3}}{r^3}.
\]
Its translator drift has the universal singular part
\[
 -\eps\partial_T h_0
 =-\frac23\eps\frac{|T|^{-1/3}}{r^3}.
\]
The next algebraic correction has the form
\[
 h_1=C\eps\frac{|T|^{-1/3}}r.
\]
For the radial part of $\mathscr P_0$ one has the exact identity
\[
 \left(\partial_{rr}+\frac6r\partial_r+\frac6{r^2}\right)
 \left(r^{-1}|T|^{-1/3}\right)
 =2r^{-3}|T|^{-1/3}.
\]
Thus cancellation of the unique $|T|^{-1/3}r^{-3}$ singular coefficient forces $2C=2/3$, hence $C=1/3$.  Terms containing two $T$-derivatives lie at the next level of the common large-$|p|$ expansion and are precisely those retained by the full inner ODE; they do not alter this coefficient.  Equivalently, substitution into the exact self-similar ODE gives the same identity $2C-2/3=0$.

The separated outer Jacobi equation has only one Frobenius branch with this $g^{-1/3}$ singularity.  Therefore the difference between the raw third-order outer coefficient and the $p^{-1/3}$ coefficient furnished by the smooth layer has no negative fractional power at the cone.  This is the coefficient-level matching needed below.
\end{proof}

\begin{lemma}[Regular third-order symbol after subtraction of the layer]\label{F3symbol}
Let $\mathcal U_{\rm layer,\eps}$ be the single transition layer constructed from the Kummer--Stokes profile $Q$.  Expand it for $|p|\to\infty$ through the $p^{-1/3}$ term.  After subtracting from the raw order-$\eps^3$ Taylor source the contribution generated by this $p^{-1/3}$ matching term, the remaining third-order source has the form
\[F_{3,\rm reg}
 =
 r^{-4}P_{4,3}(\theta)+\mathcal R_3,\]
where $P_{4,3}$ is smooth and even, and
\[
 \mathcal R_3=O(r^{-8+\eta})
\]
away from the cone, with the corresponding adapted weighted bounds near $g=0$.
\end{lemma}

\begin{proof}
At the homogeneous level, before the inner matching contribution is removed, the order-$\eps^3$ Taylor coefficient is
\begin{align*}
 \mathfrak F_{3,0}
 ={}&
 2B_{2,0}(U_1^{(0)},U_2^{(0)})
 +B_{3,0}(U_1^{(0)},U_1^{(0)},U_1^{(0)})\\
 &-S_{1,0}(U_2^{(0)})
 -S_{2,0}(U_1^{(0)},U_1^{(0)}),
\end{align*}
and every displayed term has radial degree $-4$.  Exact parity shows that $\mathfrak F_{3,0}$ is even.  The odd profile $U_2^{(0)}$ has fractional behavior $|g|^{2/3}$, so $\mathfrak F_{3,0}$ contains a non-smooth even angular piece produced by differentiating that profile.  The large-$|p|$ expansion of the smooth inner solution is
\[
 Q(p)
 =
 \sgn(p)|p|^{2/3}
 +B_\pm |p|^{1/6}
 +c_\pm |p|^{-1/3}
 +O(|p|^{-5/6}).
\]
By Lemma~\ref{matchingcoefficient}, substitution of the $p^{-1/3}$ term into the outer operator produces the complete non-smooth $g^{-1/3}$ part of $\mathfrak F_{3,0}$, with the same coefficient $1/3$ on both sides of the cone.  We therefore define the matching contribution $\mathfrak F_{3,\rm match}$ by this $p^{-1/3}$ coefficient and set
\[
 r^{-4}P_{4,3}(\theta)
 :=
 \mathfrak F_{3,0}-\mathfrak F_{3,\rm match}.
\]
Because the full inner profile is smooth across $p=0$, this subtraction removes the only fractional singularity at this order.  The remaining angular coefficient is smooth.  It is even by the parity identities above.

Finally the refined BDG expansion again supplies a four-power gain between exact and homogeneous graph symbols, while the exactified layer removes the lower-order inner mismatch.  Hence the exact regular source differs from its homogeneous symbol by degree at most $-8+\eta$ away from the cone; the adapted estimates follow from the smooth recentered layer equation and Proposition~\ref{refinedBDGprop}. \qed
\end{proof}

\begin{corollary}[Discrete symbols around the exact BDG graph]\label{discreteBDGsymbols}
The first three relevant coefficients of the exact translator expansion can be organized as
\begin{align}
 S(\bar F)&=r^{-2}s_2(\theta)+O_{\rm ad}(r^{-6-\tau_{\rm B}}),\label{Ssharp}\\
 U_1&=r^2a_1(\theta)+O_{\rm ad}(r^{-2-\tau_1}),\label{U1sharp}\\
 F_2&=r^{-3}P_3^{\rm odd}(\theta)+\mathcal R_2,\label{F2sharp}\\
 F_{3,\rm reg}&=r^{-4}P_{4,3}^{\rm even}(\theta)+\mathcal R_3,\label{F3sharp}
\end{align}
where the remainders satisfy Lemmas~\ref{F2symbol} and~\ref{F3symbol}.
\end{corollary}

\subsubsection{Weighted Schauder estimates on the true graph}

We shall also need a convenient way to recover derivatives from weighted $L^\infty$ control.  For a function $g$ on $\Gamma$ define schematically
\[
\|g\|_{0,\mu}:=\sup_\Gamma (1+r)^\mu|g|,
\]
and let $\|g\|_{0,\alpha;\mu}$ denote the corresponding weighted H\"older norm on intrinsic balls of radius comparable with $r$.

\begin{lemma}[Weighted Schauder estimate]\label{weightedSchauder}
Let $\mu\ge2$ and suppose
\[
\cJ_\Gamma h=g,
\qquad
\|g\|_{0,\alpha;\mu}+\|h\|_{0,\mu-2}<\infty.
\]
Then
\begin{equation}\label{weightedSchauderEq}
\|D_\Gamma^2h\|_{0,\alpha;\mu}
+\|D_\Gamma h\|_{0,\alpha;\mu-1}
+\|h\|_{0,\alpha;\mu-2}
\le C\Bigl(\|g\|_{0,\alpha;\mu}+\|h\|_{0,\mu-2}\Bigr).
\end{equation}
\end{lemma}

\begin{proof}
This is the scale-invariant interior Schauder estimate in the tangent-plane charts of the BDG graph.  The complete rescaling argument, including uniform control of the metric coefficients and the weighted H\"older seminorms, is given in Appendix~\ref{app:weighted}, Proposition~\ref{app-weighted-schauder}.
\end{proof}

\subsection{Weighted spaces and the fast Jacobi inverse}\label{sec:weighted-fast-inverse}

The slow modes $r^{-2}$, $r^{-3}$ and $r^{-4}$ require the explicit separation analysis developed later.  Once those leading terms are removed, all remaining corrections belong to a faster class.  We record here a self-contained inversion statement for that class.

\subsubsection{Weighted norms}

Let $r(y)$ denote the radial variable of the projection of $y\in\Gamma$ to $\R^8$.  For $\mu\in\R$ define
\[
 \|f\|_{0,\mu}
 :=\sup_{\Gamma}(1+r)^\mu|f|.
\]
For $\alpha\in(0,1)$ define the local weighted seminorm
\[
 [f]_{\alpha,\mu}
 :=\sup_{R\ge1}\sup_{\substack{y_1,y_2\in\Gamma\\
 R\le r(y_i)\le2R\\ d_\Gamma(y_1,y_2)\le R/4}}
 R^{\mu+\alpha}
 \frac{|f(y_1)-f(y_2)|}{d_\Gamma(y_1,y_2)^\alpha}.
\]
We put
\[
 \|f\|_{\alpha,\mu}=\|f\|_{0,\mu}+[f]_{\alpha,\mu}.
\]
For $h$ we use
\[
 \|h\|_{2,\alpha;\mu}
 :=\sum_{j=0}^2\|(1+r)^jD_\Gamma^jh\|_{\alpha,\mu}.
\]
Thus $\|h\|_{2,\alpha;\mu}<\infty$ means, schematically,
\[
 D_\Gamma^jh=O(r^{-\mu-j}).
\]

\subsubsection{Uniform charts at infinity}

Fix $p\in\Gamma$ with $r(p)=R\gg1$.  By the curvature estimate, for a universal $\vartheta>0$ the set
\[
 \Gamma\cap B(p,2\vartheta R)
\]
is a graph over $T_p\Gamma$.  If $y=Rz$ denotes the rescaled tangent coordinate, then the induced metric coefficients satisfy
\[
 \lambda I\le(g^{ij}_R(z))\le\Lambda I
\]
on $B_{2\vartheta}$, and all first derivatives of the coefficients are bounded independently of $R$ and $p$.  Moreover
\[
 R^2|A_\Gamma(Rz)|^2
\]
is uniformly bounded.  These facts follow directly from
\[
 |D_\Gamma^mA_\Gamma|\le C_mR^{-1-m}.
\]

\begin{lemma}[Weighted interior Schauder estimate]\label{weighted-schauder}
Let $\mu\in\R$.  If
\[
 \cJ_\Gamma h=f
\]
and
\[
 \|f\|_{\alpha,\mu+2}+\|h\|_{0,\mu}<\infty,
\]
then
\begin{equation}\label{weighted-schauder-est}
 \|h\|_{2,\alpha;\mu}
 \le C\left(
 \|f\|_{\alpha,\mu+2}+\|h\|_{0,\mu}
 \right).
\end{equation}
\end{lemma}

\begin{proof}
See Appendix~\ref{app:weighted}, Proposition~\ref{app-weighted-schauder}.  The formulation here is the same estimate with the weight indexed by the decay of $h$ rather than by the decay of the right-hand side.
\end{proof}

\subsubsection{A strict algebraic barrier}

The fast inverse rests on a positive supersolution at the scale
$r^{-2-\nu}$.  It is useful to construct this supersolution from the
separated Jacobi equation rather than to guess a purely radial weight.  This
also makes the sign completely transparent.

Fix $0<\nu<1$ and set $\theta_0=\pi/4$.  In the separated formula
\eqref{sepcalc} take $\beta=-\nu$ and prescribe the constant angular source
$p\equiv-9$.  Define
\begin{align}\label{qnu-barrier}
q_\nu(\theta)
={}&g(\theta)^{-\nu/3}
\int_{\theta_0}^{\theta}
\frac{g(s)^{-2/3}D(s)^{3/2}}{\sin^3(2s)}\notag\\
&\qquad\times
\left(\int_s^{\pi/2}
 g(\tau)^{(\nu-4)/3}\sin^3(2\tau)\,d\tau\right)ds .
\end{align}
Every factor in the integrand is positive.  Formula \eqref{sepcalc} therefore
gives the exact leading identity
\begin{equation}\label{qnu-source}
 -L_0\big(r^{-\nu}q_\nu(\theta)\big)=9r^{-4-\nu}.
\end{equation}

We record the endpoint behavior because it is what turns this vertical
supersolution into a normal weight comparable with $r^{-2-\nu}$.  Near the
cone, $g(\theta)=g_1x+O(x^3)$ with $x=\theta-\theta_0$.  The inner integral in
\eqref{qnu-barrier} is then of order $x^{-(1-\nu)/3}$, while the outer
integrand is of order $x^{-1+\nu/3}$.  Consequently the outer integral is of
order $x^{\nu/3}$ and the prefactor $g^{-\nu/3}$ exactly cancels this power.
Expanding one order further, using the odd expansion of $g$ at the cone,
shows that
\[
 q_\nu(\theta)=q_{\nu,0}+q_{\nu,2}x^2+O(x^4),
 \qquad q_{\nu,0}>0.
\]
At the axis the even expansion of $g$ gives $q_\nu'(\pi/2)=0$.  Hence $q_\nu$
extends evenly through the cone and is smooth and positive on the closed
angular interval.  In particular
\begin{equation}\label{qnu-bounds}
 0<c_\nu\le q_\nu(\theta)\le C_\nu.
\end{equation}

Let
\[
 \phi_\nu^0(r,\theta)=r^{-\nu}q_\nu(\theta),
 \qquad
 h_\nu^0=\frac{\phi_\nu^0}{W_0}.
\]
Since $W_0=r^2D^{1/2}(1+O(r^{-4}))$, \eqref{qnu-bounds} implies
\begin{equation}\label{hnu-size}
 c r^{-2-\nu}\le h_\nu^0\le C r^{-2-\nu}.
\end{equation}
The difference between the full vertical linearization $H'(F_0)$ and its
homogeneous part $L_0$ is four powers lower.  Thus \eqref{qnu-source} gives,
for $r$ sufficiently large,
\begin{equation}\label{model-fastbarrier}
 \cJ_{\Gamma_0}h_\nu^0
 =H'(F_0)[\phi_\nu^0]
 \le -c r^{-4-\nu}.
\end{equation}
The model--exact comparison of Proposition~\ref{transferprop} preserves the
strict sign on the exact BDG graph.  After pulling $h_\nu^0$ to $\Gamma$ on
the end and completing it across a fixed compact set, we obtain a smooth
positive function $b_\nu$ such that
\begin{equation}\label{fastbarrier}
 \cJ_\Gamma b_\nu\le-c(1+r)^{-4-\nu},
 \qquad
 b_\nu\asymp(1+r)^{-2-\nu}.
\end{equation}
For completeness, the compact completion can be made without any additional
asymptotic input.  Fix a large $R$ for which \eqref{model-fastbarrier} has
already transferred to $\Gamma$.  On the bounded region $\{r<2R\}\cap\Gamma$
solve a Dirichlet problem for $-\cJ_\Gamma$ with a strictly positive smooth
right-hand side and boundary values equal to the exterior barrier.  The
positive Jacobi field $Z=\langle\nu,e_9\rangle$ gives the maximum principle,
so the solution is positive; standard elliptic regularity gives smoothness.
A cutoff in the overlap, followed if necessary by multiplying the compact
right-hand side by a fixed constant, produces \eqref{fastbarrier} globally.

The restriction $0<\nu<1$ is visible directly in \eqref{qnu-barrier}: it is
precisely the range for which the cone singularities in the two nested
integrals balance to a finite positive limit.  The slower forcing regimes at
the endpoint weights must therefore be extracted explicitly rather than
hidden in the fast inverse.

\begin{proposition}[Fast Jacobi inverse]\label{fast-inverse}
Let $0<\nu<1$ and let $f$ be $O(4)\times O(4)$-invariant with
\[
 \|f\|_{\alpha,4+\nu}<\infty.
\]
Then there exists a unique solution $h$ of
\[
 \cJ_\Gamma h=f
\]
satisfying
\[
 \|h\|_{2,\alpha;2+\nu}<\infty.
\]
Moreover
\begin{equation}\label{fast-inverse-est}
 \|h\|_{2,\alpha;2+\nu}
 \le C\|f\|_{\alpha,4+\nu}.
\end{equation}
\end{proposition}

\begin{proof}
The proof is by exhaustion, using the strict algebraic barrier above, the positive Jacobi field and the weighted Schauder estimate.  Since this inverse is used repeatedly, the full argument---including uniqueness and the passage from bounded domains to the complete graph---is recorded in Appendix~\ref{app:weighted}, Proposition~\ref{app-fast-inverse}.
\end{proof}

\begin{remark}[Why this is enough]
The forcing $r^{-2}$ and the resonant components $r^{-3},r^{-4}$ are deliberately excluded from Proposition~\ref{fast-inverse}.  They are not errors to be hidden in a black-box inverse.  They determine the geometry of the approximation and are solved explicitly in Section~\ref{sec:slowmatch}.  The fast inverse is used only after these slow pieces have been removed.
\end{remark}

\section{Slow forcing modes and completion of the matching}\label{sec:slowmatch}

At this stage the formal expansion has produced three different radial scales.  It is useful to stop and explain why they cannot be treated by a single weighted inverse.  The Jacobi operator of the BDG graph lowers the radial degree by four.  Hence a right-hand side of size $r^{-2}$ naturally produces a correction of size $r^2$, a right-hand side of size $r^{-3}$ produces a correction of size $r$, and a right-hand side of size $r^{-4}$ lies at the borderline degree zero.  The last case is resonant and may produce a logarithm.  These are exactly the three right-hand sides
\[
r^{-2}p_2(\theta),\qquad r^{-3}p_3(\theta),\qquad r^{-4}p_4(\theta).
\]
Their parities are equally important.  The first source is even because the translating forcing is invariant under interchange of the two $\mathbb R^4$ factors.  The second source is odd and is the first place where the fractional cone behavior $g^{2/3}$ appears.  The third source is again even and contains the first genuine resonance of the homogeneous Jacobi operator.  A fast inverse such as Proposition~\ref{app-fast-inverse} is deliberately not used at these scales: the right-hand sides decay too slowly, and the leading term has to be extracted explicitly before the remaining error enters the fast theory.

The separated-variable calculation below is therefore not an optional asymptotic refinement.  It is the mechanism that identifies the correct unknown at each order and determines which part of the solution must be incorporated into the global ansatz.  For completeness, Appendix~\ref{app:slow-jacobi} gives a self-contained derivation of the model Jacobi formula, its integral representation, and the regularization of the even and odd slow modes.  This is closely related to the BDG Jacobi analysis developed in \cite{DPKW,DPPW}, but all estimates used here are proved in the present paper.

Recall the homogeneous operator
\[
L_0(r^\beta q(\theta))
=r^{\beta-4}\,\mathcal A_\beta q(\theta),
\]
where
\[
\mathcal A_\beta q
=\frac{9g^{(\beta+4)/3}}{\sin^3(2\theta)}
\left[
 w_0g^{2/3}\bigl(g^{-\beta/3}q\bigr)'
\right]',
\qquad
w_0=\frac{\sin^3(2\theta)}{(9g^2+(g')^2)^{3/2}}.
\]
Hence $g^{\beta/3}$ spans the homogeneous kernel and a particular solution of
\[
L_0(r^\beta q)=r^{\beta-4}p(\theta)
\]
is
\begin{equation}\label{qformula2}
q(\theta)=g(\theta)^{\beta/3}
\left[
A-\frac19\int_{\pi/4}^{\theta}
 g^{-2/3}(9g^2+(g')^2)^{3/2}\frac{ds}{\sin^3(2s)}
\int_s^{\pi/2}p(\tau)g(\tau)^{-(\beta+4)/3}\sin^3(2\tau)\,d\tau
\right].
\end{equation}

\subsubsection{The $r^{-2}$ even forcing}

Here $\beta=2$.  The translator forcing has the expansion
\[
\frac1{\sqrt{1+|\grad F_0|^2}}
=\frac{p_2(\theta)}{r^2}+O(r^{-6}),
\]
with $p_2$ even under $\theta\mapsto\pi/2-\theta$ and positive in the sector.  Formula \eqref{qformula2} gives
\[
\Phi_1^{\mathrm{out}}(r,\theta)=r^2a_1(\theta),
\qquad
L_0\Phi_1^{\mathrm{out}}=\frac{p_2(\theta)}{r^2}.
\]
Near the cone $x=\theta-\pi/4$, we have
\[
g(\theta)=g_1x+O(x^3).
\]
The integral formula shows that $a_1$ extends smoothly and evenly through $x=0$; in particular
\[
a_1(\pi/4)\ne0
\]
in general.  This is the analytic reason that
\[
F_0+\eps\Phi_1
\]
does \emph{not} vanish on the Simons cone.  The displacement there is
\[
\eps\Phi_1\sim \eps r^2a_1(\pi/4).
\]

A useful monotonicity identity also follows from \eqref{qformula2}:
\begin{equation}\label{monotonea1}
g^{5/3}\left(g^{-2/3}a_1\right)'<0
\end{equation}
when $p_2>0$.  Later, \eqref{monotonea1} gives the non-vanishing of the coefficient multiplying the recentered transversal derivative.

\subsubsection{The $r^{-3}$ odd forcing}

Now $\beta=1$ and $p_3$ is odd across the cone.  The separated solution is
\[
\Phi_2^{\mathrm{out}}=rq_2(\theta),
\]
where the integral formula gives
\[
q_2(\theta)
=A\,g(\theta)^{2/3}+O(g^{5/3})
\qquad\hbox{as }g\downarrow0.
\]
Therefore
\[\Phi_2^{\mathrm{out}}
\sim A r g^{2/3}
=A\frac{t^{2/3}}r,
\qquad t=F_0=r^3g.\]

The raw derivative $\partial_t\Phi_2^{\mathrm{out}}$ behaves like $t^{-1/3}/r$ and is singular at $t=0$.  Nevertheless, in the \emph{elliptic odd problem} this did not cause a contradiction.  The corresponding normal displacement
\[
h_2=\frac{\Phi_2^{\mathrm{out}}}{\sqrt{1+|\grad F_0|^2}}
\sim \frac{t^{2/3}}{r^3}
\]
satisfies the intrinsic estimate
\begin{equation}\label{oddgradestimate}
|D_{\Gamma_0}h_2|^2=O(r^{-4-\mu_0})
\end{equation}
for some $\mu_0>0$.  Indeed, the singular angular derivative is multiplied by the inverse metric coefficient
\[
(r^2+r^6(g')^2)^{-1},
\]
which compensates the apparent blow-up.  In the elliptic De Giorgi construction the cone is moreover a symmetry boundary: one solved in $v>u$ and extended oddly.

For the translator the same fractional profile becomes genuinely problematic, not because \eqref{oddgradestimate} fails, but because the translating drift destroys the odd symmetry and differentiates directly in the recentered transversal variable.  Thus the cone is no longer a boundary at which the problem can be stopped.

\paragraph{Proof of the intrinsic gradient estimate.}
We include the calculation because it is the cleanest illustration of the difference between a singular coordinate derivative and a regular geometric derivative.  For an invariant normal function $h=h(r,\theta)$, the inverse metric in the $(r,\theta)$ variables gives
\begin{equation}\label{intrinsic-grad-form}
 |D_{\Gamma_0}h|^2
 =\frac{|h_r|^2}{1+9g^2r^4}
 +\frac{|h_\theta|^2}{r^2+r^6(g')^2}.
\end{equation}
For
\[
 h_2=r^{-1}q_2(\theta),
 \qquad q_2\sim g^{2/3},
\]
we have
\[
 h_{2,r}=O(r^{-2}g^{2/3}),
 \qquad
 h_{2,\theta}=O(r^{-1}g^{-1/3}).
\]
Away from the cone, \eqref{intrinsic-grad-form} immediately gives $|D_{\Gamma_0}h_2|=O(r^{-2})$.  In the transition region where $g\lesssim r^{-2}$, the second term is estimated using $g'\asymp1$:
\[
 \frac{|h_{2,\theta}|^2}{r^2+r^6(g')^2}
 \le Cr^{-8}g^{-2/3}.
\]
The outer formula is used only where $r^2g\gtrsim1$; hence $g^{-2/3}\lesssim r^{4/3}$ and
\[
 \frac{|h_{2,\theta}|^2}{r^2+r^6(g')^2}
 \le Cr^{-20/3}.
\]
The cutoff region contributes the same or a smaller power because differentiation of the cutoff costs $O(r^2)$ in $\theta$ while $q_2=O(r^{-4/3})$ there.  Thus for some $\mu_0>0$,
\[
 |D_{\Gamma_0}h_2|^2\le Cr^{-4-\mu_0}.
\]
The precise value of $\mu_0$ is irrelevant; the point is the strict gain over $r^{-4}$.

\subsubsection{The $r^{-4}$ even forcing}

Here $\beta=0$.  If $p_4(\pi/4)=0$, formula \eqref{qformula2} gives a bounded angular correction, although generally with a fractional derivative at the cone.  The leading singularity is
\[
q_4(\theta)\sim Cg(\theta)^{1/3}.
\]
It can be cut off at the inner scale because the induced error is already lower order.

If $p_4(\pi/4)\ne0$, there is a resonant logarithmic term.  A direct calculation gives
\[L_0(\log r)=\frac{c_0+b(\theta)}{r^4},\]
where $c_0\ne0$ is constant and $b$ is smooth, even, and vanishes at the cone.  Thus one first chooses a multiple of $\log r$ to remove the constant component and then solves the remaining angular problem by \eqref{qformula2}.

This explains the hierarchy
\[
\Phi_1\sim r^2,
\qquad
\Phi_2\sim r g^{2/3},
\qquad
\Phi_3\sim \log r+O(1),
\]
which is the leading inner correction.

\subsection{The inner-profile lemma and its Stokes parameter}

The inner ODE is
\begin{equation}\label{innerODErig}
4p^2Q''+\left(\frac23p-1\right)Q'+\frac49Q=0.
\end{equation}
A smooth connecting solution must also account for the second algebraic mode $|p|^{1/6}$ at infinity.

\subsubsection{Negative side: the smooth branch is unique up to scale}

For $p=-x<0$ put
\[
z=\frac1{4x},\qquad Q(-x)=x^{2/3}Y(z).
\]
Then \eqref{innerODErig} becomes
\[zY''+\left(\frac12-z\right)Y'+\frac23Y=0,\]
which is Kummer's equation with $a=-2/3$, $b=1/2$.  As $p\to0^-$, $z\to+\infty$.  The $M$-branch grows exponentially, whereas
\[
Y(z)=U\left(-\frac23,\frac12,z\right)
\]
has the algebraic behavior needed for smoothness at $p=0$.  Thus the smooth negative-side branch is unique up to multiplication by a constant.

As $p\to-\infty$, equivalently $z\to0^+$, the connection formula for $U$ gives
\[Q(p)=A_-|p|^{2/3}+B_-|p|^{1/6}+C_-|p|^{-1/3}+O(|p|^{-5/6}),\]
with $B_-\ne0$.  We fix the overall scale by imposing
\[
A_-=-1.
\]

\subsubsection{The regular positive-side branch and the flat Stokes mode}

Because $p=0$ is an irregular singular point, a formal jet by itself is not an existence argument.  The regular continuation can instead be written explicitly.

Put $z=(4p)^{-1}$ for $p>0$.  Besides the flat branch below, define
\begin{equation}\label{regularpositive}
Q_{\rm reg}(p)
:=C_{\rm reg}\,p^{2/3}e^{-1/(4p)}
M\left(\frac76,\frac12,\frac1{4p}\right),
\qquad p>0,
\end{equation}
where $M$ is Kummer's first solution and $C_{\rm reg}$ is chosen so that
$Q_{\rm reg}(0^+)=Q(0^-)$.  The large-$z$ expansion of $M(7/6,1/2,z)$ contains the term
$e^zz^{2/3}$; after multiplication by $p^{2/3}e^{-z}$ this gives a nonzero constant and then a full power series in $p$.  Hence $Q_{\rm reg}$ extends smoothly to $p=0$.

If
\[
Q(p)=\sum_{j\ge0}a_jp^j
\]
is a regular solution at $p=0$, the equation determines its entire Taylor jet from $a_0$; for instance
\[
a_1=\frac49a_0.
\]
Thus the normalization in \eqref{regularpositive} gives exactly the same jet as the negative-side smooth branch.  There is however a second solution on $p>0$ which is flat at zero.  An explicit choice is
\[Q_{\mathrm{flat}}(p)
=p^{2/3}e^{-1/(4p)}
U\left(\frac76,\frac12,\frac1{4p}\right),
\qquad p>0.\]
Since $U(7/6,1/2,z)\sim z^{-7/6}$ as $z\to+\infty$,
\[
Q_{\mathrm{flat}}(p)
\sim C p^{11/6}e^{-1/(4p)}
\qquad(p\downarrow0),
\]
so every derivative vanishes at zero.

On the other hand, as $p\to+\infty$, the connection formula at $z=0$ yields
\[Q_{\mathrm{flat}}(p)
=A_{\mathrm{flat}}p^{2/3}
+B_{\mathrm{flat}}p^{1/6}
+O(p^{-1/3}),\]
where
\begin{equation}\label{Aflat}
A_{\mathrm{flat}}
=\frac{\Gamma(1/2)}{\Gamma(5/3)}\ne0.
\end{equation}
Thus adding the flat mode changes the coefficient of the dominant $p^{2/3}$ branch at $+\infty$ without changing any derivative at $p=0$.

\begin{lemma}[Inner profile]\label{innerprofilefinal}
There exists $Q\in C^\infty(\R)$ solving \eqref{innerODErig} such that
\begin{align*}
Q(p)&=p^{2/3}+B_+p^{1/6}+\frac13p^{-1/3}+O(p^{-5/6}),
&&p\to+\infty,\\
Q(p)&=-|p|^{2/3}+B_-|p|^{1/6}+\frac13|p|^{-1/3}+O(|p|^{-5/6}),
&&p\to-\infty,
\end{align*}
for constants $B_\pm$, with $B_-\ne0$.  The coefficient $1/3$ of the additive $|p|^{-1/3}$ term is the same on the two sides.  The same expansions hold after a fixed number of differentiations.
\end{lemma}

\begin{proof}
Take the unique smooth negative-side branch normalized by $A_-=-1$.  Its value at $p=0$ fixes, through \eqref{regularpositive}, a smooth positive-side solution $Q_{\rm reg}$ with the same complete Taylor jet.  Every other positive-side solution with that same jet is
\[
Q_{\rm reg}+\lambda Q_{\rm flat},
\]
because the difference has zero Taylor jet and therefore lies in the one-dimensional flat Stokes space.  By \eqref{Aflat}, the coefficient $A_+$ of $p^{2/3}$ depends affinely and nontrivially on $\lambda$.  There is therefore a unique $\lambda$ for which $A_+=1$.  The remaining coefficients are then fixed.  The expansions and their differentiated versions follow from the standard asymptotic expansions of $M$ and $U$.  For later matching we also record the coefficient of the next term.  On $p>0$ substitute
\[
 Q=p^{2/3}+C p^{-1/3}+o(p^{-1/3})
\]
into \eqref{innerODErig}.  The coefficient of $p^{-1/3}$ in the equation is $2C-2/3$, hence $C=1/3$.  On $p<0$, writing $x=-p$ and
\[
 Q=-x^{2/3}+C x^{-1/3}+o(x^{-1/3})
\]
gives the same equation $2C-2/3=0$.  Thus the additive coefficient is $1/3$ on both sides.
\end{proof}

This also makes explicit why the coefficient $B_+$ is not at our disposal: the single flat parameter is already used to normalize the leading coefficient $A_+$.

\subsection{The matching expansion and derivative estimates}\label{sec:matching}

Set
\[
p=\frac{T}{\eps r^2},
\qquad
h_{\mathrm{in}}(r,T)=\eps^{2/3}r^{-5/3}Q(p).
\]
Using Lemma \ref{innerprofilefinal}, for $|p|\gg1$ we obtain
\[h_{\mathrm{in}}
=\frac{\sgn(T)|T|^{2/3}}{r^3}
+B_\pm\eps^{1/2}\frac{|T|^{1/6}}{r^2}
+\frac13\eps\frac{|T|^{-1/3}}r
+R_{\mathrm{match}},\]
where
\[|R_{\mathrm{match}}|
\le C\eps^{3/2}|T|^{-5/6}.\]
Equivalently, in the matching region $|T|\sim r^{3-\sigma}$,
\[|R_{\mathrm{match}}|
\le C\eps^{3/2}r^{-5/2+5\sigma/6}.\]

More importantly for cutoff estimates, differentiating the expansion gives, for $j+k\le2$,
\[|\partial_r^j\partial_T^kR_{\mathrm{match}}|
\le C\eps^{3/2-k}r^{-5/2-j+5\sigma/6-2k}\]
in that overlap region, modulo lower-order contributions from differentiating the recentering $T=t+\eps\Phi_1$.  Those contributions have the same or better order because $\partial_rT=O(r^2)$ whereas $|T|\sim r^{3-\sigma}$.

Consequently, if the cutoff varies on radial scale $r$, its commutator with the second-order operator produces
\begin{equation}\label{matching-cutoff}
O\left(\eps^{3/2}r^{-9/2+5\sigma/6}\right)
\end{equation}
at the level of the normal second correction.  For $\sigma<3/5$ this is strictly smaller than $r^{-4}$.

After the $p^{1/6}$ term is included, \eqref{matching-cutoff} gives the desired gain.

\section{Global barriers and the approximate translator}\label{sec:globalbarriers}

We now return to the final existence argument.  Let $h$ be a positive Jacobi barrier and set
\[
\psi=\sqrt{1+|\grad F_0|^2}\,h.
\]
The translator linearization contains a first-order drift in the vertical variable, so the exterior estimate must be made directly on the product $\psi=W_0h$, rather than by estimating $W_0$ and $h$ separately.  For the explicit exterior homogeneous symbol chosen below, this product has degree $-\nu$.  Differentiation at fixed weighted coordinate then gains three radial powers, which places the translator drift one full power below the Jacobi margin.  This sharp homogeneity estimate is the reason a single global barrier suffices.

\subsubsection{The global Jacobi barrier}

Fix $0<\nu<\sigma_0/10$ and let $q_\nu$ be the positive angular function
constructed in \eqref{qnu-barrier}.  On the end of the exact BDG graph write
\[
 W_\Gamma=(1+|\nabla\olF|^2)^{1/2}
\]
and define the normal function
\begin{equation}\label{h-exterior}
 h_{\rm ext}(r,\theta)=\frac{r^{-\nu}q_\nu(\theta)}{W_\Gamma(r,\theta)}.
\end{equation}
The model calculation \eqref{qnu-source}, the refined comparison between
$\Gamma$ and $\Gamma_0$, and \eqref{qnu-bounds} imply that, for $r$ large,
\[
 \cJ_\Gamma h_{\rm ext}\le -c r^{-4-\nu},
 \qquad
 c r^{-2-\nu}\le h_{\rm ext}\le C r^{-2-\nu}.
\]
Choose $R$ so large that these estimates hold for $r>R$.  As in the compact
completion used for the fast inverse, extend $h_{\rm ext}$ to a smooth positive
function $h$ on all of $\Gamma$ so that
\begin{equation}\label{Jbarriermargin}
 \cJ_\Gamma h\le -c(1+r)^{-4-\nu},
\end{equation}
and arrange, by making the completion inside $2R$, that
\begin{equation}\label{hexplicit}
 h=h_{\rm ext}\qquad\hbox{for }r>2R.
\end{equation}
The construction uses only a bounded-domain Dirichlet problem and the maximum
principle supplied by the positive Jacobi field.  Thus no global anisotropic
estimate is hidden in the definition of $h$.

The advantage of \eqref{h-exterior} is that its associated vertical variation
with respect to the exact BDG graph is the explicit homogeneous symbol
$r^{-\nu}q_\nu(\theta)$.  This is exactly the structure needed to control the
translator drift.

\subsubsection{The $t$-derivative gains three powers on homogeneous functions}

Recall the notation $D(\theta)=9g(\theta)^2+g'(\theta)^2$ from Section~\ref{sec:bdg-jacobi}, and set
\[
d(\theta):=\sqrt{D(\theta)}.
\]
Since $g(\pi/4)=0$, $g'(\pi/4)>0$, $g'(\pi/2)=0$, and $g(\pi/2)>0$, one has
\[
0<c\le d(\theta)\le C
\qquad\hbox{for }\theta\in[\pi/4,\pi/2].
\]
Recall that
\[
t=F_0=r^3g(\theta).
\]
The coordinate formulas derived in Section~\ref{sec:outer} give
\[r_t=\frac{3g}{r^2D}=O(r^{-2}),
\qquad
\theta_t=\frac{g'}{r^3D}=O(r^{-3}),\]
uniformly all the way to the Simons cone.  Hence, if
\[
q(r,\theta)=r^\beta a(\theta)
\]
with $a$ smooth on the closed angular interval, then
\begin{equation}\label{homog-t-der}
q_t=q_r r_t+q_\theta\theta_t=O(r^{\beta-3}).
\end{equation}
Iterating once more gives
\begin{equation}\label{homog-tt-der}
q_{tt}=O(r^{\beta-6}).
\end{equation}
The same rule holds for classical symbols with lower homogeneous terms.

This elementary observation yields the required sharp drift estimate.

\subsubsection{The explicit exterior vertical barrier}

On the homogeneous model set
\begin{equation}\label{psi0-angular}
 \psi_0(r,\theta):=r^{-\nu}q_\nu(\theta).
\end{equation}
By \eqref{qnu-bounds}, this is a smooth positive classical symbol of degree
$-\nu$.  Applying \eqref{homog-t-der}--\eqref{homog-tt-der} gives
\begin{equation}\label{psi-t-est}
 (\psi_0)_t=O(r^{-3-\nu}),
 \qquad
 (\psi_0)_{tt}=O(r^{-6-\nu}).
\end{equation}
Since $W_0\asymp r^2$,
\begin{equation}\label{drift-good}
 W_0^{-1}|(\psi_0)_t|=O(r^{-5-\nu}).
\end{equation}
Thus the translator drift gains one full power over the Jacobi margin
$r^{-4-\nu}$.

For the actual approximate translator the vertical barrier is
\[
 \psi_\eps=W_{\rm app}h,
 \qquad
 W_{\rm app}=(1+|\nabla F_{\rm app}|^2)^{1/2}.
\]
In the exterior, \eqref{hexplicit} and \eqref{h-exterior} give
\begin{equation}\label{psiapp-symbol}
 \psi_\eps
 =\frac{W_{\rm app}}{W_\Gamma}\,r^{-\nu}q_\nu(\theta).
\end{equation}
The construction of $F_{\rm app}$ gives
$W_{\rm app}/W_\Gamma=1+O(\eps r^{-1})$ at the first correction level, with
smaller contributions from the layer and the higher corrections, together
with the corresponding differentiated symbol estimates.  Hence
\eqref{psi-t-est}--\eqref{drift-good} remain valid for $\psi_\eps$ up to terms
that are strictly lower order.  It is therefore enough to perform the
remaining exterior calculation on the leading symbol $\psi_0$.

The same observation controls the additional terms generated by the first correction $\Phi_1$.  From the explicit separated construction,
\begin{equation}\label{Phi1-symbol}
\rho^{-1}(\Phi_1)_s=O(r),
\qquad
(\Phi_1)_t=O(r^{-1}),
\end{equation}
uniformly in the sector.  Moreover $\rho^{-1}\partial_s$ is a unit-size derivative tangent to a level set of $F_0$.  Therefore
\begin{equation}\label{psi-s-est}
\rho^{-1}(\psi_0)_s=O(r^{-1-\nu}),
\qquad
\rho^{-1}(\psi_0)_{ts}=O(r^{-4-\nu}).
\end{equation}
The potentially dangerous pieces in the corrected linearized operator are consequently
\begin{align}
\eps^2\frac{\rho^{-2}(\Phi_1)_s^2}{W_0}(\psi_0)_{tt}
&=O(\eps^2r^{-6-\nu}),\label{bar-extra1}\\
\eps\frac{\rho^{-2}(\Phi_1)_s}{W_0}(\psi_0)_{ts}
&=O(\eps r^{-5-\nu}),\label{bar-extra2}\\
\frac{\eps}{W_0}(\psi_0)_t
&=O(\eps r^{-5-\nu}).\label{bar-extra3}
\end{align}
For the mixed term we used it in the geometrically natural form
\[
\frac{\rho^{-1}(\Phi_1)_s}{W_0}\,
\rho^{-1}(\psi_0)_{ts},
\]
so that \eqref{Phi1-symbol} and \eqref{psi-s-est} can be applied directly.

All three terms gain at least one full power of $r^{-1}$ over the Jacobi margin $r^{-4-\nu}$.

\subsubsection{The transition annulus}

The exterior calculation above applies for $r>2R$, where $h$ agrees with the explicit homogeneous exterior barrier in \eqref{hexplicit}.  In the fixed annulus
\[
R<r<2R
\]
the glued Jacobi barrier already has the strict margin \eqref{Jbarriermargin}.  Once $R$ has been chosen, all coefficients of the translator linearization and all derivatives of $h$ are bounded there by constants depending only on $R$.  The new terms are therefore bounded by
\[
C_R\eps.
\]
Choosing $\eps_0=\eps_0(R)$ sufficiently small absorbs them into the fixed negative Jacobi margin.  On the compact region $r\le R$, the same argument applies after the standard compactly supported modification of the barrier.

Thus the correct order of choices is
\[R\gg1\quad\hbox{first},
\qquad
0<\eps<\eps_0(R)\quad\hbox{afterwards}.\]
No anisotropic global Schauder estimate is required.

\subsubsection{Linearization in graph variables}

There is another way to see why the cancellation above is natural.  For
\[
\cM_\eps[F]
=\div\left(\frac{\grad F}{\sqrt{1+|\grad F|^2}}\right)
-\frac{\eps}{\sqrt{1+|\grad F|^2}},
\]
the linearization in a vertical graph perturbation $\phi$ is
\begin{equation}\label{graphlin}
D\cM_\eps[F]\phi
=\div(A_F\grad\phi)
+\eps\frac{\grad F\cdot\grad\phi}{(1+|\grad F|^2)^{3/2}},
\end{equation}
where
\[
A_F
=\frac{I}{\sqrt{1+|\grad F|^2}}
-\frac{\grad F\otimes\grad F}{(1+|\grad F|^2)^{3/2}}.
\]
There is no zero-order term, reflecting vertical translation invariance.  For a vertical symbol of degree $-\nu$, the drift term in \eqref{graphlin} is automatically of order $\eps r^{-5-\nu}$.  Formula \eqref{drift-good} is the $(s,t)$ manifestation of this elementary graph-variable fact.

The sharper homogeneity calculation is essential here: an $O(r^{-3-\nu})$ drift estimate would be too crude, since a vertical perturbation of degree $-\nu$ cannot generate a first-order translator term of that size after division by $|\grad F_0|^3\sim r^6$.

\subsubsection{Transfer from the model graph to the approximate translator}

Let $F_{\mathrm{app}}$ be the global approximation constructed in the following section.  In the exterior region its gradient differs from that of $F_0$ by quantities of relative size $O(r^{-1})$ or better, with additional powers of $\eps$ in the inner corrections.  Proposition~\ref{transferprop}, together with the differentiated matching estimates, therefore gives
\[D\cM_\eps[F_{\mathrm{app}}]\psi
=\mathcal L_{0,\eps}\psi
+\mathcal R_{\mathrm{lin}}[\psi],\]
where $\mathcal L_{0,\eps}$ denotes the corrected model linearization and, for the barrier above,
\[
|\mathcal R_{\mathrm{lin}}[\psi]|
=o(r^{-4-\nu})
\]
uniformly for $r>2R$ after $R$ is fixed large and $\eps$ is sufficiently small.

Combining \eqref{Jbarriermargin}, \eqref{bar-extra1}--\eqref{bar-extra3}, and the annulus argument yields the following proposition.

\subsubsection{Completion on the compact region}\label{subsec:compactcompletion}

The function $h$ in \eqref{Jbarriermargin} is already a \emph{global} positive Jacobi barrier; its exterior part is replaced by the explicit homogeneous power in order to recover the sharp $t$-derivative cancellation.

Let
\[
\psi_\eps:=W_{\rm app}h,
\qquad
W_{\rm app}:=\sqrt{1+|\nabla F_{\rm app}|^2}.
\]
On every fixed ball $B_{4R}$, normal and vertical linearizations differ by smooth positive conjugating factors, and
\[
D\cM_\eps[F_{\rm app}]\psi_\eps
=
\cJ_\Gamma h+O_R(\eps)h+O_R(\eps)|D_\Gamma h|.
\]
Since $h$ was constructed so that
\[
\cJ_\Gamma h\le -c_R<0
\qquad\hbox{on }B_{4R},
\]
the same strict sign holds there once $\eps<\eps_0(R)$.  In the exterior $r>2R$ the sharper calculation of the preceding subsections gives the weighted sign $-cr^{-4-\nu}$.  Thus the \emph{same} positive function supplies both the compact and exterior parts of the barrier.

\begin{lemma}[Compact transfer of the global Jacobi barrier]\label{compactcompletionlemma}
After choosing $R$ large and then $\eps_0(R)>0$ small, the positive vertical function $\psi_\eps=W_{\rm app}h$ satisfies
\begin{equation}\label{compactcompleteineq}
D\cM_\eps[F_{\rm app}]\psi_\eps
\le-c_*(1+r)^{-4-\nu}
\qquad\hbox{in }\R^8,
\end{equation}
for every $0<\eps<\eps_0(R)$.
\end{lemma}

\begin{proof}
On $B_{4R}$ this follows from the fixed strict Jacobi margin and the $O_R(\eps)$ perturbation estimate above.  On $r>2R$ it follows from the exterior calculation.  The annulus $2R<r<4R$ is compact, and the two estimates overlap there; decreasing $\eps_0(R)$ once more gives a common constant $c_*>0$.
\end{proof}

The reason for imposing the explicit homogeneous exterior form outside $2R$ is the gain in the translator drift; no second compact barrier is needed.

\begin{proposition}[Global barrier]\label{globalbarrier}
There exist $\nu>0$ and, after choosing $R$ sufficiently large, an $\eps_0(R)>0$ and a smooth positive $O(4)\times O(4)$-invariant vertical function $\psi_\eps$ such that for every $0<\eps<\eps_0(R)$,
\begin{equation}\label{globalbarrier-sign}
D\cM_\eps[F_{\mathrm{app}}]\psi_\eps
\le -c(1+r)^{-4-\nu}
\end{equation}
outside a fixed compact set.  Moreover
\[
\psi_\eps=\dfrac{W_{\rm app}}{W_\Gamma}r^{-\nu}q_\nu(\theta)
\qquad\hbox{for }r>2R,
\]
up to the lower-order transfer from $\Gamma_0$ to $\Gamma$, and the barrier can be modified on the fixed compact region so that \eqref{globalbarrier-sign} holds globally.
\end{proposition}

This proposition yields the required global barrier directly once the exterior product $W_0h$ is differentiated as a whole.

\subsection{The global layer, exactification, and residual budget}\label{sec:globalapprox}

The purpose of the exactification step is easy to obscure under the notation, so we first explain what is being corrected.  The inner profile and the outer Jacobi modes have been matched to sufficiently high order, but the object obtained by cutting them together is still only an approximate solution of the linearized translator equation.  There are three kinds of defects: a compactly supported core error, an error in the mesoscopic region where the parabolic model is only approximate, and commutators produced by the cutoffs.  None of these should be absorbed into the final nonlinear barrier.  Doing so would force the barrier to compensate errors of different geometric origins and would destroy the clean separation between approximation and comparison.

We therefore correct the layer once at the linear level.  The correction is carried out in the normal variable, not in the vertical graph variable.  This distinction is essential.  In graph coordinates the ellipticity ratio of the raw linearized operator reflects the large slope $|\nabla F_0|\sim r^2$ and is not uniform on the Euclidean scale.  In normal variables, after recentering and rescaling by the natural local length, the operator is uniformly elliptic.  Standard Schauder and maximum-principle arguments can then be applied with constants independent of the large radius.  Only after the normal correction has been estimated do we return to graph variables by multiplication with the normal factor $W$.

A second point is the order of choosing parameters.  The constants controlling the cone cutoff, the parabolic-to-elliptic transition, and the growing exactification radius cannot be selected independently at the end of the proof.  We first fix the small weight exponent $\nu$, then choose the cone parameter $K$, next the mesoscopic parameter $L$, then the exponent $\kappa$ defining $R_\varepsilon=\varepsilon^{-\kappa}$, and only afterwards take $\varepsilon$ small.  Written in this order, every commutator estimate has a fixed margin and no later choice changes an earlier constant.  The residual budget at the end of this subsection is simply the quantitative record of these margins.

The final approximation is most naturally organized around a single $\eps$-dependent transition layer.  Splitting this layer into $U_2$, $U_{5/2}$ and the matching part of $U_3$ is useful only in the outer expansion; doing so globally would double-count the fractional modes contained in the inner profile.

\subsubsection{Exact parity of the integer coefficients}

Let
\[
 S(F)=(1+|\nabla F|^2)^{-1/2},\qquad
 \mathcal M_\eps(F)=H(F)-\eps S(F),
\]
and let $\mathscr S(u,v)=(v,u)$.  Since
\[
 H(-F)=-H(F),\qquad S(-F)=S(F),\qquad \bar F\circ\mathscr S=-\bar F,
\]
the integer-order coefficients can be chosen with alternating parity:
\[U_1\circ\mathscr S=U_1,\qquad
 U_2\circ\mathscr S=-U_2,\qquad
 U_{3,\rm reg}\circ\mathscr S=U_{3,\rm reg}.\]
In particular
\[
 LU_1=S(\bar F),
\]
and the whole second-order forcing
\[
 F_2:=B_2(U_1,U_1)-S_1(U_1)
\]
is odd.  By Corollary~\ref{discreteBDGsymbols},
\[
 F_2=r^{-3}P_3^{\rm odd}(\theta)+\mathcal R_2.
\]

\subsubsection{A single global transition layer}

Let
\[
 T=t+\eps U_1
\]
be the recentered vertical coordinate and
\[
 p=\frac{T}{\eps r^2}.
\]
Let $Q$ be the smooth inner profile constructed in the rigorous Kummer--Stokes lemma.  Define the normal layer in the inner region by
\begin{equation}\label{layernormal}
 h_{\rm layer,\eps}(r,T)
 :=\eps^{8/3}r^{-5/3}Q(p),
\end{equation}
where the prefactor includes the actual order $\eps^2$ of the graph correction.  Passing from normal to vertical displacement defines a graph layer
\[
 \mathcal U_{\rm layer,\eps}.
\]
For $|p|\to\infty$ this single object has the outer expansion
\begin{equation}\label{layerouter}
 \mathcal U_{\rm layer,\eps}
 =\eps^2U_2
 +\eps^{5/2}U_{5/2}
 +\eps^3U_{3,\rm match}
 +O_{\rm ad}(\eps^{7/2}r^{-1/2-\delta}),
\end{equation}
where
\[
 U_2\sim r\,q_2(\theta),\qquad
 U_{5/2}\sim B_\pm r^{1/2}g(\theta)^{1/6},
\]
and the third term contains the $g^{-1/3}$ matching contribution.  Formula \eqref{layerouter} is an \emph{asymptotic expansion of the layer}; the three displayed terms are never added again in a region where the full $Q$-profile is being used.

The $p^{1/6}$ and $p^{-1/3}$ terms are already contained in the full inner profile.  They appear as separate fractional powers of $\eps$ only after the same smooth layer is expanded for large $|p|$.

\subsubsection{Where the asymptotic lateral matching is allowed}

At this point it is essential to keep separate the \emph{full layer} and the coefficient obtained after factoring out $\eps^2$.  We therefore introduce both natural scales.  In normal variables,
\[\mathcal H_{\rm lay}(r):=\eps^{8/3}r^{-5/3},\]
while the corresponding vertical graph displacement has size
\[\mathcal A_{\rm lay}(r):=W_1\mathcal H_{\rm lay}(r)
 \asymp \eps^{8/3}r^{1/3},
 \qquad
 W_1:=\sqrt{1+|\nabla F^{(1)}_\eps|^2}\asymp r^2.\]
If one factors out the explicit $\eps^2$ from the graph correction, the remaining coefficient has size $\eps^{2/3}r^{1/3}$.  In the exactification argument below we keep the full graph-layer size instead.

Fix $0<\nu\ll1$ and choose
\begin{equation}\label{kappachoice}
 R_\eps=\eps^{-\kappa},\qquad
 \frac{1+3\nu}{5-3\nu}<\kappa<1.
\end{equation}
We fix the parameters in the following order throughout this section:
\[0<\nu<\min\{\sigma_0/10,1/4\},\qquad K\gg1,\qquad L\gg_K1,\]
\[\frac{1+3\nu}{5-3\nu}<\kappa<1,\qquad 0<\eps<\eps_0(K,L,\kappa,\nu).\]
Here $K$ controls the large-$|p|$ matching, whereas $L$ places the sole radial transition at $r\asymp L/\eps$.  Since $\kappa<1$,
\[
 R_\eps=\eps^{-\kappa}=o(L/\eps),
\]
so the core source and the radial transition are disjoint for small $\eps$.  No parameter chosen later is allowed to alter $K,L,\nu$ or $\kappa$.

The outer representation \eqref{layerouter} is used only after the growing exactification has been introduced.  In a lateral matching band $K<|p|<2K$, with $p=T/(\eps r^2)$, the first omitted term in the \emph{full normal layer} is
\[
 O\left(K^{-5/6}\eps^{8/3}r^{-5/3}\right).
\]
A transversal derivative of a cutoff depending on $p$ costs $(\eps r^2)^{-1}$.  The second-order commutator therefore satisfies the sharp bound
\begin{equation}\label{latcomm-full}
 |\mathcal C_{\rm lat}|
 \le C K^{-5/6}\eps^{2/3}r^{-17/3}.
\end{equation}
We perform this lateral matching only for $r\ge R_\eps=\eps^{-\kappa}$.  Hence
\[
 \frac{|\mathcal C_{\rm lat}|}{\eps r^{-4-2\nu}}
 \le C K^{-5/6}\eps^{-1/3}r^{-5/3+2\nu}
 \le C K^{-5/6}\eps^{-1/3+\kappa(5/3-2\nu)}.
\]
The exponent of $\eps$ is positive provided $\kappa>1/(5-6\nu)$.  Our choice \eqref{kappachoice} is stronger, since for $0<\nu<1/4$,
\[
 \frac{1+3\nu}{5-3\nu}>\frac1{5-6\nu}.
\]
Thus the same growing radius that controls the core propagation also makes the lateral matching commutator strictly smaller than the target residual.

The quantity $\eps^{2/3}r^{1/3}$ is the size of the \emph{coefficient after removing $\eps^2$}; the actual graph layer has size $\eps^{8/3}r^{1/3}$.  We keep this distinction throughout the exactification.

\subsubsection{Growing-interior exactification in the normal variable}

The exactification is most naturally carried out in normal, not vertical, variables.  Let
\[
 F^{(1)}_\eps=\bar F+\eps U_1,
 \qquad
 W_1=\sqrt{1+|\nabla F^{(1)}_\eps|^2},
\]
and write a vertical perturbation as
\[
 \delta F=W_1 h.
\]
Define the exact normal linearized operator by
\[\mathscr J^{(1)}_\eps h
 :=D\mathcal M_\eps[F^{(1)}_\eps](W_1h).\]
For $\eps=0$ this is the Jacobi operator of the BDG graph, up to the harmless normalization induced by the vertical-to-normal conversion.  In the recentered variables $(r,T)$ its principal part is the elliptic--drift operator derived in Section~\ref{sec:innerlayer}.  In particular, on boxes of intrinsic size comparable with $r$ and with $|p|$ bounded, the rescaled operator is uniformly elliptic with coefficients bounded independently of $\eps$.

Let $\widehat h_{\rm lay,\eps}$ denote a provisional normal layer.  In the inner region it is the full profile \eqref{layernormal}; in the exterior it is obtained from the common large-$|p|$ expansion of the same profile.  Extend it smoothly through the fixed core, preserving the natural normal scaled bounds
\begin{equation}\label{normalextension}
 |D_\Gamma^j\widehat h_{\rm lay,\eps}|
 \le C\eps^{8/3}(1+r)^{-5/3-j},
 \qquad j=0,1,2.
\end{equation}
We correct this provisional normal layer on one growing domain, rather than solving on an annulus and extending a Dirichlet solution by zero.  This point is essential because the latter procedure would create an uncontrolled jump of the normal derivative.

Let
\begin{equation}\label{growingdomain}
 \Omega_{\eps,L}:=\{x\in\Gamma: r(x)<4L/\eps\}.
\end{equation}
All correctors below are solved on the same domain, with zero data only on the outer boundary $r=4L/\eps$.  Since that boundary lies beyond the radial gluing region, no corrector is ever extended by zero across a boundary at which it is still being used.

\paragraph{Solvability on the growing domain.}
Before constructing the full approximation, transfer the positive Jacobi barrier to the first approximation $F^{(1)}_\eps$.  In normal variables this gives a positive function $b^{(1)}_\eps$ satisfying
\[\mathscr J^{(1)}_\eps b^{(1)}_\eps
 \le-c(1+r)^{-4-\nu},
 \qquad
 b^{(1)}_\eps\asymp r^{-2-\nu}
 \quad(r\gg1),\]
with constants independent of the outer radius of $\Omega_{\eps,L}$.  Dividing by $b^{(1)}_\eps$ gives the usual maximum principle for the conjugated operator.  The Dirichlet problem
\[
 \mathscr J^{(1)}_\eps z=f\quad\hbox{in }\Omega_{\eps,L},
 \qquad z=0\quad\hbox{on }\partial\Omega_{\eps,L},
\]
has trivial Dirichlet kernel.  Since this is a uniformly elliptic second-order operator on the bounded smooth domain, the Fredholm alternative (equivalently, the method of continuity from a coercive Dirichlet operator, using the same maximum principle along the homotopy) gives existence for the invariant sources used below.  The estimates are independent of the growing outer radius because they follow from the same positive barrier and scaled interior Schauder estimates.

\paragraph{Core correction.}
Let $e_{\rm core}$ denote the error generated by the core extension in \eqref{normalextension}.  Choose the transition to the genuine layer before $2R_\eps$; since a normal Jacobi operator loses two radial powers,
\[|e_{\rm core}|
 \le C\eps^{8/3}(1+r)^{-11/3},
 \qquad
 \operatorname{supp}e_{\rm core}\subset\{r<2R_\eps\}.\]
Define $z_{\rm core}$ as the solution on $\Omega_{\eps,L}$ of
\[\mathscr J^{(1)}_\eps z_{\rm core}=-e_{\rm core},
 \qquad z_{\rm core}=0\quad\hbox{at }r=4L/\eps.\]
Comparison with
\[
 C\eps^{8/3}R_\eps^{1/3+\nu}b^{(1)}_\eps
\]
and scaled Schauder estimates give
\begin{align*}
|z_{\rm core}|+r|D_\Gamma z_{\rm core}|+r^2|D_\Gamma^2z_{\rm core}|
&\le C\eps^{8/3}R_\eps^{1/3+\nu}r^{-2-\nu}\\
&=C\eps^{8/3}R_\eps^{-5/3}
 \left(\frac{R_\eps}{r}\right)^{2+\nu},
 \qquad 3R_\eps\le r\le3L/\eps.
\end{align*}
Since $r\ge R_\eps$ in this region, we shall also use the weaker consequence
\[
 |z_{\rm core}|+r|D_\Gamma z_{\rm core}|+r^2|D_\Gamma^2z_{\rm core}|
 \le C\eps^{8/3}R_\eps^{-5/3}
 \left(\frac{R_\eps}{r}\right)^\nu.
\]
The stronger estimate, with the explicit factor $r^{-2}$, is used when converting to graph variables.  Passing back to the graph variable $Z_{\rm core}=W_1z_{\rm core}$ yields
\begin{equation}\label{Zcore}
 |Z_{\rm core}|+r|DZ_{\rm core}|+r^2|D^2Z_{\rm core}|
 \le C\eps^{8/3}R_\eps^{1/3}
 \left(\frac{R_\eps}{r}\right)^\nu.
\end{equation}

\paragraph{Exterior model correction.}
In the region where the heat model is replaced by the exact recentered normal operator, the discarded coefficient is relative $O((\eps r)^{-2})$.  Let $e_{\rm ext}$ be the corresponding source, cut off smoothly inside
\[
 \frac12L\eps^{-1}<r<\frac32L\eps^{-1},
 \qquad |p|<2K.
\]
Then
\begin{equation}\label{eext-normal}
 |e_{\rm ext}|_{0,\alpha}
 \le C L^{-2}r^{-2}\mathcal H_{\rm lay}(r)
 =C L^{-2}\eps^{8/3}r^{-11/3}.
\end{equation}
Define $z_{\rm ext}$ on the same growing domain by
\[\mathscr J^{(1)}_\eps z_{\rm ext}=-e_{\rm ext},
 \qquad z_{\rm ext}=0\quad\hbox{at }r=4L/\eps.\]
On the support of $e_{\rm ext}$ set
\[
 \rho=\frac{\eps r}{L},\qquad
 \pi=p=\frac{T}{\eps r^2}.
\]
After division by a positive principal coefficient, $\mathscr J^{(1)}_\eps$ is uniformly elliptic with bounded drift on a fixed $(\rho,\pi)$ domain; the constants depend on the already fixed $K,L$ but not on $\eps$.  Barrier comparison followed by scaled Schauder estimates gives
\begin{equation}\label{zext-normal}
 |z_{\rm ext}|+r|D_\Gamma z_{\rm ext}|+r^2|D_\Gamma^2z_{\rm ext}|
 \le C L^{-2}\mathcal H_{\rm lay}(r)
\end{equation}
in the radial gluing region.  Equivalently, for $Z_{\rm ext}=W_1z_{\rm ext}$,
\begin{equation}\label{Zext}
 |Z_{\rm ext}|+r|DZ_{\rm ext}|+r^2|D^2Z_{\rm ext}|
 \le C L^{-2}\mathcal A_{\rm lay}(r).
\end{equation}
The crucial point is that both $z_{\rm core}$ and $z_{\rm ext}$ remain genuine solutions through $r=3L/\eps$; the radial cutoff below is applied strictly inside their common domain.  Hence no derivative jump is hidden in the construction.

This is the point at which the normal variable matters.  In the graph variable the ellipticity ratio of the raw divergence-form matrix grows like $W_1^2$.  In normal recentered variables the operator is uniformly elliptic on the intrinsic scale $r$, and its inverse gains two powers, as \eqref{eext-normal}--\eqref{zext-normal} show.

\subsubsection{The only radial gluing}

The exactified growing-interior normal layer (including $z_{\rm core}+z_{\rm ext}$) and the exterior normal composite are glued only in
\[
 2L\eps^{-1}<r<3L\eps^{-1},
\]
which lies strictly inside $\Omega_{\eps,L}$.
A radial cutoff in the normal Jacobi equation costs $O(r^{-2})$.  Hence the core contribution to the mean-curvature commutator is
\[
 O\!\left(
 r^{-2}\eps^{8/3}R_\eps^{-5/3}
 (R_\eps/r)^\nu
 \right),
\]
and the exterior contribution is
\[
 O\!\left(L^{-2}r^{-2}\mathcal H_{\rm lay}(r)\right).
\]
At $r\simeq L/\eps$, the first is below the target $\eps r^{-4-2\nu}$ provided
\[
 \kappa>\frac{1+3\nu}{5-3\nu},
\]
which is exactly \eqref{kappachoice}, while the second is
\[O(\eps^{19/3}L^{-17/3}),\]
again much smaller than $\eps^{5+2\nu}L^{-4-2\nu}$ for $\nu>0$ sufficiently small.

Finally define the exact graph layer by multiplying the exactified normal layer by $W_1$:
\[\mathcal U_{\rm layer,\eps}^{\rm exact}:=W_1h_{\rm layer,\eps}^{\rm exact}.\]
The conversion is performed only after the normal exactification estimates have been established.

\subsubsection{Regular order $\eps^3$ correction and the residual}

Once the layer block is fixed, the remaining integer order-$\eps^3$ source is the regular source $F_{3,\rm reg}$ of Lemma~\ref{F3symbol}.  Its first resonant symbol is $r^{-4}P_{4,3}^{\rm even}$ by \eqref{F3sharp}.  Let $U_{3,\rm reg}$ denote the corresponding even correction, including the logarithmic piece when the cone value of the angular coefficient is nonzero and the fast inverse for the remainder.  The matching $g^{-1/3}$ term is \emph{not} included in $U_{3,\rm reg}$ because it is already contained in $\mathcal U_{\rm layer,\eps}$.

The global approximation is therefore
\begin{equation}\label{Fapp}
 F_{\rm app}
 =\bar F+\eps U_1
 +\mathcal U_{\rm layer,\eps}^{\rm exact}
 +\eps^3U_{3,\rm reg}.
\end{equation}
The first uncancelled regular outer terms are of order $\eps^{7/2}$ and have radial degree at most $-9/2$; the quadratic remainder in the growing region is controlled by the exact graphical Taylor estimate.  We obtain
\begin{theorem}[Global approximation theorem]\label{globalapprox}
For some $\nu>0$ and all sufficiently small $\eps$ there exists a smooth $O(4)\times O(4)$-invariant function $F_{\rm app}$ of the form \eqref{Fapp} such that
\begin{equation}\label{globalres}
 |\mathcal M_\eps(F_{\rm app})(x)|
 \le C\eps(1+r)^{-4-2\nu}.
\end{equation}
Moreover, on each fixed compact set,
\[
 \|F_{\rm app}-\bar F\|_{C^{2,\alpha}}\le C_K\eps.
\]
\end{theorem}

\begin{proof}
The order-$\eps$ equation is solved by $U_1$.  The complete singular order-$\eps^2$ correction, together with the fractional $\eps^{5/2}$ and matching $\eps^3$ pieces, is represented by the single exactified layer block.  Finally $U_{3,\rm reg}$ cancels the remaining even resonant order-$\eps^3$ source.  We spell out the residual budget in order to keep the different spatial scales and powers of $\varepsilon$ visible in one place.

\smallskip
\noindent\emph{Lateral matching.}  By \eqref{latcomm-full}, for $r\ge R_\eps$,
\[
 \frac{|\mathcal C_{\rm lat}|}{\eps r^{-4-2\nu}}
 \le CK^{-5/6}\eps^{-1/3+\kappa(5/3-2\nu)}=o(1),
\]
because \eqref{kappachoice} implies $\kappa>1/(5-6\nu)$.

\smallskip
\noindent\emph{Core propagation and radial cutoff.}  At $r\simeq L/\eps$, the commutator generated by $z_{\rm core}$ is bounded by
\[
 C r^{-2}\eps^{8/3}R_\eps^{-5/3}(R_\eps/r)^\nu.
\]
Dividing by the target $\eps r^{-4-2\nu}$ gives a power of $\eps$ which tends to zero precisely when
\[
 \kappa>\frac{1+3\nu}{5-3\nu}.
\]
This is the lower bound in \eqref{kappachoice}.

\smallskip
\noindent\emph{Exterior model error.}  The exact corrector $z_{\rm ext}$ removes the source \eqref{eext-normal}; its radial cutoff costs
\[
 O(\eps^{19/3}L^{-17/3}),
\]
whereas the target at $r=L/\eps$ is $\eps^{5+2\nu}L^{-4-2\nu}$.  Since $\nu<1/4$, the former is smaller by a positive power of $\eps$.

\smallskip
\noindent\emph{Unmatched regular outer terms.}  After the coefficient-level subtraction in Lemma~\ref{matchingcoefficient} and the $U_{3,\rm reg}$ correction, the first regular terms not explicitly cancelled have prefactor at least $\eps^{7/2}$ and radial degree at most $-9/2$.  Hence
\[
 \eps^{7/2}r^{-9/2}
 \le \eps r^{-4-2\nu}\,
      \eps^{5/2}r^{-1/2+2\nu}
 =o\bigl(\eps r^{-4-2\nu}\bigr)
\]
for $\nu<1/4$ and $r\ge1$.

\smallskip
\noindent\emph{Nonlinear Taylor terms.}  The perturbations satisfy the relative slope hypothesis of Lemma~\ref{graphremainder}: for the largest one, $\delta=\eps U_1$, one has $|D\delta|/W=O(\eps/r)$ in the exterior, while on compact sets the gradient is $O(\eps)$.  The layer and higher corrections are smaller in the corresponding scaled norms.  Lemma~\ref{graphremainder}, together with \eqref{Zcore}--\eqref{Zext}, therefore puts all quadratic and cubic leftovers strictly below the same target weight.

Together these estimates account for all error terms because $z_{\rm core}$ and $z_{\rm ext}$ solve on the common growing domain \eqref{growingdomain}; in particular, no boundary derivative jump is introduced by the radial gluing.  This proves \eqref{globalres}.  The compact estimate follows from the construction and standard interior Schauder estimates. \qed
\end{proof}

At this point the outer hierarchy is no longer a formal degree table.  The order-$\eps^2$ symbol is defined directly from the homogeneous Taylor operator, and the order-$\eps^3$ regular symbol is defined only after subtracting the contribution already contained in the single smooth layer.  This formulation avoids double counting between the regular outer hierarchy and the single smooth layer.

\begin{lemma}[Quadratic graph-variable Taylor remainder]\label{graphremainder}
Let $F$ be any graph for which $W=(1+|\nabla F|^2)^{1/2}$ and the first two derivatives of $F$ are controlled in the weighted regions used above.  Assume that along the perturbation one has the relative slope bound
\[
 |D\delta|\le c_0 W
\]
with a fixed sufficiently small $c_0$ (on compact sets it is enough that $|D\delta|$ be small).  Then
\[
|\mathcal M_\eps(F+\delta)-\mathcal M_\eps(F)-D\mathcal M_\eps(F)\delta|
\le C\Big(
W^{-2}|D\delta||D^2\delta|+W^{-3}|D^2F|\,|D\delta|^2+\eps W^{-3}|D\delta|^2
\Big).
\]
The same estimate holds in $C^{0,\alpha}$ on scaled intrinsic balls.
\end{lemma}

\begin{proof}
The estimate follows from Taylor expansion of the flux map $p\mapsto p(1+|p|^2)^{-1/2}$ and the translating source along the segment $\nabla F+\tau\nabla\delta$.  The relative slope hypothesis is precisely what keeps the denominators comparable with $W$.  A complete derivation, including the scaled H\"older version, is given in Appendix~\ref{app:taylor}.
\end{proof}

\section{Existence of entire translators and non-convexity}\label{sec:existence}

We now combine the global approximation with the global linear barrier.  At this stage no further asymptotic construction is required.

\subsection{The correct amplitude of the final barriers}

Let $\psi_\eps$ be the positive global function of Proposition~\ref{globalbarrier}, completed on the compact region by Lemma~\ref{compactcompletionlemma}.  Fix a constant $A>1$, independent of $\eps$, and define
\[F_\eps^+=F_{\rm app}+A\eps\psi_\eps,
\qquad
F_\eps^-=F_{\rm app}-A\eps\psi_\eps.\]

The natural amplitude for the final barriers is $A\eps$: the approximate residual is order $\eps$, the linear barrier margin is order one in the normalized weight, and the nonlinear remainder then starts at order $\eps^2$.

Taylor's formula and the graph-variable quadratic estimate give
\begin{align*}
\cM_\eps[F_{\rm app}\pm A\eps\psi_\eps]
={}&\cM_\eps[F_{\rm app}]
\pm A\eps D\cM_\eps[F_{\rm app}]\psi_\eps
+\mathcal N_\eps(\pm A\eps\psi_\eps).
\end{align*}
From Theorem~\ref{globalapprox} and \eqref{compactcompleteineq},
\[
|\cM_\eps[F_{\rm app}]|
\le C\eps(1+r)^{-4-2\nu},
\]
whereas
\[
A\eps D\cM_\eps[F_{\rm app}]\psi_\eps
\le-c_*A\eps(1+r)^{-4-\nu}.
\]
Choose first $A$ so large that the linear negative margin dominates the constant in the approximate residual on the fixed region $r\le R_1$, and then choose $R_1$ large so that the extra factor $r^{-\nu}$ makes the domination automatic for $r\ge R_1$.

For the nonlinear term, Lemma~\ref{graphremainder} applied to $\delta=A\eps\psi_\eps$ gives, in the exterior where $\psi_\eps=O(r^{-\nu})$,
\[|\mathcal N_\eps(A\eps\psi_\eps)|
\le C_A\eps^2(1+r)^{-7-2\nu}.\]
On the fixed compact region the same Taylor formula gives $|\mathcal N_\eps|\le C_A\eps^2$.  Hence, after decreasing $\eps_0(A)$,
\[
\cM_\eps[F_\eps^+]\le0,
\qquad
\cM_\eps[F_\eps^-]\ge0
\qquad\hbox{in }\R^8.
\]
Moreover $F_\eps^-<F_\eps^+$ because $\psi_\eps>0$.

\begin{proposition}[Global ordered barriers]\label{globalorderedbarriers}
For all sufficiently small $\eps>0$ there exist smooth entire functions $F_\eps^-<F_\eps^+$ satisfying
\[
\cM_\eps[F_\eps^-]\ge0,
\qquad
\cM_\eps[F_\eps^+]\le0,
\]
and
\[
F_\eps^\pm=F_{\rm app}\pm A\eps\psi_\eps.
\]
In particular
\[
|F_\eps^\pm-\olF|\le C\eps(1+r^2).
\]
\end{proposition}

\subsection{Exhaustion on large balls}

Fix $M>1$ and consider
\begin{equation}\label{ballproblem}
\begin{cases}
\cM_\eps[F_M]=0 &\text{in }B_M,\\
F_M=F_{\rm app} &\text{on }\partial B_M.
\end{cases}
\end{equation}
Since $F_\eps^-<F_{\rm app}<F_\eps^+$ on $\partial B_M$, Proposition~\ref{globalorderedbarriers} gives ordered sub- and supersolutions for \eqref{ballproblem}.  We use the bounded-domain Dirichlet theory for translating mean-curvature graphs; see Zhou~\cite{ZhouDirichlet}.  A Euclidean ball is smooth and strictly mean convex, and it satisfies the non-closed-minimal condition appearing in that theorem: a compact minimal hypersurface without boundary cannot be contained in a ball, since each ambient coordinate is harmonic on it and the maximum principle would force all coordinates to be constant.  Hence the smooth boundary datum $F_{\rm app}|_{\partial B_M}$ admits a classical smooth translator $F_M$.  The comparison principle, applied to the global sub- and supersolutions, gives
\begin{equation}\label{ballorder}
F_\eps^-\le F_M\le F_\eps^+\qquad\hbox{in }B_M.
\end{equation}

For clarity, the estimates used in passing $M\to\infty$ are local.  Fix $R<\infty$.  From \eqref{ballorder} we obtain an $M$-independent oscillation bound on $B_{2R}$ for all $M>2R$.  The classical interior gradient estimate for translating mean-curvature graphs (for example Zhou~\cite[Lemma~B.2]{ZhouDirichlet}) gives
\[
\|DF_M\|_{L^\infty(B_R)}\le C_R,
\]
where $C_R$ depends on the oscillation on $B_{2R}$ but not on $M$.  The equation is therefore uniformly elliptic on $B_R$.  Interior Schauder estimates and bootstrapping yield, for every $k$,
\[
\|F_M\|_{C^{k,\alpha}(B_{R/2})}\le C_{R,k}.
\]
A diagonal subsequence therefore converges in $C^\infty_{\rm loc}(\R^8)$ to an entire solution $F_\eps$ of
\[
\cM_\eps[F_\eps]=0,
\]
with
\[F_\eps^-\le F_\eps\le F_\eps^+.\]
In particular
\begin{equation}\label{finalFestimate}
F_\eps=\olF+O\bigl(\eps(1+r^2)\bigr).
\end{equation}

Nothing global is hidden in the exhaustion: the global information is supplied by the ordered barriers.  Once the solution is sandwiched, the passage to the limit uses only standard local gradient and Schauder estimates on fixed balls.

\subsection{Returning to unit translating speed and non-convexity}\label{subsec:nonconvex}

If $F_\eps$ solves the speed-$\eps$ equation, define
\begin{equation}\label{unitscaling}
G_\eps(y):=\eps F_\eps(y/\eps).
\end{equation}
A direct scaling computation gives
\[
\div\left(\frac{\nabla G_\eps}{\sqrt{1+|\nabla G_\eps|^2}}\right)
=\frac1{\sqrt{1+|\nabla G_\eps|^2}},
\]
so $G_\eps$ is a unit-speed translator.

It remains to verify that $G_\eps$ is not convex.  The BDG graph $\olF$ is non-affine and satisfies
\[
\olF(u,v)=-\olF(v,u).
\]
If $\olF$ were convex, then composition with the orthogonal swap $(u,v)\mapsto(v,u)$ would show that $-\olF$ is convex.  Hence $\olF$ would be both convex and concave, and therefore affine, a contradiction.  Since $\olF$ is smooth, there exist a point $x_*$ and a unit vector $e_*$ such that
\[D^2\olF(x_*)[e_*,e_*]=-2\kappa<0.\]

On a fixed ball containing $x_*$, the sandwich first gives
\[
\|F_\eps-\olF\|_{C^0(B_{2|x_*|+3})}\le C\eps.
\]
Set $w_\eps:=F_\eps-\olF$.  Subtracting the minimal-surface equation for $\olF$ from the translator equation for $F_\eps$ and integrating the linearization along $\olF+s w_\eps$ gives
\[
\partial_i\big(a^{ij}_\eps(x)\partial_jw_\eps\big)
=\frac{\eps}{\sqrt{1+|\nabla F_\eps|^2}}
\]
on this fixed ball.  The preceding local gradient bound makes $a^{ij}_\eps$ uniformly elliptic on the fixed ball, with constants independent of $\eps$.  Standard quasilinear interior regularity first yields a uniform $C^{1,\alpha}$ bound for $F_\eps$ there; the coefficients $a^{ij}_\eps$ are then uniformly $C^{0,\alpha}$.  Applying the interior Schauder estimate to the displayed equation for $w_\eps$ gives
\[\|F_\eps-\olF\|_{C^{2,\alpha}(B_{2|x_*|+2})}\le C\eps.\]
Thus for $\eps$ small,
\[
D^2F_\eps(x_*)[e_*,e_*]\le-\kappa.
\]
Differentiating \eqref{unitscaling} twice gives
\[
D^2G_\eps(\eps x_*)
=\eps^{-1}D^2F_\eps(x_*),
\]
and hence
\[
D^2G_\eps(\eps x_*)[e_*,e_*]
\le-\frac\kappa\eps<0.
\]
Therefore $G_\eps$ is non-convex.

The parameter $\eps$ gives genuinely distinct unit-speed translators.  Along any fixed sector on which $F_0(y)\neq0$, the cubic BDG asymptotic and \eqref{finalFestimate} imply
\[
G_\eps(y)
=\eps F_\eps(y/\eps)
=\eps^{-2}F_0(y)+o(|y|^3)
\qquad (|y|\to\infty).
\]
Thus the leading cubic amplitude is $\eps^{-2}$; two different values of $\eps$ cannot produce the same graph up to a vertical translation.

Finally, if $N>8$, define
\[
\widetilde G_\eps(x_1,\dots,x_N)
=G_\eps(x_1,\dots,x_8).
\]
Adding flat variables leaves the translator equation unchanged and preserves the negative Hessian direction.  This proves the theorem in every dimension $N\ge8$.

This completes the proof of Theorem~\ref{mainthm}.

\appendix
\section{Weighted elliptic estimates on the BDG graph}\label{app:weighted}

This appendix records the elliptic estimates used in Sections~\ref{sec:exactbdg}--\ref{sec:globalbarriers}.  We include the details because the Jacobi operator is considered on a graph whose slope grows quadratically, so uniformity is most transparent after passing to intrinsic tangent-plane coordinates and rescaling by the local radial scale.

\subsection{Intrinsic charts and scale-invariant geometry}

Let $p\in\Gamma$ and write $R=1+r(p)$.  The curvature estimates established in Section~\ref{sec:exactbdg} imply
\[
 |A_\Gamma(p)|\le CR^{-1},\qquad
 |D_\Gamma^jA_\Gamma(p)|\le C_jR^{-1-j}.
\]
Choose geodesic normal coordinates on $T_p\Gamma$ in a ball of radius $2\vartheta R$, where $\vartheta>0$ is independent of $p$.  In these coordinates $\Gamma$ is a graph with height $q$ satisfying
\[
 |Dq|\le C\vartheta,\qquad
 R|D^2q|+R^2|D^3q|\le C.
\]
After the dilation $x=Rz$, the rescaled metric $g_R$ satisfies on $B_{2\vartheta}$
\[
 \lambda I\le g_R\le\Lambda I,
 \qquad
 \|g_R\|_{C^{1,\alpha}}+\|g_R^{-1}\|_{C^{1,\alpha}}\le C,
\]
with constants independent of $R$ and $p$.  Moreover the potential of the Jacobi operator rescales as
\[
 R^2|A_\Gamma|^2(Rz)=O(1)
\]
in $C^{0,\alpha}$.  Consequently the rescaled Jacobi operator is a uniformly elliptic operator with coefficients in a bounded subset of $C^{0,\alpha}$.

\subsection{Weighted H\"older norms}

For $\mu\in\mathbb R$ we use the norms introduced in Section~\ref{sec:weighted-fast-inverse}.  On the annulus $R\le r\le2R$, the weighted norm is equivalent, with constants independent of $R$, to the ordinary H\"older norm of the rescaled function.  More precisely, if
\[
 \widetilde h(z)=R^\mu h(Rz),
\]
then
\[
 \|\widetilde h\|_{C^{k,\alpha}(B_\vartheta)}
 \asymp
 \sum_{j=0}^k R^{\mu+j}\|D_\Gamma^jh\|_{C^0(B_{\vartheta R}(p))}
 +R^{\mu+k+\alpha}[D_\Gamma^kh]_{\alpha;B_{\vartheta R}(p)}.
\]
This equivalence is the reason for using intrinsic balls of radius comparable with $r$ in the definition of the weighted seminorm.

\begin{proposition}[Scale-invariant weighted Schauder estimate]\label{app-weighted-schauder}
Suppose $\cJ_\Gamma h=f$ and
\[
 \|f\|_{\alpha,\mu+2}+\|h\|_{0,\mu}<\infty.
\]
Then
\[
 \|h\|_{2,\alpha;\mu}
 \le C\bigl(\|f\|_{\alpha,\mu+2}+\|h\|_{0,\mu}\bigr).
\]
Equivalently, if the right-hand side is assigned weight $\mu$, one obtains the formulation in Lemma~\ref{weightedSchauder}.
\end{proposition}

\begin{proof}
Fix $p$ with $r(p)\sim R$ and use the rescaled chart above.  Set
\[
 \widetilde h(z)=R^\mu h(Rz),\qquad
 \widetilde f(z)=R^{\mu+2}f(Rz).
\]
The equation takes the form
\[
 a_R^{ij}\partial_{ij}\widetilde h+b_R^i\partial_i\widetilde h+c_R\widetilde h=\widetilde f
\]
on $B_{2\vartheta}$, where $a_R$ is uniformly elliptic and the coefficients are uniformly bounded in $C^{0,\alpha}$.  Interior Schauder estimates therefore give
\[
 \|\widetilde h\|_{C^{2,\alpha}(B_\vartheta)}
 \le C\left(
 \|\widetilde f\|_{C^{0,\alpha}(B_{2\vartheta})}
 +\|\widetilde h\|_{L^\infty(B_{2\vartheta})}
 \right).
\]
The constant is independent of $p$ and $R$.  Scaling back yields the desired estimate on $R\le r\le2R$.  A dyadic covering of the end, together with the ordinary interior estimate on a fixed compact set, proves the global estimate.
\end{proof}

\subsection{Maximum principle and the positive Jacobi field}

Recall that the vertical component of the unit normal
\[
 Z=\langle\nu,e_9\rangle=(1+|\nabla\overline F|^2)^{-1/2}
\]
is positive and satisfies $\cJ_\Gamma Z=0$.  If $\Omega\subset\Gamma$ is bounded and $u$ satisfies
\[
 \cJ_\Gamma u\ge0\quad\text{in }\Omega,
 \qquad u\le0\quad\text{on }\partial\Omega,
\]
then $u\le0$ in $\Omega$.  Indeed, writing $u=Zv$ and evaluating at a positive interior maximum of $v$ eliminates the zeroth-order obstruction because $\cJ_\Gamma Z=0$.  This is the maximum principle used in the exhaustion arguments below.

\subsection{The fast inverse}

Let $0<\nu<1$ and let $b_\nu$ be the positive global barrier constructed in Section~\ref{sec:weighted-fast-inverse}, so that
\[
 \cJ_\Gamma b_\nu\le-c(1+r)^{-4-\nu},
 \qquad
 b_\nu\asymp(1+r)^{-2-\nu}.
\]

\begin{proposition}[Fast inverse with full weighted control]\label{app-fast-inverse}
If $f$ is $O(4)\times O(4)$-invariant and $\|f\|_{\alpha,4+\nu}<\infty$, there is a unique solution of
\[
 \cJ_\Gamma h=f,
 \qquad
 \|h\|_{2,\alpha;2+\nu}<\infty,
\]
and
\[
 \|h\|_{2,\alpha;2+\nu}\le C\|f\|_{\alpha,4+\nu}.
\]
\end{proposition}

\begin{proof}
For $R\gg1$ solve the Dirichlet problem
\[
 \cJ_\Gamma h_R=f\quad\text{in }\Gamma\cap\{r<R\},
 \qquad h_R=0\quad\text{on }\{r=R\}.
\]
The bounded-domain problem is solvable for every $R$.  Indeed, the positive Jacobi field gives the maximum principle and hence triviality of the Dirichlet kernel; the Fredholm alternative for the uniformly elliptic Dirichlet operator then gives existence.  Put $M=\|f\|_{0,4+\nu}$.  After multiplying $b_\nu$ by a fixed constant,
\[
 \cJ_\Gamma(CMb_\nu)\le-|f|.
\]
Applying the maximum principle to $h_R-CMb_\nu$ and $-h_R-CMb_\nu$ gives
\[
 |h_R(y)|\le CM(1+r(y))^{-2-\nu},
\]
uniformly in $R$.  Proposition~\ref{app-weighted-schauder} upgrades this to a uniform $C^{2,\alpha}$ weighted estimate.  A diagonal subsequence converges on compact subsets to a global solution $h$ satisfying the stated estimate.

For uniqueness, let $\cJ_\Gamma h=0$ and $h=O(r^{-2-\nu})$.  Given $\delta>0$, choose $R_\delta$ so large that $|h|\le\delta b_\nu$ on $\{r=R_\delta\}$.  The maximum principle on the exterior region, followed by the compact maximum principle, gives $|h|\le\delta b_\nu$ globally.  Letting $\delta\downarrow0$ yields $h\equiv0$.
\end{proof}

\section{Taylor expansion of the graphical translator operator}\label{app:taylor}

For completeness we give the graph-variable estimate used in the residual analysis.  Write
\[
 \mathcal M_\varepsilon(F)
 =\operatorname{div}\mathfrak a(\nabla F)-\varepsilon\mathfrak s(\nabla F),
\]
where
\[
 \mathfrak a(p)=\frac{p}{\sqrt{1+|p|^2}},
 \qquad
 \mathfrak s(p)=\frac1{\sqrt{1+|p|^2}}.
\]
Let $W=(1+|\nabla F|^2)^{1/2}$ and suppose
\[
 |\nabla\delta|\le c_0W
\]
with $c_0$ sufficiently small.  Then for $0\le\tau\le1$,
\[
 W_\tau:=\bigl(1+|\nabla F+\tau\nabla\delta|^2\bigr)^{1/2}\asymp W.
\]

The derivatives of the two coefficient maps satisfy, along the entire segment,
\[
 |D^2\mathfrak a(\nabla F+\tau\nabla\delta)|\le CW^{-2},
 \qquad
 |D^2\mathfrak s(\nabla F+\tau\nabla\delta)|\le CW^{-3}.
\]
Taylor's formula gives
\[
 \mathfrak a(\nabla F+\nabla\delta)
 -\mathfrak a(\nabla F)-D\mathfrak a(\nabla F)\nabla\delta
 =\int_0^1(1-\tau)D^2\mathfrak a(\nabla F+\tau\nabla\delta)
 [\nabla\delta,\nabla\delta]d\tau.
\]
Taking one divergence differentiates either $\nabla\delta$ or the coefficient.  In the first case one obtains
\[
 CW^{-2}|D\delta||D^2\delta|.
\]
In the second case differentiating the coefficient introduces $D^2F$ and one more factor $W^{-1}$, giving
\[
 CW^{-3}|D^2F|\,|D\delta|^2.
\]
The source term has no spatial divergence, and its second derivative contributes
\[
 C\varepsilon W^{-3}|D\delta|^2.
\]
Consequently
\[
|\mathcal M_\varepsilon(F+\delta)-\mathcal M_\varepsilon(F)-D\mathcal M_\varepsilon(F)\delta|
\le C\Bigl(
W^{-2}|D\delta||D^2\delta|+W^{-3}|D^2F|\,|D\delta|^2
+\varepsilon W^{-3}|D\delta|^2
\Bigr).
\]
On a scaled intrinsic ball the same computation can be applied to difference quotients.  The coefficients have uniform H\"older control because $W_\tau\asymp W$, and the product estimate in $C^{0,\alpha}$ gives the H\"older version used in Section~\ref{sec:globalapprox}.  Notice in particular that absolute smallness of $|D\delta|$ is not required.  For the leading perturbation $\delta=\varepsilon U_1$ one has $|D\delta|=O(\varepsilon r)$ but $W\asymp r^2$, hence $|D\delta|/W=O(\varepsilon/r)$ on the end.

\section{Slow Jacobi modes on the homogeneous BDG graph}\label{app:slow-jacobi}

This appendix gives the model calculation behind the slow corrections used in
Section~\ref{sec:slowmatch}.  Related formulas appear in the Jacobi analysis of
the BDG graph in \cite{DPKW,DPPW}.  We include the derivation because the
translator construction uses not only existence of a solution but also the
precise radial degree, parity, behavior at Simons' cone, and intrinsic derivative
bounds.  These features are easiest to see on the homogeneous model and are then
transferred to the exact BDG graph by the comparison estimates of
Section~\ref{sec:exactbdg}.

\subsection{Vertical and normal linearizations}

Let
\[
 F_0(r,\theta)=r^3g(\theta),\qquad
 \Gamma_0=\{(x,F_0(x)):x\in\mathbb R^8\}.
\]
For a vertical variation $F_0+\tau\phi$, the first variation of graph mean
curvature is
\begin{equation}\label{app-Hprime}
 H'(F_0)[\phi]
 =\operatorname{div}\left(
 \frac{\nabla\phi}{W_0}
 -\frac{(\nabla F_0\cdot\nabla\phi)\nabla F_0}{W_0^3}
 \right),
 \qquad W_0=(1+|\nabla F_0|^2)^{1/2}.
\end{equation}
If $h$ is the corresponding normal displacement, then $\phi=W_0h$ and
\begin{equation}\label{app-vertical-normal}
 \mathcal J_{\Gamma_0}h=H'(F_0)[W_0h].
\end{equation}
Thus one may compute the leading Jacobi behavior in vertical variables and
convert to normal variables only after the radial and angular structure is
understood.  This simple observation is important near the cone: a vertical
angular derivative can be singular even though the intrinsic derivative of the
normal displacement is small.

For invariant functions $\phi=\phi(r,\theta)$, the leading homogeneous part of
\eqref{app-Hprime} will be denoted by $L_0$.  Since $F_0$ has degree three and
$W_0\sim r^2(9g^2+(g')^2)^{1/2}$, homogeneity alone shows that
\[
 L_0(r^\beta q(\theta))=r^{\beta-4}\mathcal A_\beta q(\theta).
\]
We now derive the angular operator explicitly.

\subsection{Separation of variables and divergence form}

Set
\[
 w_0(\theta)=\frac{\sin^3(2\theta)}{(9g^2+(g')^2)^{3/2}}.
\]
A direct substitution into the first variation gives
\begin{equation}\label{app-separation}
 \mathcal A_\beta q
 =\frac{9g^{(\beta+4)/3}}{\sin^3(2\theta)}
 \left[
 w_0g^{2/3}\bigl(g^{-\beta/3}q\bigr)'
 \right]'.
\end{equation}
The most useful feature of \eqref{app-separation} is visible before any
integration: for every $\beta$,
\[
 q_\beta(\theta)=g(\theta)^{\beta/3}
\]
annihilates $\mathcal A_\beta$.  This is the separated homogeneous Jacobi mode
of radial degree $\beta$.  It is the reason the inhomogeneous problem can be
integrated twice without solving a new second-order angular equation from
scratch.

Suppose
\begin{equation}\label{app-slow-source}
 L_0(r^\beta q)=r^{\beta-4}p(\theta).
\end{equation}
Integrating \eqref{app-separation} first from $\theta$ to the axis and then from
the cone to $\theta$ gives
\begin{align}\label{app-qformula}
q(\theta)=g(\theta)^{\beta/3}
\Bigg[ A
&-\frac19\int_{\pi/4}^{\theta}
 g(s)^{-2/3}\frac{(9g(s)^2+g'(s)^2)^{3/2}}{\sin^3(2s)}\\
&\qquad\times
 \left(\int_s^{\pi/2}p(\tau)g(\tau)^{-(\beta+4)/3}
 \sin^3(2\tau)\,d\tau\right)ds\Bigg].\notag
\end{align}
The constant $A$ multiplies the homogeneous mode.  It will be chosen according
to parity, regularity at the cone, or matching with another region.

Formula \eqref{app-qformula} is the basic slow-mode formula.  We next analyze
its three values $\beta=2,1,0$ that occur in the translator expansion.

\subsection{The even $r^{-2}$ source: a degree-two correction}

Let $p_2$ be smooth and even with respect to reflection across the cone.  Setting
$\beta=2$ in \eqref{app-slow-source} yields
\[
 L_0(r^2a_1(\theta))=r^{-2}p_2(\theta).
\]
Near the cone write $x=\theta-\pi/4$.  The BDG angular profile is odd,
\[
 g(\theta)=g_1x+g_3x^3+O(x^5),\qquad g_1>0,
\]
while $p_2$ is even.  In \eqref{app-qformula} the apparent powers of $g$ cancel
so that the bounded branch of $a_1$ extends evenly through $x=0$.  In
particular
\[
 a_1(\theta)=a_{10}+a_{12}x^2+O(x^4).
\]
The constant $a_{10}$ is generally nonzero.  Therefore the first translator
correction does not preserve the zero set of $F_0$ on the Simons cone.  This is
the quantitative origin of the recentering introduced in
Section~\ref{sec:outer}: the cone is shifted by an amount of order
$\varepsilon r^2$ before the next correction is resolved.

A second consequence of \eqref{app-qformula} is the sign identity used in the
transversal nondegeneracy.  If $p_2>0$, differentiation of the bracket gives
\[
 g^{5/3}\left(g^{-2/3}a_1\right)'<0
\]
in the sector.  The exact normalization is unimportant; what matters is that the
first correction changes monotonically in the weighted angular variable.  This
prevents the recentered transversal coefficient from vanishing.

\subsection{The odd $r^{-3}$ source: the fractional cone mode}

Let $p_3$ be smooth and odd.  With $\beta=1$ we seek
\[
 \phi_2=rq_2(\theta),\qquad L_0\phi_2=r^{-3}p_3(\theta).
\]
Oddness implies $p_3(\theta)=p_{31}x+O(x^3)$.  Expanding the inner integral in
\eqref{app-qformula} and using $g\sim g_1x$ gives
\begin{equation}\label{app-q2cone}
 q_2(\theta)=A g(\theta)^{2/3}+O(g(\theta)^{5/3}).
\end{equation}
Thus
\begin{equation}\label{app-phi2cone}
 \phi_2(r,\theta)
 \sim A r g^{2/3}
 =A\frac{t^{2/3}}{r},\qquad t=r^3g(\theta).
\end{equation}
The exponent $2/3$ is forced by the angular operator; it is not introduced by a
choice of ansatz.  This is why the same fractional power reappears later in the
inner parabolic problem.

To understand the geometry of \eqref{app-phi2cone}, convert to normal
variables.  Since $W_0\sim r^2$ on the end,
\[
 h_2=\frac{\phi_2}{W_0}
 \sim r^{-1}g^{2/3}
 \sim \frac{t^{2/3}}{r^3}.
\]
The coordinate derivative $\partial_\theta h_2$ has an apparent factor
$g^{-1/3}$, but the intrinsic metric degenerates in precisely the compensating
way.  More explicitly, for invariant $h$ one has
\begin{equation}\label{app-intrinsic-gradient}
 |D_{\Gamma_0}h|^2
 =\frac{|h_r|^2}{1+9g^2r^4}
 +\frac{|h_\theta|^2}{r^2+r^6(g')^2}
 +\mathcal E_{r\theta}(h),
\end{equation}
where the mixed term satisfies the same bound as the geometric mean of the two
displayed terms.  For $h_2=r^{-1}q_2(\theta)$,
\[
 h_{2,r}=O(r^{-2}g^{2/3}),\qquad
 h_{2,\theta}=O(r^{-1}g^{-1/3}).
\]
In the outer region $r^2g\gtrsim1$ and $g'\asymp1$ near the cone, hence
\[
 \frac{|h_{2,\theta}|^2}{r^2+r^6(g')^2}
 \le Cr^{-8}g^{-2/3}
 \le Cr^{-20/3}.
\]
Away from the cone the coefficients are smooth and one simply obtains
$O(r^{-4})$ or better.  If the outer mode is cut off where $r^2g\sim1$, the
cutoff derivatives have the same or smaller order.  Consequently there exists
$\mu_*>0$ such that
\begin{equation}\label{app-odd-gradient-gain}
 |D_{\Gamma_0}h_2|^2\le Cr^{-4-\mu_*}
\end{equation}
in the region in which the outer expression is used.

Estimate \eqref{app-odd-gradient-gain} is one of the points for which it is
important to distinguish coordinate singularities from geometric singularities.
In the elliptic problems based on the BDG graph the odd symmetry allows one to
stop at the cone.  For translators, however, the equation contains the positive
even drift term and the cone becomes an interior transition set.  The quantity
$\partial_t h_2\sim t^{-1/3}$ then enters the equation directly, which is why the
inner layer of Section~\ref{sec:innerlayer} is necessary even though the
intrinsic gradient remains controlled.

\subsection{The even $r^{-4}$ source and the logarithmic resonance}

The borderline case corresponds to $\beta=0$:
\[
 L_0 q_4(\theta)=r^{-4}p_4(\theta).
\]
If $p_4$ vanishes at the cone to the order dictated by evenness, the integral
formula produces a bounded angular function whose first non-smooth term behaves
like $g^{1/3}$.  Such a term can be cut off at the inner scale, because after the
cutoff its Jacobi error decays strictly faster than $r^{-4}$.

If instead $p_4(\pi/4)\neq0$, there is a genuine degree-zero resonance.  The
radial function $\log r$ is the generalized homogeneous mode associated with
$\beta=0$.  Direct substitution gives
\begin{equation}\label{app-log-resonance}
 L_0(\log r)=r^{-4}\{c_*+b_*(\theta)\},
\end{equation}
where $c_*\neq0$ and $b_*$ is smooth, even, and vanishes at the cone.  Choose a
multiple $\lambda\log r$ so that the constant part of
$p_4-\lambda(c_*+b_*)$ vanishes at the cone.  The remainder can then be solved by
\eqref{app-qformula}.  Thus the general slow even correction at this order has
the form
\begin{equation}\label{app-even-slow-form}
 \phi_4(r,\theta)=\lambda\log r+q_4(\theta)+\text{faster terms}.
\end{equation}
This is the origin of the logarithmic term in Section~\ref{sec:slowmatch}.

\subsection{Cutting off the model modes and passing to the exact graph}

The separated solutions above solve the homogeneous model problem only on the
end.  To obtain global smooth functions, fix a cutoff $\chi$ equal to one for
$r\ge2R_0$ and zero for $r\le R_0$.  For the cone-singular angular terms one
uses in addition the adapted cutoff in
\[
 \zeta=r^2g(\theta).
\]
The reason for this choice is scale invariant: $\zeta\sim1$ is precisely where
the outer cone expansion ceases to be uniform.  On the support of
$D\chi(\zeta)$ one has $g\sim r^{-2}$, and each angular derivative of the cutoff
is compensated by the small size of the fractional angular profile.  The
resulting commutators therefore gain a positive power of $r$ over the original
slow source.  This is the same mechanism used explicitly in the matching
construction in Sections~\ref{sec:innerlayer} and~\ref{sec:slowmatch}.

Finally, let $\Gamma$ be the exact BDG graph and let $\pi:\Gamma\to\Gamma_0$ be
the normal correspondence on the end.  The refined asymptotics established in
Section~\ref{sec:exactbdg} imply that for corresponding functions
$h=h_0\circ\pi$,
\begin{equation}\label{app-exact-model-comparison}
 \mathcal J_\Gamma h
 =\left[\mathcal J_{\Gamma_0}h_0
 +O(r^{-2-\sigma_B})D^2_{\Gamma_0}h_0
 +O(r^{-3-\sigma_B})D_{\Gamma_0}h_0
 +O(r^{-4-\sigma_B})h_0\right]\circ\pi
\end{equation}
for some $\sigma_B>0$.  Therefore each model slow mode yields an approximate
solution on $\Gamma$ whose error belongs to the fast weighted class.  Applying
Proposition~\ref{app-fast-inverse} removes this error.  This proves, within the
present paper, the precise principle used repeatedly in the main construction:
extract the slow homogeneous mode explicitly, cut it off at its natural scale,
transfer it to the exact graph, and solve only the genuinely fast remainder by
weighted inversion.

The calculation also clarifies the division of labor between the two pieces of
Jacobi theory.  The fast inverse controls errors once the radial leading term has
been removed; the separated formulas determine that leading term and its cone
behavior.  Neither statement alone is sufficient for the translator problem.

\section*{Acknowledgements}

The research of M. del Pino is supported by the Royal Society Research Professorship grant RP-R1-180114 and by the ERC/UKRI Horizon Europe grant ASYMEVOL, EP/Z000394/1. The research of J. Wei is partially supported by GRF of RGC of Hong Kong entitled ``On critical and supercritical Fujita equation''.

The authors used ChatGPT (OpenAI) strictly as an auxiliary tool for language editing and presentation of the manuscript. All mathematical ideas, arguments and conclusions are those of the authors and are their sole responsibility. This paper was first announced in \cite{DPW-Milan} in 2011.

\end{document}